\documentclass{amsart}
\usepackage{amssymb}
\usepackage[mathcal]{euscript}
\usepackage[cmtip,all]{xy}
\usepackage{amsmath}\usepackage{amsfonts}
\usepackage{graphicx} 
\usepackage{verbatim} 
\usepackage{tikz}

\newcommand{\bi}{\textnormal{\textbf{i}}} \newcommand{\alb}{\mathbb{A}} \newcommand{\cB}{\mathcal{B}} \newcommand{\cC}{\mathcal{C}} \newcommand{\FF}{\mathbb{F}} \newcommand{\ZZ}{\mathbb{Z}} \newcommand{\NN}{\mathbb{N}} 
\newcommand{\cT}{\mathcal{T}} \newcommand{\cO}{\mathcal{O}} \newcommand{\cV}{\mathcal{V}} \newcommand{\cW}{\mathcal{W}} \newcommand{\cH}{\mathcal{H}} \newcommand{\med}{\;|\;} \newcommand{\CC}{\mathbb{C}}
\DeclareMathOperator{\supp}{\mathrm{Supp}\,} \DeclareMathOperator{\Aut}{\mathrm{Aut}} \DeclareMathOperator{\Diag}{\mathrm{Diag}} \DeclareMathOperator{\Stab}{\mathrm{Stab}}  \DeclareMathOperator{\Inn}{\mathrm{Inn}} 
\DeclareMathOperator{\End}{\mathrm{End}} \DeclareMathOperator{\Sym}{\mathrm{Sym}} \DeclareMathOperator{\Str}{\mathrm{Str}} 
\DeclareMathOperator{\AAut}{\mathbf{Aut}} \DeclareMathOperator{\mmu}{\boldsymbol{\mu}} \DeclareMathOperator{\SStr}{\mathbf{Str}}
\DeclareMathOperator{\chr}{\mathrm{char}\,} \DeclareMathOperator{\TKK}{\mathrm{TKK}} \DeclareMathOperator{\lspan}{\mathrm{span}} 

\newcommand{\Ort}{\mathrm{O}} \newcommand{\Spin}{\mathrm{Spin}} \newcommand{\GL}{\mathrm{GL}} \newcommand{\SL}{\mathrm{SL}}
\newcommand{\RT}{\mathrm{RT}} \newcommand{\Univ}{\mathcal{U}}
\DeclareMathOperator{\Hom}{\mathrm{Hom}\,} \DeclareMathOperator{\id}{\mathrm{id}} 
\DeclareMathOperator{\Der}{\mathrm{Der}\,}  \DeclareMathOperator{\Innder}{\mathrm{Innder}\,}  \DeclareMathOperator{\Rad}{\mathrm{Rad}\,}
\DeclareMathOperator{\Grad}{\mathrm{Grad}} \DeclareMathOperator{\TKKGrad}{\mathrm{TKKGrad}}
\DeclareMathOperator{\rk}{\textnormal{rk}}
\DeclareMathOperator{\im}{\textnormal{im}} 
\newcommand{\Def}{\textnormal{Def}}
\newcommand{\SSpin}{\mathbf{Spin}}
\newcommand{\Lie}{\mathrm{Lie}}
\newcommand{\frso}{{\mathfrak{so}}}
\newcommand{\GG}{\mathbf{G}}
\newcommand{\cJ}{\mathcal{J}}
\newcommand{\subo}{_{\bar 0}}
\newcommand{\subuno}{_{\bar 1}}

\newtheorem{theorem}{Theorem}
\newtheorem{proposition}[theorem]{Proposition}
\newtheorem{lemma}[theorem]{Lemma}
\newtheorem{corollary}[theorem]{Corollary}

\theoremstyle{definition}
\newtheorem{df}[theorem]{Definition}
\newtheorem{example}[theorem]{Example}
\newtheorem{notation}[theorem]{Notation}

\theoremstyle{remark}
\newtheorem{remark}[theorem]{Remark}

\numberwithin{equation}{section} 
\numberwithin{theorem}{section} 

\begin{document}

\title[Fine gradings on simple exceptional Jordan pairs and triple systems]{Fine gradings on simple exceptional \\ Jordan pairs and triple systems}

\author[Diego Aranda-Orna]{Diego Aranda-Orna${}^\star$}
\address{Departamento de Matem\'{a}ticas
 e Instituto Universitario de Matem\'aticas y Aplicaciones,
 Universidad de Zaragoza, 50009 Zaragoza, Spain}
\email{daranda@unizar.es}

\thanks{${}^\star$Supported by the Spanish Ministerio de Econom\'{\i}a y Competitividad---Fondo Europeo de Desarrollo Regional (FEDER) MTM2013-45588-C3-2-P and by the Diputaci\'on General de Arag\'on---Fondo Social Europeo (Grupo de Investigaci\'on de \'Algebra)}

\date{}

\begin{abstract}
We give a classification up to equivalence of the fine group gradings by abelian groups on the Jordan pairs and triple systems of types bi-Cayley and Albert, under the assumption that the base field is algebraically closed of characteristic different from $2$. The associated Weyl groups are computed. We also determine, for each fine grading on the bi-Cayley and Albert pairs, the induced grading on the exceptional simple Lie algebra given by the Tits-Kantor-Koecher construction.
\end{abstract}

\maketitle


\section{Introduction and preliminaries}

In the classification of simple Jordan pairs (\cite{L75}) there are four infinite families and two exceptional cases. The importance of Jordan pairs is due to the Tits-Kantor-Koecher (TKK) construction, which allows to construct a 3-graded Lie algebra from a Jordan pair, which is simple if and only if so is the Jordan pair. It is also known that a 3-graded Lie algebra determines a Jordan pair. The exceptional simple Jordan pairs, namely, the types bi-Cayley and Albert, allow to construct Lie algebras of types $E_6$ and $E_7$ via the TKK construction.

The main goal of this paper is to study and classify fine gradings on Jordan pairs and triple systems of types bi-Cayley and Albert. By grading we usually mean grading by an abelian group, and we always assume that the base field $\FF$ is algebraically closed of characteristic different from $2$, unless otherwise stated. 

This paper is organized as follows.

In this section, we first give the basic definitions relevant to Jordan pairs and triple systems, and then recall the classification of orbits under the action of the automorphism group of finite-dimensional simple Jordan pairs. Finally, we recall the classification of fine group gradings on the Cayley algebra and on the Albert algebra, which will be used later to construct gradings on the Jordan systems under consideration. A nice reference to study the classification of fine gradings on Cayley and Albert algebras is \cite{EKmon}.

In the second section, we introduce the basic definitions concerning gradings on Jordan pairs and triple systems, and prove some general results about such gradings. In particular, we study how the gradings on a Jordan pair induce gradings on the associated Lie algebra given by the TKK construction. We prove, using Peirce decompositions, that fine gradings on finite-dimensional semisimple Jordan pairs have $1$-dimensional homogeneous components. We also prove that the trace form behaves well with respect to the gradings on Jordan pairs and triple systems.

In the third section we recall the definition of the bi-Cayley and Albert pairs and triple systems, and introduce some important automorphisms. For the bi-Cayley systems, we describe the automorphism groups and their orbits. A construction of fine gradings on the bi-Cayley and Albert pairs and triple systems is given. 

The main goal of the paper is reached in the fourth section, where we classify, up to equivalence, the fine gradings by abelian groups on the bi-Cayley and Albert pairs and triple systems (over an algebraically closed field of characteristic not $2$).

In the fifth section, we study how the fine gradings on the bi-Cayley and Albert pairs induce gradings on the Lie algebra obtained by means of the TKK construction.

Finally, in the sixth section we compute the Weyl groups of all the fine gradings on the bi-Cayley and Albert pairs and triple systems.

\subsection{Jordan pairs and triple systems}
We will now recall from \cite{L75} some basic definitions about Jordan pairs and Jordan triple systems.

As already mentioned, we will assume throughout the paper that the ground field $\FF$ is algebraically closed of characteristic different from $2$, unless indicated otherwise. The superscript $\sigma$ will always take the values $+$ and $-$, and will be omitted when there is no ambiguity.

Let $\cV^+$ and $\cV^-$ be vector spaces over $\FF$, and let $Q^\sigma \colon \cV^\sigma \rightarrow \Hom(\cV^{-\sigma},\cV^\sigma)$ be quadratic maps. Define trilinear maps, $\cV^\sigma \times \cV^{-\sigma} \times \cV^\sigma \rightarrow \cV^\sigma$, $(x,y,z)\mapsto \{x,y,z\}^\sigma$, and bilinear maps, $D^\sigma \colon \cV^\sigma \times \cV^{-\sigma} \rightarrow \End(\cV^\sigma)$, by the formulas
\begin{equation} \{x,y,z\}^\sigma = D^\sigma(x,y)z := Q^\sigma(x,z)y \end{equation}
where $Q^\sigma(x,z) = Q^\sigma(x+z) - Q^\sigma(x) - Q^\sigma(z)$. Note that $\{x,y,z\}=\{z,y,x\}$ and $\{x,y,x\}=2Q(x)y$. 

We will write $x^\sigma$ to emphasize that $x$ is an element of $\cV^\sigma$. Alternatively, we may write $(x,y)\in\cV$ to mean $x\in\cV^+$ and $y\in\cV^-$. The map $Q^\sigma(x)$ is also denoted by $Q^\sigma_x$.

\smallskip

A {\em (quadratic) Jordan pair} is a pair $\cV=(\cV^+,\cV^-)$ of vector spaces and a pair $(Q^+, Q^-)$ of quadratic maps $Q^\sigma \colon \cV^\sigma \rightarrow \Hom(\cV^{-\sigma},\cV^\sigma)$ such that the following identities hold in all scalar extensions:
\begin{itemize}
\item[(QJP1)] $D^\sigma(x,y) Q^\sigma(x) = Q^\sigma(x) D^{-\sigma}(y,x)$,
\item[(QJP2)] $D^\sigma(Q^\sigma(x)y, y) = D^\sigma(x, Q^{-\sigma}(y)x)$,
\item[(QJP3)] $Q^\sigma(Q^\sigma(x)y) = Q^\sigma(x)Q^{-\sigma}(y)Q^\sigma(x)$.
\end{itemize}

A {\em (linear) Jordan pair} is a pair $\cV=(\cV^+,\cV^-)$ of vector spaces with trilinear products $\cV^\sigma \times \cV^{-\sigma} \times \cV^\sigma \to \cV^\sigma $, $(x,y,z) \mapsto \{x,y,z\}^\sigma$, satisfying the following identities:
\begin{itemize}
\item[(LJP1)] $\{x,y,z\}^\sigma = \{z,y,x\}^\sigma$,
\item[(LJP2)] $[D^\sigma(x,y), D^\sigma(u,v)] = D^\sigma(D^\sigma(x,y)u,v) - D^\sigma(u,D^{-\sigma}(y,x)v)$,
\end{itemize}
where $D^\sigma(x,y)z=\{x,y,z\}^\sigma$.

Note that, under the assumption $\chr\FF\neq2$, the definitions of quadratic and linear Jordan pairs are equivalent.

\smallskip

A pair $\cW=(\cW^+,\cW^-)$ of subspaces of a Jordan pair $\cV$ is called a {\em subpair} (respectively an {\em ideal}) if $Q^\sigma(\cW^\sigma)\cW^{-\sigma} \subseteq \cW^\sigma$ (respectively $Q^\sigma(\cW^\sigma)\cV^{-\sigma} + Q^\sigma(\cV^\sigma)\cW^{-\sigma} + \{\cV^\sigma, \cV^{-\sigma}, \cW^\sigma\} \subseteq \cW^\sigma $). We say that $\cV$ is {\em simple} if its ideals are only the trivial ones and the maps $Q^\sigma$ are nonzero.

A {\em homomorphism} $h\colon \cV \rightarrow \cW$ of Jordan pairs is a pair $h=(h^+,h^-)$ of $\FF$-linear maps $h^\sigma \colon \cV^\sigma \rightarrow \cW^\sigma$ such that $h^\sigma(Q^\sigma(x)y) = Q^\sigma(h^\sigma(x))h^{-\sigma}(y)$ for all $x\in\cV^\sigma$, $y\in\cV^{-\sigma}$. By linearization, this implies $h^\sigma(\{x,y,z\}) = \{h^\sigma(x),h^{-\sigma}(y),h^\sigma(z)\}$ for all $x\in\cV^\sigma$, $y\in\cV^{-\sigma}$. {\em Isomorphisms} and {\em automorphisms} are defined in the obvious way. The ideals are precisely the kernels of homomorphisms.

A {\em derivation} is a pair $\Delta=(\Delta^+,\Delta^-) \in \End(\cV^+)\times\End(\cV^-)$ such that $\Delta^\sigma(Q^\sigma(x)y) = \{\Delta^\sigma(x),y,x\} + Q^\sigma(x)\Delta^{-\sigma}(y)$ for all $x\in\cV^\sigma$, $y\in\cV^{-\sigma}$. For $(x,y)\in\cV$, the pair $\nu(x,y):=(D(x,y),-D(y,x)) \in \mathfrak{gl}(\cV^+)\oplus\mathfrak{gl}(\cV^-)$ is a derivation, which is usually called the {\em inner derivation} defined by $(x,y)$. It is well-known that $\Innder(\cV) := \lspan\{\nu(x,y) \med (x,y)\in\cV \}$ is an ideal of $\Der(\cV)$.

\smallskip

A {\em (quadratic) Jordan triple system} is a vector space $\cT$ with a quadratic map $P \colon \cT \rightarrow \End(\cT)$ such that the following identities hold in all scalar extensions:
\begin{itemize}
\item[(QJT1)] $L(x,y) P(x) = P(x) L(y,x)$,
\item[(QJT2)] $L(P(x)y, y) = L(x, P(y)x)$,
\item[(QJT3)] $P(P(x)y) = P(x)P(y)P(x)$,
\end{itemize}
where $L(x,y)z = P(x,z)y$ and $P(x,z) = P(x+z)-P(x)-P(z)$.

A {\em (linear) Jordan triple system} is a vector space $\cT$ with a trilinear product $\cT\times\cT\times\cT\to\cT$, $(x,y,z) \mapsto \{x,y,z\}$, satisfying the following identities:
\begin{itemize}
\item[(LJT1)] $\{x,y,z\} = \{z,y,x\}$,
\item[(LJT2)] $[D(x,y), D(u,v)] = D(D(x,y)u,v) - D(u,D(y,x)v)$,
\end{itemize}
where $D(x,y)z=\{x,y,z\}$.

Note that, under the assumption $\chr\FF\neq2$, the definitions of quadratic and linear Jordan triple systems are equivalent. In a quadratic Jordan triple system, the triple product is given by $\{x,y,z\}=L(x,y)z$.

\smallskip

A {\em homomorphism} of Jordan triple systems is an $\FF$-linear map $f\colon \cT \rightarrow \cT'$ such that $f(P(x)y) = P(f(x))f(y)$ for all $x,y\in\cT$. The rest of basic concepts, including isomorphisms and automorphisms, are defined in the obvious way. Recall that a linear Jordan algebra $J$ has an associated Jordan triple system $\cT$ with quadratic product $P(x) = U_x := 2L^2_x - L_{x^2}$, and similarly, a Jordan triple system $\cT$ has an associated Jordan pair $\cV = (\cT, \cT)$ with quadratic products $Q^\sigma = P$.

\subsection{Peirce decompositions and orbits}

We will now recall some well-known definitions related to Jordan pairs (for more details, see \cite{L75}, \cite{L91a}, \cite{L91b}, \cite{ALM05}).
 
An element $x\in\cV^\sigma$ is called \emph{invertible} if $Q^\sigma(x)$ is invertible, and in this case, $x^{-1}:=Q^\sigma(x)^{-1}x$ is said to be the \emph{inverse} of $x$. The set of invertible elements of $\cV^\sigma$ is denoted by $(\cV^\sigma)^\times$. A Jordan pair $\cV$ is called \emph{division pair} if $\cV\neq0$ and every nonzero element is invertible. The pair $\cV$ is said to be \emph{local} if the noninvertible elements of $\cV$ form a proper ideal, say $\mathcal{N}$; in this case, $\cV/\mathcal{N}$ is a division pair.

For a fixed $y\in\cV^{-\sigma}$, the vector space $\cV^\sigma$ with the operators
\begin{equation} x^2=x^{(2,y)}:=Q^\sigma(x) y, \qquad U_x=U_x^{(y)}:=Q^\sigma(x)Q^{-\sigma}(y), \end{equation}
becomes a Jordan algebra, which is denoted by $\cV^\sigma_y$. An element $(x,y)\in\cV$ is called \emph{quasi-invertible} if $x$ is quasi-invertible in the Jordan algebra $\cV^\sigma_y$, i.e., if $1-x$ is invertible in the unital Jordan algebra $\FF1+\cV^\sigma_y$ obtained from $\cV^\sigma_y$ by adjoining a unit element. In that case, $(1-x)^{-1}=1+z$ for some $z\in\cV^\sigma$, and $x^y:=z$ is called the \emph{quasi-inverse} of $(x,y)$. An element $x\in\cV^\sigma$ is called \emph{properly quasi-invertible} if $(x,y)$ is quasi-invertible for all $y\in\cV^{-\sigma}$. The \emph{(Jacobson) radical} of $\cV$ is $\Rad \cV := (\Rad\cV^+,\Rad\cV^-)$, where $\Rad\cV^\sigma$ is the set of properly quasi-invertible elements of $\cV^\sigma$. Note that $\Rad\cV$ is an ideal of $\cV$. We say that $\cV$ is \emph{semisimple} if $\Rad \cV=0$, and $\cV$ is \emph{quasi-invertible} or \emph{radical} if $\cV=\Rad \cV$. Of course, finite-dimensional simple Jordan pairs are semisimple, and finite-dimensional semisimple Jordan pairs are a direct sum of simple Jordan pairs.

An element $x\in\cV^\sigma$ is called \emph{von Neumann regular} (or vNr, for short) if there exists $y\in\cV^{-\sigma}$ such that $Q(x) y=x$. A Jordan pair is called \emph{von Neumann regular} if $\cV^+$ and $\cV^-$ consist of vNr elements. A pair $e=(x,y)\in\cV$ is called \emph{idempotent} if $Q(x) y=x$ and $Q(y) x=y$. Recall from \cite[Lemma 5.2]{L75} that if $x\in\cV^+$ is vNr and $Q(x) y=x$, then $(x, Q(y)x)$ is an idempotent; therefore, every vNr element can be completed to an idempotent. An element $x\in\cV^\sigma$ is called \emph{trivial} if $Q(x)=0$. A Jordan pair $\cV$ is called \emph{nondegenerate} if it contains no nonzero trivial elements. 

Given a Jordan pair $\cV$, a subspace $\mathcal{I}\subseteq \cV^\sigma$ is called an {\em inner ideal} if $Q^\sigma(\mathcal{I})(\cV^{-\sigma})\subseteq \mathcal{I}$. Given an element $x\in\cV^\sigma$, the {\em principal inner ideal} generated by $x$ is defined by $[x]:=Q(x)\cV^{-\sigma}$. The {\em inner ideal} generated by $x\in\cV^\sigma$ is defined by $(x):=\FF x+[x]$.

\begin{theorem}[{\cite[Th. 10.17]{L75}}] \label{semisimple}
The following conditions on a Jordan pair $\cV$ with dcc on principal inner ideals are equivalent: 
\begin{itemize}
\item[i)] $\cV$ is von Neumann regular;
\item[ii)] $\cV$ is semisimple;
\item[iii)] $\cV$ is nondegenerate.
\end{itemize}
\end{theorem}

\smallskip

For any $x\in\cV^\sigma$ and $y\in\cV^{-\sigma}$, the Bergmann operator is defined by
\begin{equation*}
B(x,y)=\id_{\cV^\sigma}-D(x,y)+Q(x)Q(y).
\end{equation*} 
In case $(x,y)\in\cV$ is quasi-invertible, the map
\begin{equation*}
\beta(x,y):=(B(x,y),B(y,x)^{-1})
\end{equation*}
is an automorphism, called the \textit{inner automorphism} defined by $(x,y)$. The \textit{inner automorphism group}, $\Inn(\cV)$, is the group generated by the inner automorphisms. 

Recall (\cite[Th. 5.4]{L75}) that given an idempotent $e=(e^+,e^-)$ of $\cV$, the linear operators
\begin{equation}
E_2^\sigma=Q(e^\sigma)Q(e^{-\sigma}), \quad E_1^\sigma=D(e^\sigma,e^{-\sigma})-2E^\sigma_2 , \quad E_0^\sigma=B(e^\sigma,e^{-\sigma}),
\end{equation}
are orthogonal idempotents of $\End(\cV^\sigma)$ whose sum is the identity, and we have the so-called \emph{Peirce decomposition}: $\cV^\sigma=\cV^\sigma_2 \oplus \cV^\sigma_1 \oplus \cV^\sigma_0$, where $\cV^\sigma_i=\cV^\sigma_i(e):=E_i^\sigma(\cV^\sigma)$. Moreover, this decomposition satisfies 
\begin{equation} \label{peirceproperties}
Q(\cV^\sigma_i)\cV^{-\sigma}_j\subseteq\cV^\sigma_{2i-j}, \qquad \{\cV^\sigma_i,\cV^{-\sigma}_j,\cV^\sigma_k\}\subseteq\cV^\sigma_{i-j+k},
\end{equation}
for any $i,j,k\in\{0,1,2\}$ (note that we use the convention $\cV_i^\sigma=0$ for $i\notin\{0,1,2\}$). In particular, $\cV_i=(\cV^+_i,\cV^-_i)$ is a subpair of $\cV$ for $i=0,1,2$.

We recall a few more definitions related to idempotents. Two nonzero idempotents $e$ and $f$ are called \emph{orthogonal} if $f\in\cV_0(e)$; this is actually a symmetric relation. An \emph{orthogonal system of idempotents} is an ordered set of pairwise orthogonal idempotents; it is usually denoted by $(e_1,\dots,e_r)$ in case it is finite, and there is an associated Peirce decomposition (but we will not use this more general version). An orthogonal system of idempotents is called \emph{maximal} if it is not properly contained in a larger orthogonal system of idempotents. It is known that a finite sum of pairwise orthogonal idempotents is again an idempotent. A nonzero idempotent $e$ is called \emph{primitive} if it cannot be written as the sum of two nonzero orthogonal idempotents. We say that $e$ is a \emph{local idempotent} (respectively a \emph{division idempotent}) if $\cV_2(e)$ is a local pair (respectively a division pair). In general, division idempotents are local, and local idempotents are primitive. If $\cV$ is semisimple, then the local idempotents are exactly the division idempotents. A \emph{frame} is a maximal set among orthogonal systems of local idempotents. Two frames of a simple finite-dimensional Jordan pair have always the same number of idempotents; that number of idempotents is called the {\em rank} of $\cV$ (see \cite[Def.~15.18]{L75}), and we have:

\begin{theorem}[{\cite[Th. 17.1]{L75}}] \label{orbitframes}
Let $\cV$ be a simple finite-dimensional Jordan pair over an algebraically closed field $\FF$. Let $(c_1,\dots,c_r)$ and $(e_1,\dots,e_r)$ be frames of $\cV$. Then there exists an inner automorphism $g$ of $\cV$ such that $g(c_i)=e_i$ for $i=1,\dots,r$.
\end{theorem}

Let $\cV$ be a semisimple Jordan pair and $x\in\cV^\sigma$. The \emph{rank} of $x$, $\rk(x)$, is defined as the supremum of the lenghts of all finite chains $[x_0] \subseteq [x_1] \subseteq \cdots \subseteq [x_n]$ of principal inner ideals $[x_i]=Q(x_i)\cV^{-\sigma}$ where each $x_i$ belongs to the inner ideal $(x)=\FF x+[x]$ generated by $x$, and the length of the chain is the number of strict inclusions (for more details, see \cite{L91a}). Hence, given a chain of length $n=\rk(x)$, we have $x_0=0$ and $[x_n]=[x]$. Two elements $x,z\in\cV^\sigma$ are called \emph{orthogonal} ($x\perp z$) if they are part of orthogonal idempotents, i.e., $x=e^\sigma$ and $z=c^\sigma$ for some orthogonal idempotents $e$ and $c$. For any $x,z\in\cV^\sigma$, $\rk(x+z)\leq\rk(x)+\rk(z)$; and in case that $x$ and $z$ have finite rank, the equality holds if and only if $x\perp z$ (\cite[Th. 3]{L91a}). Recall from \cite{L91b} that the \emph{capacity} of a Jordan pair $\cV$, $\kappa(\cV)$, is the infimum of the cardinalities of all finite sets of orthogonal division idempotents whose Peirce-$0$-space is zero (if there are no such idempotents, the capacity is $+\infty$). 

Recall that if $e=(x,y)$ is an idempotent, then $\rk(x)=\rk(y)$ (\cite[Cor. 1 of Th. 3]{L91a}), and this common value will be called the {\em rank} of $e$. In general, if $\rk(x)<\infty$, then $\rk(x)=\kappa(\cV_2(e))$ (\cite[Proposition~3]{L91a}); hence, $x$ has rank $1$ if and only if $\cV_2(e)$ is a division pair (i.e., the division idempotents are exactly the rank $1$ idempotents), and since $\FF$ is algebraically closed, this is equivalent to the condition $\im Q_x = \FF x$ (see \cite[Lemma 15.5]{L75}).

An element $x\in\cV^\sigma$ is called \emph{diagonalizable} if there exist orthogonal division idempotents $d_1,\dots,d_t$ such that $x=d_1^\sigma+\dots+d_t^\sigma$, and it is called \emph{defective} if $Q_y x=0$ for all rank one elements $y\in\cV^{-\sigma}$. The only element which is both diagonalizable and defective is $0$. If $\cV$ is simple, every element is either diagonalizable or defective (\cite[Cor. 1]{L91a}). The \textit{defect} of $\cV$ is $\Def(\cV) := (\Def(\cV^+), \Def(\cV^-))$, where $\Def(\cV^\sigma)$ denotes the set of defective elements of $\cV^\sigma$. For the definition of the generic trace of $\cV$, which is a bilinear map $\cV^+\times\cV^- \to\FF$ usually denoted by $m_1$ or $t$, see \cite[Def.~16.2]{L75}.

\begin{lemma}[{\cite[1.2.b]{ALM05}}] \label{defect}
Let $\cV$ be a semisimple finite-dimensional Jordan pair over an algebraically closed field $\FF$.
The defect is the kernel of the generic trace $t$ in the sense that
\begin{align*}
& x\in\Def(\cV^+) \Leftrightarrow t(x,\cV^-)=0, \\
& y\in\Def(\cV^-) \Leftrightarrow t(\cV^+,y)=0.
\end{align*} 
\end{lemma}

In this paper we only consider the case with $\chr\FF\neq2$, and in this case the defect of a semisimple Jordan pair is always zero (see \cite[Theorem~2]{L91a}).

\begin{proposition}[{\cite[1.9.(a)]{ALM05}}] \label{orbits}
Let $\cV$ be a simple finite-dimensional Jordan pair of rank $r$ over an algebraically closed field and such that $\Def(\cV)=0$, and let $\sigma\in\{\pm\}$. Then the automorphism group $\Aut\cV$ and the inner automorphism group $\Inn\cV$ have the same orbits on $\cV^\sigma$, and these orbits are described as follows: two elements $x,y \in \cV^\sigma$ belong to the same orbit if and only if $\rk(x)=\rk(y)$. Hence there are $r+1$ orbits, corresponding to the possible values $0,\dots,r$ of the rank function.
\end{proposition}

\begin{proposition}[{\cite[1.10.(a)]{ALM05}}] \label{invertibleorbits}
Let $\cV$ be a simple finite-dimensional Jordan pair containing invertible elements over an algebraically closed field and satisfying $\Def(\cV)=0$. Then $\Aut\cV$ acts transitively on $(\cV^\sigma)^\times$.
\end{proposition}

\begin{remark} \label{remarkorbitidempotents}
Given a finite-dimensional semisimple Jordan pair $\cV$, each idempotent $e$ of rank $r$ decomposes as a sum of $r$ orthogonal idempotents of rank $1$ (see \cite[Cor.~2 of Th.~1]{L91a}). By \eqref{peirceproperties}, we also have $\cV_2^\sigma(e) = \im Q_{e^\sigma}$, so the rank of $e$ in $\cV$ coincides with the rank of $e$ in $\cV_2(e)$.

Furthermore, if $\cV$ is simple, all sets consisting of $n$ orthogonal idempotents $e_1,\dots, e_n$ of fixed ranks $r_1,\dots, r_n$, respectively, are in the same orbit under the automorphism group. Indeed, first note that the Peirce subpaces $\cV_2(e_i)$ are semisimple Jordan pairs (because the vNr property is inherited by these subpairs and by Theorem~\ref{semisimple}); hence the idempotent $e_i$ decomposes as sum of $r_i$ orthogonal idempotents $e_{i, 1},\dots,e_{i, r_i}$ of rank $1$ in the corresponding Peirce subspace $\cV_2(e_i)$, and it suffices to apply Theorem~\ref{orbitframes}. In particular, idempotents of rank $r$ are in the same orbit.
\end{remark}

\subsection{Gradings on Cayley and Albert algebras} \label{sectionCayleyAlbert}

We assume that the reader is familiar with {\em Hurwitz algebras}, i.e., unital composition algebras (for a basic introduction see \cite[Chap. 2]{ZSSS}), and also with the exceptional Jordan algebra, usually called the {\em Albert algebra} (see \cite{J68}). A classification of gradings on Hurwitz algebras was given in \cite{Eld98}. On the other hand, a classification of gradings on the Albert algebra over an algebraically closed field of characteristic different from $2$ was obtainded in \cite{EK12a}. The reader may consult \cite[Chapters 4 and 5]{EKmon}.

\smallskip

Recall that a {\em composition algebra} is an algebra $C$ (not necessarily associative) with a quadratic form $n\colon C \to \FF$, called the {\em norm}, which is nondegenerate and multiplicative. The $8$-dimensional Hurwitz algebras are called {\em Cayley algebras} or {\em octonion algebras}. Since our field $\FF$ is algebraically closed, there is only one Cayley algebra up to isomorphism, which will be denoted by $\cC$. The {\em standard involution} of a Hurwitz algebra is given by $x \mapsto \bar x := n(x,1)1 - x$, where $n(x,y)=n(x+y)-n(x)-n(y)$.

 Recall that, as $\chr\FF\neq2$, there are two fine gradings up to equivalence on $\cC$, which are a $\ZZ^2$-grading, also called {\em Cartan grading}, and a $\ZZ_2^3$-grading. (In case that $\chr\FF=2$, the Cartan grading is the only fine grading up to equivalence on $\cC$.) There is a homogeneous basis associated to the Cartan grading (often called {\em good basis} or {\em canonical basis} by other autors), which will be referred to as {\em Cartan basis}, and the product for this basis is given by the next table:
\begin{center} \label{cartanproduct}
\begin{tabular}{c||cc|ccc|ccc|}
 & $e_1$ & $e_2$ & $u_1$ & $u_2$ & $u_3$ & $v_1$ & $v_2$ & $v_3$ \\
 \hline\hline
 $e_1$ & $e_1$ & 0 & $u_1$ & $u_2$ & $u_3$ & 0 & 0 & 0 \\
 $e_2$ & 0 & $e_2$ & 0 & 0 & 0 & $v_1$ & $v_2$ & $v_3$ \\
 \hline
 $u_1$ & 0 & $u_1$ & 0 & $v_3$ & $-v_2$ & $-e_1$ & 0 & 0 \\
 $u_2$ & 0 & $u_2$ & $-v_3$ & 0 & $v_1$ & 0 & $-e_1$ & 0 \\
 $u_3$ & 0 & $u_3$ & $v_2$ & $-v_1$ & 0 & 0 & 0 & $-e_1$ \\
 \hline
 $v_1$ & $v_1$ & 0 & $-e_2$ & 0 & 0 & 0 & $u_3$ & $-u_2$ \\
 $v_2$ & $v_2$ & 0 & 0 & $-e_2$ & 0 & $-u_3$ & 0 & $u_1$ \\
 $v_3$ & $v_3$ & 0 & 0 & 0 & $-e_2$ & $u_2$ & $-u_1$ & 0 \\
 \hline
\end{tabular}
\end{center}
The decomposition $\cC = \FF e_1 \oplus \FF e_2 \oplus U \oplus V$, with $U=\lspan\{u_1,u_2,u_3\}$ and $V=\lspan\{v_1,v_2,v_3\}$, is the {\em Peirce decomposition} associated to the idempotents $e_1$ and $e_2$. Note that the elements of the Cartan basis are isotropic for the norm, and paired as follows: $n(e_1,e_2) = 1 = n(u_i,v_i)$, $n(e_i, u_j) = n(e_i, v_j) = n(u_j, v_k) = 0$ for any $i=1,2$ and $j\neq k = 1,2,3$, and $n(u_i, u_j) = n(v_i, v_j) = 0$ for any $i,j=1,2,3$. The degree of the Cartan grading is defined by $\deg(e_1)=0=\deg(e_2)$, $\deg(u_1)=(1,0)=-\deg(v_1)$, $\deg(u_2)=(0,1)=-\deg(v_2)$ and $\deg(v_3)=(1,1)=-\deg(u_3)$.

On the other hand, any homogeneous orthonormal basis $\{x_i\}_{i=1}^8$ of $\cC$ associated to the $\ZZ_2^3$-grading (this grading only exists if $\chr\FF\neq2$ and can be obtained applying three times the Cayley-Dickson doubling process, see \cite{Eld98} for the construction), will be called a {\em Cayley-Dickson basis} of $\cC$, and we will usually assume that $x_1=1$. Note that, if $\{x_i\}_{i=1}^8$ and $\{y_i\}_{i=1}^8$ are Cayley-Dickson bases, then there exist some $\varphi \in \Aut\cC$, signs $s_i \in \{\pm1\}$ and permutation $\sigma$ of the indices such that $\varphi(x_i) = s_i y_{\sigma(i)}$.

Also, recall that $\cC$ with the new product $x*y := \bar x \bar y$ is called the {\em para-Cayley algebra}, which is sometimes denoted by $\bar \cC$.

\smallskip

The Albert algebra $\alb$ is defined as the hermitian $3\times3$-matrices over $\cC$, that is 
\begin{align*}\alb=\cH_3(\cC,{}^{-}) 
=& \left\{ \left( \begin{array}{ccc} \alpha_1 & \bar a_3 &  a_2 \\ a_3 & \alpha_2 & \bar a_1 \\ \bar a_2 & a_1 & \alpha_3 \end{array} \right) \med \alpha_i\in \FF, \; a_i\in\cC \right\} \\
=& \FF E_1 \oplus \FF E_2 \oplus \FF E_3 \oplus \iota_1(\cC) \oplus \iota_2(\cC) \oplus \iota_3(\cC),
\end{align*}
where \begin{align*}
E_1&=\left( \begin{array}{ccc} 1 &  0 &  0 \\ 0 & 0 & 0 \\ 0 & 0 & 0 \end{array} \right), &
E_2&=\left( \begin{array}{ccc} 0 &  0 &  0 \\ 0 & 1 & 0 \\ 0 & 0 & 0 \end{array} \right), &
E_3&=\left( \begin{array}{ccc} 0 &  0 &  0 \\ 0 & 0 & 0 \\ 0 & 0 & 1 \end{array} \right), \\
\iota_1(a)&= 2 \left( \begin{array}{ccc} 0 &  0 &  0 \\ 0 & 0 & \bar a \\ 0 & a & 0 \end{array} \right), &
\iota_2(a)&= 2 \left( \begin{array}{ccc} 0 &  0 &  a \\ 0 & 0 & 0 \\ \bar a & 0 & 0 \end{array} \right), &
\iota_3(a)&= 2 \left( \begin{array}{ccc} 0 &  \bar a &  0 \\ a & 0 & 0 \\ 0 & 0 & 0 \end{array} \right),
\end{align*}
for all $a\in\cC$, with multiplication $XY=\frac{1}{2}(X\cdot Y+Y\cdot X)$, where $X\cdot Y$ denotes the usual product of matrices. Then, $E_i$ are orthogonal idempotents with $\sum E_i=1$, and
\begin{equation} \begin{aligned}
& E_i\iota_i(a)=0, \quad E_{i+1}\iota_i(a)=\frac{1}{2}\iota_i(a)=E_{i+2}\iota_i(a), \\
& \iota_i(a)\iota_{i+1}(b)=\iota_{i+2}(\bar a\bar b), \quad \iota_i(a)\iota_i(b)=2n(a,b)(E_{i+1}+E_{i+2}),
\end{aligned} \end{equation}
where the subscripts are taken modulo 3.

Any element $X = \sum_{i=1}^3(\alpha_iE_i + \iota_i(a_i))$ of the Albert algebra satisfies the degree $3$ equation
\begin{equation} 
X^3 - T(X)X^2 + S(X)X - N(X)1 = 0, 
\end{equation}
where the linear form $T$ (called the \emph{trace}), the quadratic form $S$, and the cubic form $N$ (called the \emph{norm}) are given by:
\begin{equation} \begin{aligned}
& T(X) = \alpha_1 + \alpha_2 + \alpha_3, \\
& S(X) = \frac{1}{2}(T(X)^2 - T(X^2)) = \sum_{i=1}^3(\alpha_{i+1}\alpha_{i+2} -4 n(a_i)), \\ 
& N(X) = \alpha_1\alpha_2\alpha_3 + 8n(a_1, \bar a_2 \bar a_3) -4\sum_{i=1}^3 \alpha_i n(a_i).
\end{aligned} \end{equation}
The Albert algebra $\alb$ has a {\em Freudenthal adjoint} given by $x^\#:=x^2-T(x)x+S(x)1$, with linearization $x\times y:=(x+y)^\#-x^\#-y^\#$.

\smallskip

There are four fine gradings up to equivalence on $\alb$, with universal groups: $\ZZ^4$ (the Cartan grading), $\ZZ_2^5$, $\ZZ\times\ZZ_2^3$ and $\ZZ_3^3$ (the last one does not occur if $\chr\FF=3$). We recall the construction of these gradings now.

\medskip

Let $B_1 = \{e_i, u_j, v_j \med i=1,2, \; j=1,2,3 \}$ be a Cartan basis of $\cC$. We will call the basis $\{E_i, \iota_i(x) \med x\in B_1, \, i=1,2,3 \}$ a {\em Cartan basis} of $\alb$. The $\ZZ^4$-grading on $\alb$ is defined using this basis as follows:
\begin{equation}\begin{aligned}
& \deg E_i = 0, \quad \deg \iota_i(e_1)=a_i=-\deg\iota_i(e_2), \\
& \deg \iota_i(u_i)=g_i=-\deg\iota_i(v_i), \\
& \deg \iota_i(u_{i+1}) = a_{i+2} + g_{i+1} = -\deg \iota_i(v_{i+1}), \\  
& \deg \iota_i(u_{i+2}) = -a_{i+1} + g_{i+2} = -\deg \iota_i(v_{i+2}),
\end{aligned}\end{equation}
for $i=1,2,3$ mod $3$, and where
\begin{equation*}\begin{aligned}
a_1 = (1,0,0,0), \quad a_2 = (0,1,0,0), \quad a_3 = (-1,-1,0,0), \\
g_1 = (0,0,1,0), \quad g_2 = (0,0,0,1), \quad g_3 = (0,0,-1,-1).
\end{aligned}\end{equation*}

\medskip

Let now $B_2 = \{x_i\}_{i=1}^8$ be a Cayley-Dickson basis of $\cC$ with degree map $\deg_\cC$. The $\ZZ_2^5$-grading on $\alb$ is constructed by imposing that the elements of the basis $\{E_i, \iota_i(x) \med x\in B_2, \; i=1,2,3 \}$ are homogeneous with:
\begin{equation}\begin{aligned}
& \deg E_i = 0, \quad \deg \iota_1(x) = (\bar 1, \bar 0, \deg_\cC x), \\
& \deg \iota_2(x) = (\bar 0, \bar 1, \deg_\cC x), \quad \deg \iota_3(x) = (\bar 1, \bar 1, \deg_\cC x).
\end{aligned}\end{equation}

\medskip

Take $\bi\in\FF$ with $\bi^2=-1$. The $\ZZ\times\ZZ_2^3$-grading on $\alb$ is constructed using the following elements of $\alb$:
\begin{equation} \label{notGradAlbertZxZ23} \begin{aligned}
& E=E_1, \quad \widetilde{E}=1-E=E_2+E_3, \quad S_\pm=E_3-E_2\pm\frac{\bi}{2}\iota_1(1), \\ 
& \nu(a)=\bi\iota_1(a),  \quad \nu_\pm(x)=\iota_2(x)\pm\bi\iota_3(\bar x),
\end{aligned}\end{equation}
for any $a\in\cC_0 = \{ y\in\cC \med \text{tr}(y)=0 \}$ and $x\in\cC$, where the product is:
\begin{equation}\begin{aligned}
& E\widetilde{E}=0, \quad ES_\pm=0, \quad E\nu(a)=0, \quad E\nu_\pm(x)=\frac{1}{2}\nu_\pm(x), \\
& \widetilde{E}S_\pm=S_\pm, \quad \widetilde{E}\nu(a)=\nu(a), \quad \widetilde{E}\nu_\pm(x)=\frac{1}{2}\nu_\pm(x), \\
& S_\pm S_\pm=0, \quad S_+S_-=2\widetilde{E}, \quad S_\pm\nu(a)=0, \\
& S_\pm\nu_\mp(x)=\nu_\pm(x), \quad S_\pm\nu_\pm(x)=0, \\
& \nu(a)\nu(b)=-2n(a,b)\widetilde{E}, \quad \nu(a)\nu_\pm(x)=\pm\nu_\pm(xa),\\
& \nu_\pm(x)\nu_\pm(y)=2n(x,y)S_\pm, \quad \nu_+(x)\nu_-(y)=2n(x,y)(2E+\widetilde{E})-\nu(\bar xy-\bar yx),
\end{aligned}\end{equation}
for any $x,y \in\cC$ and $a,b\in\cC_0$.

Fix a Cayley-Dickson basis $B_2$ of $\cC$, with degree map $\deg_\cC$. The degree map of the $\ZZ\times\ZZ_2^3$-grading is given by:
\begin{equation} \label{gradAlbertZxZ23} \begin{aligned}
& \deg S_\pm = (\pm 2, \bar0, \bar0, \bar0), \quad \deg \nu_\pm(x) = (\pm1, \deg_\cC x), \\
& \deg E = 0 = \deg \widetilde{E}, \quad \deg \nu(a) = (0, \deg_\cC a),
\end{aligned}\end{equation}
for homogeneous $x\in \cC$ and $a\in \cC_0$.

\medskip

Finally, we recall the construction of the $\ZZ_3^3$-grading on $\alb$. Let $\omega$ be a cubic root of $1$ in $\FF$. Consider the order $3$ automorphism $\tau$ of $\cC$ given by $e_i \mapsto e_i$ for $i=1,2$ and $u_j \mapsto u_{j+1}$, $v_j \mapsto v_{j+1}$ for $j=1,2,3$. Write $\tilde{\iota}_i(x) = \iota_i(\tau^i(x))$ for $i=1,2,3$, and $x\in\cC$. Then, the $\ZZ_3^3$-grading is determined by the homogeneous generators
\begin{equation*}
X_1 = \sum_{i=1}^3 \tilde{\iota}_i(e_1), \quad X_2 = \sum_{i=1}^3 \tilde{\iota}_i(u_1), \quad X_3 = \sum_{i=1}^3 \omega^{-i} E_i,
\end{equation*}
with degrees
\begin{equation} \label{Z33grading}
\deg X_1  = (\bar1,\bar0,\bar0), \quad \deg X_2  = (\bar0,\bar1,\bar0), \quad \deg X_3  = (\bar0,\bar0,\bar1).
\end{equation}


\section{Generalities about gradings}

\subsection{Gradings on Jordan pairs and triple systems}

Let $\cV=(\cV^+,\cV^-)$ be a Jordan pair and let $S$ be a set. Given two decompositions of vector spaces $\Gamma^\sigma: \cV^\sigma=\bigoplus_{s\in S}\cV^\sigma_s$, we will say that $\Gamma=(\Gamma^+,\Gamma^-)$ is an $S$-{\em grading} on $\cV$ if for any $s_1,s_2,s_3\in S$ there is $s\in S$ such that $\{\cV^\sigma_{s_1},\cV^{-\sigma}_{s_2},\cV^\sigma_{s_3}\}\subseteq \cV^\sigma_s$ for all $\sigma\in\{+,-\}$. In this case, we also say that $\Gamma$ is a {\em set grading} on $\cV$. The set $\supp\Gamma=\supp\Gamma^+\cup\supp\Gamma^-$ is called the {\em support} of the grading, where $\supp\Gamma^\sigma:=\{s\in S\med\cV^\sigma_s\neq0\}$. The vector space $\cV_s^+\oplus\cV_s^-$ is called the {\em homogeneous component of degree} $s$. If $0\neq x\in \cV^\sigma_s$, we say that $x$ is {\em homogeneous of degree} $s$, and we write $\deg(x)=s$. 

Let $\Gamma^\sigma: \cV^\sigma=\bigoplus_{s\in S}\cV^\sigma_s$ and $\widetilde{\Gamma}^\sigma: \cV^\sigma=\bigoplus_{t\in T}\cV^\sigma_t$ be two set gradings on a Jordan pair $\cV$. We say that $\Gamma$ is a {\em refinement} of $\widetilde{\Gamma}$, or that $\widetilde{\Gamma}$ is a {\em coarsening} of $\Gamma$, if for any $s\in S$ there is $t\in T$ such that $\cV^\sigma_s\subseteq \cV^\sigma_t$ for $\sigma\in\{+,-\}$. The refinement is said to be {\em proper} if some containment $\cV^\sigma_s\subseteq \cV^\sigma_t$ is strict. A set grading with no proper refinement is called a {\em fine} grading.

\medskip

Let $G$ be an abelian group. Given two decompositions $\Gamma^\sigma: \cV^\sigma=\bigoplus_{g \in G}\cV^\sigma_g$, we will say that $\Gamma=(\Gamma^+,\Gamma^-)$ is a $G$-{\em grading} on $\cV$ if $\{\cV^\sigma_g,\cV^{-\sigma}_h,\cV^\sigma_k\}\subseteq \cV^\sigma_{g+h+k}$ for any $g,h,k\in G$ and $\sigma\in\{+,-\}$. A set grading by a set $S$ on $\cV$ will be called {\em realizable as a group grading}, or a {\em group grading}, if $S$ is contained in some abelian group $G$ such that the subspaces $\cV^\sigma_g := \cV^\sigma_s$ for $g=s\in S$ and $\cV^\sigma_g := 0$ for $g\notin S$ define a $G$-grading. In this paper, by a {\em grading} we will mean a group grading. In particular, a grading is called {\em fine} if it has no proper refinements in the class of group gradings. We will not consider gradings by nonabelian groups.

Let $\Gamma$ be a set grading on $\cV$. The \emph{universal group} of $\Gamma$, which is denoted by $\Univ(\Gamma)$, is defined as the abelian group generated by $\supp\Gamma$ with the relations $s_1+s_2+s_3=s$ when $0\neq\{\cV^\sigma_{s_1},\cV^{-\sigma}_{s_2},\cV^\sigma_{s_3}\}\subseteq \cV^\sigma_s$ for some $\sigma\in\{+,-\}$. Note that this defines a group grading $\Gamma'$ by $\Univ(\Gamma)$, which is a coarsening of $\Gamma$, and it is clear that $\Gamma$ and $\Gamma'$ have the same homogeneous components if and only if $\Gamma$ is realizable as a group grading. 

Suppose that a group grading $\Gamma$ on $\cV$ admits a realization as a $G_0$-grading for some abelian group $G_0$. Then $G_0$ is isomorphic to the universal group of $\Gamma$ if and only if for any other realization of $\Gamma$ as a $G$-grading there is a unique homomorphism $G_0 \rightarrow G$ that restricts to the identity on $\supp\Gamma$.

Given a $G$-grading $\Gamma$ on $\cV$ and a group homomorphism $\alpha \colon G\rightarrow H$, we define the \emph{induced} $H$-grading ${}^\alpha\Gamma$ determined by setting $\cV_h^\sigma := \bigoplus_{g\in\alpha^{-1}(h)}\cV^\sigma_g$. Then ${}^\alpha\Gamma$ is a coarsening of $\Gamma$. In case $\Gamma$ is given by its universal group, i.e., $G = \Univ(\Gamma)$, then any coarsening of $\Gamma$ (in the class of group gradings) is of the form ${}^\alpha\Gamma$ for some homomorphism $\alpha \colon \Univ(\Gamma) \rightarrow H$.

\begin{example}
Consider the Jordan pair $\cV = (\FF,\FF)$ associated to the Jordan algebra $J=\FF$, i.e., with products $U_x(y) = x^2y$ for $x,y\in\FF$. Then, the trivial grading on $\cV$ has universal group $\ZZ_2$ and support $\{\bar1\}$. On the other hand, for a nonzero Jordan pair with zero product, the trivial grading has universal group $\ZZ$ and support $\{1\}$.
\end{example}

Let $\Gamma_1^\sigma: \cV^\sigma=\bigoplus_{s\in S}\cV^\sigma_s$ and $\Gamma_2^\sigma: \cW^\sigma=\bigoplus_{t\in T}\cW^\sigma_t$ be graded Jordan pairs. An isomorphism of Jordan pairs $\varphi=(\varphi^+,\varphi^-) \colon \cV\rightarrow \cW$ is said to be an {\em equivalence} of graded Jordan pairs if, for each $s\in S$, there is (a unique) $t\in T$ such that $\varphi^\sigma(\cV^\sigma_s)=\cW^\sigma_t$ for all $\sigma\in\{+,-\}$. In that case, $\Gamma_1$ and $\Gamma_2$ are said to be {\em equivalent}.

Given a $G$-grading $\Gamma$ on $\cV$, the \emph{automorphism group of $\Gamma$}, $\Aut(\Gamma)$, is the group of self-equivalences of $\Gamma$. The \emph{stabilizer of $\Gamma$}, $\Stab(\Gamma)$, is the group of $G$-automorphisms of $\Gamma$, i.e., the group of automorphisms of $\cV$ that fix the homogeneous components. The \emph{diagonal group of $\Gamma$}, $\Diag(\Gamma)$, is the subgroup of $\Stab(\Gamma)$ consisting of the automorphisms that act by multiplication by a nonzero scalar on each homogeneous component. The \emph{Weyl group of $\Gamma$} is the quotient group $\cW(\Gamma)=\Aut(\Gamma)/\Stab(\Gamma)$, which can be regarded as a subgroup of $\Sym(\supp\Gamma)$ and also of $\Aut(\Univ(\Gamma))$.

\begin{proposition} \label{finegradingproperty} 
Let $\Gamma$ be a fine grading on a Jordan pair $\cV$ and let $G$ be its universal group. Then, there is a group homomorphism $\pi \colon G \rightarrow \ZZ$ such that $\pi(g)=\sigma1$ if $\cV^\sigma_g\neq0$ for some $\sigma\in\{+,-\}$. In particular, $\supp\Gamma^+$ and $\supp\Gamma^-$ are disjoint.
\end{proposition}
\begin{proof}
Define a $G\times\ZZ$-grading $\Gamma'$ on $\cV$ by means of $\cV^\sigma_{(g,\sigma1)} := \cV^\sigma_g$ and $\cV^\sigma_{(g,-\sigma1)}:=0$. Then, $\Gamma'$ is a refinement of $\Gamma$. But $\Gamma$ is fine, so $\Gamma'$ and $\Gamma$ have the same homogeneous components. Since $G=\Univ(\Gamma)$, there is a (unique) homomorphism $\phi \colon G \rightarrow G\times\ZZ$ such that $\cV^\sigma_{\phi(g)} = \cV^\sigma_g$ for all $g\in G$ and $\sigma\in\{+,-\}$. Therefore, $\phi$ has the form $\phi(g)=(g,\pi(g))$ for some homomorphism  $\pi \colon G \rightarrow \ZZ$. By definition of $\Gamma'$, $\pi$ satisfies $\pi(g)=\sigma1$ if $\cV^\sigma_g\neq0$. In particular, $\supp\Gamma^+ = \pi^{-1}(1)$ and $\supp\Gamma^- = \pi^{-1}(-1)$ are disjoint.
\end{proof}

Let $\cT$ be a Jordan triple system and $S$ a set. Consider a decomposition $\Gamma:\cT=\bigoplus_{s\in S}\cT_s$. We call $\Gamma$ an {\em $S$-grading} if, for any $s_1,s_2,s_3\in S$, there is $s\in S$ such that $\{\cT_{s_1},\cT_{s_2},\cT_{s_3}\}\subseteq\cT_s$.

Let $G$ be an abelian group and consider a decomposition $\Gamma:\cT=\bigoplus_{g\in G}\cT_g$. We say that $\Gamma$ is a {\em $G$-grading} if $\{\cT_{g},\cT_{h},\cT_{k}\}\subseteq\cT_{g+h+k}$ for any $g, h, k\in G$. A set grading is said to be {\em realizable as a group grading}, or a {\em group grading}, if $S$ is contained in some abelian group $G$ such that the subspaces $\cT_g := \cT_s$ for $g=s\in S$ and $\cT_g := 0$ for $g\notin S$ define a $G$-grading.

The rest of definitions about gradings on Jordan triple systems are analogous to those given for graded Jordan pairs.

\begin{df} \label{homogeneousbilinearform}
Given a graded algebra $A$, a bilinear form $b\colon A \times A \to \FF$ will be called {\em homogeneous of degree $0$}, or simply \emph{homogeneous}, if we have $g+h=0$ whenever $b(A_g, A_h)\neq0$. (Analogous definition for a bilinear form on a graded Jordan triple system.) Similarly, given a graded Jordan pair $\cV$, a bilinear form $b\colon \cV^+\times\cV^- \rightarrow \FF$ will be called \emph{homogeneous} if we have $g+h=0$ whenever $b(\cV^+_g,\cV^-_h)\neq0$.
\end{df}

Let $J$ be a Jordan algebra. Consider its associated Jordan pair $\cV=(J,J)$ and Jordan triple system $\cT=J$. Then, any $G$-grading $\Gamma$ on $J$ is a $G$-grading on $\cT$. In the same way, any $G$-grading $\Gamma$ on $\cT$ (or on $J$) induces a $G$-grading on $\cV$, given by $(\Gamma,\Gamma)$.
We say that a $G$-grading $\widetilde{\Gamma}$ on $\cV$ is a $G$-grading on $J$ (respectively on $\cT$) when $\widetilde{\Gamma}$ equals $(\Gamma,\Gamma)$ for some $G$-grading $\Gamma$ on $J$ (respectively on $\cT$). If $\varphi=(\varphi^+,\varphi^-)\in\Aut(\cV)$, denote $\widehat{\varphi}:=(\varphi^-,\varphi^+)\in\Aut(\cV)$. Notice that $\widehat{\varphi_1\varphi_2}=\widehat{\varphi}_1\widehat{\varphi}_2$ and $\widehat{1}_\cV=1_\cV$, so $\; {}_{\widehat{}} \in\Aut(\Aut(\cV))$. Moreover, $\widehat{\widehat{\varphi}}=\varphi$ and $\Aut(\cT)=\{\varphi\in\Aut(\cV) \med \widehat{\varphi}=\varphi\}$. We can consider, with natural identifications, that $\Aut(J)\leq\Aut(\cT)\leq\Aut(\cV)$. 

Let $\Gamma_J$ be a $G$-grading on a Jordan algebra $J$ and $\Gamma_\cT$ the same $G$-grading on the Jordan triple system $\cT=J$. Since $\Aut(J)\leq\Aut(\cT)$, we have $\Aut(\Gamma_J)\leq\Aut(\Gamma_\cT)$ and $\Stab(\Gamma_J)\leq\Stab(\Gamma_\cT)$. Thus, 
\begin{align*}
\cW(\Gamma_J) & = \Aut(\Gamma_J) / \Stab(\Gamma_J) = \Aut(\Gamma_J) / (\Stab(\Gamma_\cT) \cap \Aut(\Gamma_J)) \\
& \cong (\Aut(\Gamma_J) \cdot \Stab(\Gamma_\cT)) / \Stab(\Gamma_\cT) \leq \Aut(\Gamma_\cT) / \Stab(\Gamma_\cT) = \cW(\Gamma_\cT).
\end{align*}
In the same manner, if $\Gamma_\cT$ is $G$-grading on a Jordan triple system $\cT$ and $\Gamma_\cV = (\Gamma_\cT, \Gamma_\cT)$ is the induced $G$-grading on the associated Jordan pair $\cV$, we have natural identifications: $\Aut(\Gamma_\cT)\leq\Aut(\Gamma_\cV)$, $\Stab(\Gamma_\cT)\leq\Stab(\Gamma_\cV)$ and $\cW(\Gamma_\cT)\leq\cW(\Gamma_\cV)$.

\medskip

Let $\Gamma$ be a $G$-grading on a Jordan pair $\cV$ with degree $\deg$. Fix $g\in G$. For any homogeneous elements $x^+\in \cV^+$ and $y^-\in \cV^-$, set $\deg_g(x^+):=\deg(x^+)+g$, $\deg_g(y^-):=\deg(y^-)-g$. This defines a new $G$-grading, which will be denoted by $\Gamma^{[g]}$ and called the {\em $g$-shift of $\Gamma$}. Note that, although $\Gamma$ and $\Gamma^{[g]}$ may fail to be equivalent (because the shift may collapse or split a homogeneous subspace of $\cV^+$ with another of $\cV^-$), the intersection of their homogeneous components with $\cV^\sigma$ coincide for each $\sigma$. It is clear that $(\Gamma^{[g]})^{[h]}=\Gamma^{[g+h]}$. Similarly, if $\Gamma$ is a $G$-grading on a Jordan triple system $\cT$ and $g\in G$ has order $2$, we can define the {\em $g$-shift} $\Gamma^{[g]}$ with the new degree $\deg_g(x) := \deg(x) + g$.

\bigskip

A nice introduction to affine group schemes, including the relation between gradings on algebras and their automorphism group schemes, can be found in \cite[Section 1.4 and Appendix A]{EKmon}; note that these results also hold for Jordan pairs, and we will use them in this section without mentioning.

\begin{df}
If $A$ is an algebra and $R$ is an associative commutative unital $\FF$-algebra, we will denote the $R$-algebra $A \otimes R$ by $A_R$. Denote by $\AAut(A)$ the automorphism group scheme of $A$, so that $\AAut(A)(R) = \Aut_R(A_R)$. Recall that $G$-gradings on $A$ correspond to morphisms $G^D \to \AAut(A)$. The morphism corresponding to a grading $\Gamma$ will be denoted by $\eta_\Gamma$. For any $R$-point $f\in G^D(R)$, $f$ is a group homomorphism $G\rightarrow R^\times$, and $\eta_\Gamma(f)$ is defined by
\[
\eta_\Gamma(f)(x_g\otimes r)=x_g\otimes f(g)r,
\]
for any $g\in G$, $x_g\in A_g$ and $r\in R$. We will use similar notations for Jordan pairs and triple systems. 
\end{df}

Note that the homogeneous components of a grading $\Gamma$ are, in a way, the eigenspaces of the action of $G^D$ via $\eta_\Gamma$.

The following result is a generalization of \cite[Theorem~3.7(a), Eq.~(1)]{N85} to the case of affine group schemes and $\chr\FF\neq2$. 

\begin{theorem} \label{triplesystemautom}
Let $\FF$ be an arbitrary field of characteristic not $2$. Let $J$ be a finite-dimensional central simple Jordan $\FF$-algebra with associated Jordan triple system $\cT$. Then, there is an isomorphism of affine group schemes $\AAut(\cT) \simeq \AAut(J) \times \mmu_2$.
\end{theorem}
\begin{proof}
Recall that the product of $\cT$ is given by $\{x,y,z\} := x(yz) + z(xy) - (xz)y$. Denote by $\cT^-$ the Lie triple system associated to $J$, that is, $\cT^- = J$ with the product $[x,y,z]:=\{x,y,z\}-\{y,x,z\}$. Then we have that $[x,y,z] = -2((xz)y-x(zy)) = -2(x,z,y)$. The center of $\cT^-$ is defined by $Z(\cT^-) := \{x\in J \med [x,J,J] = 0 \} = \{x\in J \med (x,J,J) = 0 \}$. From the identities $(x,y,z) = -(z,y,x)$ and $(x,y,z) + (y,z,x) + (z,x,y) = 0$ we obtain that
$Z(\cT^-) = \{x\in J \med (x,J,J) = (J,x,J) = (J,J,x) = 0 \} = Z(J) = \FF1$. In consequence, for each associative commutative unital $\FF$-algebra $R$ we have $Z((\cT^-)_R) = R1$. Note that for each $\varphi\in\Aut_R(\cT_R)$ we also have $\varphi\in\Aut_R((\cT^-)_R)$, and hence $\varphi(R1) = R1$. In particular, $\varphi(1) = r1$ for some $r\in R$. Since $\varphi$ is bijective, there is some $s\in R$ such that $1 = \varphi(s1) = s\varphi(1) = sr1$, which shows that $r\in R^\times$. On the other hand, we have $r1 = \varphi(1) = \varphi(\{1,1,1\}) = \{\varphi(1),\varphi(1),\varphi(1)\} = r^3 1$ with $r\in R^\times$, which implies that $r^2 = 1$, that is $r\in \mmu_2(R)$. 

Recall that the automorphisms of a Jordan algebra $J$ are exactly the automorphisms of the associated Jordan triple system $\cT_J$ that fix the unit $1$ of $J$. Indeed, if $f\in\Aut(\cT_J)$ with $f(1)=1$, then, since $\{x,1,z\} = xz$ for all $x,z\in J$, we have $f(xz) = f(\{x,1,z\}) = \{f(x), f(1), f(z)\} = \{f(x), 1, f(z)\} = f(x)f(z)$, hence $f\in\Aut(J)$.

Note that the map $\delta_r \colon x \mapsto rx$ is an order $2$ automorphism of $\cT_R$ and $\delta_r\varphi(1) = 1$, so that $\delta_r\varphi\in \Aut_R(J_R)$. Hence $\varphi = \delta_r \psi = \psi \delta_r$ with $\psi \in \Aut_R(J_R)$. We conclude that $\Aut_R(\cT_R) \cong \Aut_R(J_R) \times \mmu_2(R)$ for each $R$.
\end{proof}

\begin{corollary} \label{triplegrads} 
Let $J$ be a finite-dimensional central simple Jordan $\FF$-algebra with associated Jordan triple system $\cT$. Then, the map that sends a $G$-grading on $J$ to the same $G$-grading on $\cT$ gives a bijective correspondence from the equivalence classes of gradings on $J$ to the equivalence classes of gradings on $\cT$.
\end{corollary}
\begin{proof}
Let $\Gamma$ be a $G$-grading on $\cT$ and $\eta_\Gamma\colon G^D \to\AAut(\cT)$ its associated morphism. Consider the projection morphism $\pi\colon \AAut(J)\times\mmu_2 \to \AAut(J)$ and the isomorphism $f\colon\AAut(\cT) \to \AAut(J)\times\mmu_2$ of Theorem~\ref{triplesystemautom}. Also, note that the elements of $\mmu_2(R)$ are identified with the scalar automorphisms of $\cT_R$ of the form $r1$ with $r\in R^\times$ and $r^2 = 1$, which implies that the composition $\pi \circ f \circ \eta_\Gamma \colon G^D \to\AAut(J)$ determines the equivalence class of $\Gamma$. Then the result follows because the morphisms $G^D \to\AAut(J)$ are in correspondence with the equivalence classes of gradings on $J$.
\end{proof}

\begin{remark} 
Note that fine gradings on $\cT$ correspond to maximal quasitori of $\AAut(\cT)$, which are the direct product of a maximal quasitorus of $\AAut(J)$ and $\mmu_2$.
\end{remark}

\begin{corollary} \label{tripleWeyl} 
Let $J$ be a finite-dimensional central simple Jordan $\FF$-algebra with associated Jordan triple system $\cT$. Let $\Gamma_J$ be a $G$-grading on $J$ and $\Gamma_\cT$ the same $G$-grading on $\cT$. Then $\cW(\Gamma_\cT)=\cW(\Gamma_J)$.
\end{corollary}
\begin{proof}
From Theorem~\ref{triplesystemautom} we know that $\Aut(\cT) \cong \Aut(J) \times \{\pm1\}$. Hence $\Aut(\Gamma_\cT) \cong \Aut(\Gamma_J) \times \{\pm1\}$ and the result follows.
\end{proof}

\begin{proposition} \label{unitytriples}
Let $J$ be a  Jordan $\FF$-algebra with unity $1$, and let $\cT$ be its associated Jordan triple system. Let $\Gamma$ be a $G$-grading on $\cT$. 
If $J$ is central simple, then $1$ is homogeneous. Moreover, if $G = \Univ(\Gamma)$ and $1$ is homogeneous, then $\deg(1)$ has order $2$.
\end{proposition}
\begin{proof}
We know that $1$ is invariant under $\AAut(J)$. Hence, if $J$ is central simple, $\FF1$ is invariant under $\AAut(\cT) = \AAut(J)\times\mmu_2$, and also under $G^D$ for any $G$-grading (where $G^D$ acts via the morphism $\eta_\Gamma\colon G^D \to \AAut(\cT)$ producing the grading). In consequence, $1$ is homogeneous.

Suppose now that $G = \Univ(\Gamma)$ and that $1$ is homogeneous. Note that the trivial grading on $\cT$ has universal group $\ZZ_2$ and support $\{\bar1\}$. Since $G = \Univ(\Gamma)$, the trivial grading is induced from $\Gamma$ by some epimorphism $\varphi \colon \Univ(\Gamma) \to \ZZ_2$. Since $\varphi$ sends all elements of the support to $\bar1$, $\deg(1)$ has at least order $2$. On the other hand, $U_1(1) = 1$ implies that $2\deg(1) = 0$, and we can conclude that $\deg(1)$ has order $2$.
\end{proof}

\begin{remark}
Given a unital Jordan algebra $J$ with associated Jordan triple system $\cT$ and a grading $\Gamma$ on $\cT$, it is not true in general that $1$ is homogeneous. For example, take $J = \cT = \FF \times \FF$ and consider the $\ZZ_2^2$-grading on $\cT$ given by $\cT_{(\bar1,\bar0)} = \FF \times 0$ and $\cT_{(\bar0,\bar1)} = 0\times \FF$.
\end{remark}

\begin{proposition}\label{regulargradspairs} 
Let $J$ be a Jordan $\FF$-algebra with unity $1$, and $G$ an abelian group. Consider the associated Jordan pair $\cV=(J,J)$.
If $\Gamma$ is a set grading on $\cV$ such that $1^+$ (or $1^-$) is homogeneous, then the restriction of $\Gamma$ to $J=\cV^+$ induces a set grading on $J$. 
If $\Gamma$ is a $G$-grading on $\cV$ such that $1^+$ (or $1^-$) is homogeneous, then the restriction of the shift $\Gamma^{[g]}$, with $g=-\deg(1^+)$, to $J=\cV^+$ induces a $G$-grading $\Gamma_J$ on $J$. Moreover, if $G=\Univ(\Gamma)$ then the universal group $\Univ(\Gamma_J)$ is isomorphic to the subgroup of $\Univ(\Gamma)$ generated by  $\supp \Gamma^{[g]}$, and if in addition $\Gamma$ is fine we also have  that $\Univ(\Gamma)$ is isomorphic to $\Univ(\Gamma_J) \times \ZZ$.
\end{proposition}
\begin{proof} 
Let $\Gamma$ be a set grading on $\cV$ with $1^+$ homogeneous. Since $U_{1^+}(y^-)=y^+$ for any $y$, the homogeneous components of $\cV^+$ and $\cV^-$ coincide. But from $\{x,1,z\}=xz$ with $\chr\FF\neq2$, it follows that $\Gamma$ induces a set grading on $J=\cV^+$, where $J_s=\cV^+_s$.

Assume that $\Gamma$ is also a $G$-grading. From $U_{1^+}(1^-)=1^+$, we get $\deg(1^+)+\deg(1^-)=0$. Take $g:=-\deg(1^+)=\deg(1^-)$. The grading $\Gamma^{[g]}$ satisfies $\deg_g(1^+)=0=\deg_g(1^-)$. But since $U_{1^+}(x^-)=x^+$, we have $\deg_g(x^+)=\deg_g(x^-)$. From $\{x,1,z\}=xz$ we obtain $\deg_g(x)+\deg_g(z)=\deg_g(xz)$, so $\Gamma^{[g]}$ induces a $G$-grading on $J=\cV^+$. 

Suppose now that $G=\Univ(\Gamma)$ and set $H = \langle \supp \Gamma^{[g]} \rangle$. Note that $\Gamma^{[g]}$ can be regarded as a $\Univ(\Gamma)$-grading and also as an $H$-grading; and similarly $\Gamma_J$ can be regarded as a $\Univ(\Gamma_J)$-grading and as an $H$-grading. By the universal property of the universal group, the $H$-grading $\Gamma_J$ is induced from the $\Univ(\Gamma_J)$-grading $\Gamma_J$ by an homomorphism $\varphi_1 \colon \Univ(\Gamma_J) \to H$ that restricts to the identity in the support. On the other hand, the $\Univ(\Gamma_J)$-grading $\Gamma_J$ induces a $\Univ(\Gamma_J)$-grading $(\Gamma_J, \Gamma_J)$ on $\cV$ that is a coarsening of $\Gamma$, and therefore $(\Gamma_J,\Gamma_J)$ is induced from $\Gamma$ by some epimorphism $\varphi \colon \Univ(\Gamma) \to \Univ(\Gamma_J)$. Let $\varphi_2 \colon H \to \Univ(\Gamma_J)$ be the restriction of $\varphi$ to $H$. Note that $g\in\ker\varphi$, which implies that the $\Univ(\Gamma_J)$-grading $(\Gamma_J, \Gamma_J)$ is induced from the $H$-grading $\Gamma^{[g]}$ by $\varphi_2$, and also that $\varphi_2$ is an epimorphism which is the identity in the support. Since each epimorphism $\varphi_i$ is the identity in the support, both compositions $\varphi_1\varphi_2$ and $\varphi_2\varphi_1$ must be the identity and hence $\Univ(\Gamma_J) \cong H$.

Assume now that $\Gamma$ is fine and denote by $\Gamma_H$ the grading $\Gamma^{[g]}$ regarded as an $H$-grading. Note that $\Univ(\Gamma) = \langle \supp\Gamma^{[g]}, g \rangle = \langle \supp\Gamma_H, g \rangle = \langle H, g \rangle$. Consider $H$ as a subgroup of $H \times \ZZ \cong H \times \langle g_0 \rangle$, where the element $g_0$ has infinite order. The $H$-grading $\Gamma_H$ can be regarded as an $H\times \langle g_0 \rangle$-grading, and the shift $(\Gamma_H)^{[g_0]}$ defines another $H\times \langle g_0 \rangle$-grading where $\deg(1^+) = g_0$. Since the $H\times \langle g_0 \rangle$-grading $(\Gamma_H)^{[g_0]}$ is a coarsening of the $\Univ(\Gamma)$-grading $\Gamma$ (because $\Gamma$ is fine and by Proposition~\ref{finegradingproperty}), by the universal property there is an epimorphism $\Univ(\Gamma) = \langle H, g \rangle \to H\times \langle g_0 \rangle$ that sends $-g \mapsto g_0$ and fixes the elements of $H$. In consequence, $H \cap \langle g \rangle = 0$, $\langle g \rangle \cong \ZZ$, and we can conclude that $\Univ(\Gamma) = \langle H, g \rangle \cong H \times \ZZ \cong \Univ(\Gamma_J) \times \ZZ$.
\end{proof}

\begin{proposition}\label{regulargradstriples} 
Let $J$ be a Jordan $\FF$-algebra with unity $1$, and $G$ an abelian group. Consider the associated Jordan triple system $\cT$.
If $\Gamma$ is a set grading on $\cT$ such that $1$ is homogeneous, then $\Gamma$ induces a set grading on $J$. 
If $\Gamma$ is a $G$-grading on $\cT$ such that $1$ is homogeneous, then the shift $\Gamma^{[g]}$ with $g=\deg(1)$ induces a $G$-grading $\Gamma_J$ on $J$. Moreover, if $G=\Univ(\Gamma)$ then $\Univ(\Gamma_J)$ is isomorphic to the subgroup of $\Univ(\Gamma)$ generated by $\supp \Gamma^{[g]}$, and if in addition $\Gamma$ is fine we also have that  $\Univ(\Gamma)$ is isomorphic to $\Univ(\Gamma_J) \times \ZZ_2$.
\end{proposition}
\begin{proof} 
Let $\Gamma$ is a set grading on $\cT$ with $1$ homogeneous; since $\{x,1,z\}=xz$ with $\chr\FF\neq2$ it follows that $\Gamma$ induces a set grading on $J$. Assume from now on that $\Gamma$ is a $G$-grading on $\cT$ with $1$ homogeneous of degree $g$.
Proposition~\ref{unitytriples} shows that $g$ has order $1$ or $2$. Hence the shift $\Gamma^{[g]}$ defines a $G$-grading on $\cT$ with degree $\deg_g(x) = \deg(x)+g$, where $\deg_g(1) = 0$. Set $H = \langle \supp\Gamma^{[g]}\rangle$. The rest of the proof follows using the same arguments of the proof of Proposition~\ref{regulargradspairs}, but using $\cT$ instead of the Jordan pair $\cV=(J,J)$.
\end{proof}

\subsection{Gradings induced by the TKK construction}

Let $\cV$ be a Jordan pair. Recall that the inner derivations are defined by $\nu(x,y):=(D(x,y),-D(y,x)) \in \mathfrak{gl}(\cV^+)\oplus\mathfrak{gl}(\cV^-)$, where $(x,y) \in\cV$. Consider the 3-graded Lie algebra $L=\TKK(\cV)=L^{-1}\oplus L^0\oplus L^1$ defined by the TKK construction, due to Tits, Kantor and Koecher (see \cite{CS11} and references therein). That is, 
$$ L^{-1}=\cV^-, \quad L^1=\cV^+, \quad L^0=\lspan\{\nu(x,y) \med (x,y) \in \cV \}, $$ 
and the multiplication is given by 
$$ [a+X+b,c+Y+d]:=(Xc-Ya)+([X,Y]+\nu(a,d)-\nu(c,b))+(Xd-Yb) $$
for each $X,Y\in L^0$, $a,c\in L^1$, $b,d\in L^{-1}$. This 3-grading will be called the TKK-\emph{grading}. A $G$-grading on $L=\TKK(\cV)$ will be called TKK-\emph{compatible} if $L^1$ and $L^{-1}$ are $G$-graded subspaces (and, therefore, so is $L^0=[L^1,L^{-1}]$). In this case, we denote $L^n_g = L^n \cap L_g$ for $n\in\{-1,0,1\}$ and $g\in G$.

Consider a Jordan pair $\cV$ with associated Lie algebra $L=\TKK(\cV)$. Let $\Gamma$ be a $G$-grading on $\cV$.  For each homogeneous $x\in \cV^+_g$ and $y\in \cV^-_h$, $D(x,y)$ is a graded endomorphism of $\cV^+$ of degree $g+h$, and similarly for $\cV^-$. Hence $L^0$ is graded with $L^0_g=\{\nu\in L^0 \med [\nu,L^\sigma_h]\subseteq L^\sigma_{g+h}\ \forall h\in G\}$, and 
we can extend $\Gamma$ to a TKK-compatible $G$-grading $E_G(\Gamma) : L = \bigoplus_{g\in G} L_g$, where $ L^1_g=\cV^+_g$, $L^{-1}_g=\cV^-_g$ and $L^0_g = \lspan\{\nu(x,y) \med \deg(x^+) + \deg(y^-) =g \}$. Conversely, any TKK-compatible $G$-grading $\widetilde{\Gamma}$ on $L$ restricts to a $G$-grading $R_G(\widetilde{\Gamma})$ on $\cV$, because $\{x^+, y^-, z^+\} = [[x^+, y^-], z^+]$.

Denote by $\Grad_G(\cV)$ the set of $G$-gradings on $\cV$, and by $\TKKGrad_G(L)$ the set of TKK-compatible $G$-gradings on $L$. We will call $E_G \colon \Grad_G(\cV) \rightarrow \TKKGrad_G(L)$ the {\em extension} map and $R_G \colon \TKKGrad_G(L) \rightarrow \Grad_G(\cV)$ the {\em restriction} map.

\begin{theorem} \label{correspondence} 
Let $\cV$ be a Jordan pair with associated Lie algebra $L=\TKK(\cV)$, and let $G$ be an abelian group. Then, the maps $E_G$ and $R_G$ are inverses of each other. Coarsenings are preserved by the correspondence, i.e., given a $G_i$-grading $\Gamma_i$ on $\cV$ with extended $G_i$-grading $\widetilde{\Gamma}_i = E_{G_i}(\Gamma_i)$ on $L$, for $i=1,2$,  and a homomorphism $\alpha \colon G_1\rightarrow G_2$, then $\Gamma_2 = {}^\alpha\Gamma_1$ if and only if $\widetilde{\Gamma}_2 = {}^\alpha\widetilde{\Gamma}_1$. If $G=\Univ(\Gamma)$, then $G=\Univ(E_G(\Gamma))$. Moreover, $\Gamma$ is fine and $G=\Univ(\Gamma)$ if and only if $E_G(\Gamma)$ is fine and $G=\Univ(E_G(\Gamma))$.
\end{theorem}
\begin{proof}
By construction, $E_G$ and $R_G$ are inverses of each other.

Assume that $\Gamma_2 = {}^\alpha\Gamma_1$ for some homomorphism $\alpha\colon G_1\rightarrow G_2$. Then, $\cV^\sigma_g\subseteq\cV^\sigma_{\alpha(g)}$ for any $\sigma=\pm$ and $g\in G_1$. Thus, $L^{\sigma1}_g\subseteq L^{\sigma1}_{\alpha(g)}$, which implies that $L^0_g=\sum_{g_1+g_2=g}[L^1_{g_1},L^{-1}_{g_2}] \subseteq \sum_{g_1+g_2=g}[L^1_{\alpha(g_1)},L^{-1}_{\alpha(g_2)}] \subseteq L^0_{\alpha(g)}$. Hence, $\widetilde{\Gamma}_1$ refines $\widetilde{\Gamma}_2$ and $\widetilde{\Gamma}_2 = {}^\alpha\widetilde{\Gamma}_1$. Conversely, if $\widetilde{\Gamma}_2 = {}^\alpha\widetilde{\Gamma}_1$, by restriction we obtain $\Gamma_2 = {}^\alpha\Gamma_1$. We have proved that coarsenings are preserved.

Consider $\widetilde{\Gamma}=E_G(\Gamma)$ with $G=\Univ(\Gamma)$. Note that $\Univ(\Gamma)$ and $\Univ(\widetilde{\Gamma})$ are generated by $\supp\Gamma$. Since the $\Univ(\widetilde{\Gamma})$-grading $\widetilde{\Gamma}$ restricts to $\Gamma$ as a $\Univ(\widetilde{\Gamma})$-grading, there is a unique homomorphism $G=\Univ(\Gamma)\rightarrow\Univ(\widetilde{\Gamma})$ that is the identity in $\supp\Gamma$; conversely, $\Gamma$ extends to $\widetilde{\Gamma}$ as a $G$-grading, so there is a unique homomorphism $\Univ(\widetilde{\Gamma})\rightarrow G$ that is the identity in $\supp\widetilde{\Gamma}$ (and in $\supp\Gamma$); therefore the compositions $\Univ(\widetilde{\Gamma})\rightarrow G \rightarrow\Univ(\widetilde{\Gamma})$ and $G \rightarrow\Univ(\widetilde{\Gamma}) \rightarrow G$ are the identity map, and $G=\Univ(\widetilde{\Gamma})$.

Suppose again that $\widetilde{\Gamma}=E_G(\Gamma)$. Note that if $\widetilde{\Gamma}$ is fine in the class of $\TKK$-compatible gradings, then the supports of $L^0$, $L^1$ and $L^{-1}$ are disjoint and therefore $\widetilde{\Gamma}$ is also fine in the class of all abelian group gradings on $L$. Now, note that $\Gamma$ is a fine $G$-grading on $\cV$ with $G=\Univ(\Gamma)$ if and only if $\Gamma$ satisfies the following property: if $\Gamma = {}^\alpha\Gamma_0$ for some $G_0$-grading $\Gamma_0$, where $G_0$ is generated by $\supp\Gamma_0$, and $G$ is generated by $\supp\Gamma$, and $\alpha \colon G_0 \rightarrow G$ is an homomorphism, then $\alpha$ is an isomorphism. The same is true for TKK-compatible gradings. Since the coarsenings are preserved in the correspondence, so does this property, and therefore, $\Gamma$ is fine and $G=\Univ(\Gamma)$ if and only if $\widetilde{\Gamma}$ is fine and $G=\Univ(\widetilde{\Gamma})$.
\end{proof}

\begin{remark} The fact that $\widetilde{\Gamma}=E_G(\Gamma)$ with $G=\Univ(\widetilde{\Gamma})$, in general, does not imply that $G=\Univ(\Gamma)$. We will show this now.

First, take a $G$-grading $\Gamma_\cV$ on a Jordan pair $\cV$ such that there are nonzero elements in $\supp L^0$ for $\widetilde{\Gamma}_\cV=E_G(\Gamma_\cV)$, and assume that $G=\Univ(\Gamma_\cV)$ (it is not hard to find examples satisfying this). By Theorem~\ref{correspondence}, $G=\Univ(\widetilde{\Gamma}_\cV)$. Now, consider the Jordan pair $\cW=\cV\oplus\cV'$ given by two copies of $\cV$. There is a $G\times G$-grading $\Gamma$ on $\cW$, where $\cW^\sigma_{(g,0)} = \cV^\sigma_g$, $\cW^\sigma_{(0,g)} = \cV'^\sigma_g$. Besides, $\Univ(\Gamma)=G\times G$, so by Theorem~\ref{correspondence}, we have $\Univ(\widetilde{\Gamma})=G\times G$ too, where $\widetilde{\Gamma}=E_{G\times G}(\Gamma)$. It suffices to find a proper coarsening $\widetilde{\Gamma}_1$ of $\widetilde{\Gamma}$ such that the restricted grading on $\cW$ has the same homogeneous components as $\Gamma$. Actually, if $G_1=\Univ(\widetilde{\Gamma}_1)$, then $\widetilde{\Gamma}_1=E_{G_1}R_{G_1}(\widetilde{\Gamma}_1)=E_{G_1}(\Gamma)$ (where $\Gamma$ is regarded as a $G_1$-grading) would be a proper coarsening of $\widetilde{\Gamma}=E_{\Univ(\Gamma)}(\Gamma)$, and therefore $G_1 \ncong \Univ(\Gamma)$.

Consider the $G\times\ZZ$-grading $\Gamma_1$ on $\cW$ given by $\cW^\sigma_{(g,0)} = \cV^\sigma_g$, $\cW^\sigma_{(g,\sigma1)} = \cV'^\sigma_g$. Then, $\Gamma_1$ and $\Gamma$ have the same homogeneous components. The extension $\widetilde{\Gamma}_1 = E_{G\times\ZZ}(\Gamma_1)$ is a proper coarsening of $\widetilde{\Gamma}$, because $\widetilde{\Gamma}_1$ satisfies $\supp L^0 \cap \supp L'^0 = \supp L^0 \neq \{0\}$ (where $L'=\TKK(\cV')$) and for $\widetilde{\Gamma}$ we had $\supp L^0 \cap \supp L'^0 = \{0\}$. This proves the claim of the Remark.
\end{remark}

\smallskip

Any automorphism $\varphi=(\varphi^+,\varphi^-)$ of $\cV$ extends in a unique way to an automorphism $\widetilde{\varphi}$ of $L$ (that leaves $L_{-1}$ and $L_1$ invariant, so $L_0$ is invariant too). Indeed, it must satisfy $\widetilde{\varphi}(\nu(x^+,y^-))=\widetilde{\varphi}([x^+,y^-])=[\varphi^+(x^+),\varphi^-(y^-)]=\nu(\varphi^+(x^+),\varphi^-(y^-))=\varphi\nu(x^+,y^-)\varphi^{-1}$, and this formula indeed defines an automorphism $\widetilde{\varphi}$ of $L$. Then, we can identify $\Aut\cV$ with a subgroup of Aut$L$, and so we have $\Aut\Gamma\leq\Aut\widetilde{\Gamma}$ and $\Stab\Gamma\leq\Stab\widetilde{\Gamma}$. We can also identify $\cW(\Gamma)\leq\cW(\widetilde{\Gamma})$. Indeed, 
\begin{align*}
\cW(\Gamma) &= \Aut\Gamma/\Stab\Gamma=\Aut\Gamma/(\Stab\widetilde{\Gamma} \cap \Aut\Gamma) \\
& \cong(\Aut\Gamma\cdot\Stab\widetilde{\Gamma})/\Stab\widetilde{\Gamma}\leq\Aut\widetilde{\Gamma}/\Stab\widetilde{\Gamma}=\cW(\widetilde{\Gamma}).
\end{align*}

But although $\cW(\Gamma)\leq\cW(\widetilde{\Gamma})$, these Weyl groups do not coincide in general, at least for the bi-Cayley and Albert Jordan pairs. Actually, we will see that their fine gradings are, up to equivalence, of the form $\Gamma=(\Gamma^+,\Gamma^-)$, where $\Gamma^+$ and $\Gamma^-$ have the same homogeneous components, and hence there is an order 2 automorphism of $L$ that interchanges $\cV^+\leftrightarrow \cV^-$ and belongs to $\Aut\widetilde{\Gamma}\setminus\Aut\Gamma$, so $\cW(\Gamma) < \cW(\widetilde{\Gamma})$.


\subsection{Some facts about gradings on semisimple Jordan pairs}

\begin{remark}\label{remarkgrads}
Notice that, as a consequence of Equation~\eqref{peirceproperties}, the Peirce spaces associated to an idempotent $e$ define a $\ZZ$-grading $\Gamma$ where the subspace $\cV^\sigma_i$ has degree $\sigma(i+1)$, and $\supp\Gamma=\{\pm1,\pm2,\pm3\}$. If a Jordan pair $\cV$ has a $G$-grading $\Gamma$, an idempotent $e=(e^+,e^-)$ of $\cV$ will be called \emph{homogeneous} if $e^\sigma$ is homogeneous in $\cV^\sigma$ for each $\sigma$. In that case, we have $\deg(e^+)+\deg(e^-)=0$, which implies that the projections $E_i^\sigma=E_i^\sigma(e)$ are homogeneous maps of degree $0$, and therefore the Peirce spaces $\cV_i^\sigma=E_i^\sigma(\cV^\sigma)$ are graded. If in addition the graded Jordan pair $\cV$ is semisimple, then any nonzero homogeneous element $x=e^\sigma\in\cV^\sigma_g$ can be completed to a homogeneous idempotent $e = (e^+, e^-) \in\cV$; indeed, we can take a homogeneous element $y\in\cV^{-\sigma}_{-g}$ such that $Q_x y=x$ (because $\cV$ is vNr and the quadratic products are homogeneous maps), and in consequence $e=(e^+,e^-)$ with $e^{-\sigma} := Q_y x$ is a homogeneous idempotent.

Since homogeneous elements are completed to homogeneous idempotents and these produce graded Peirce subspaces, it follows that we can always choose a maximal orthogonal system of idempotents whose elements happen to be homogeneous.
\end{remark}

\begin{theorem} \label{dim1}
Let $\cV$ be a finite-dimensional semisimple Jordan pair and $\Gamma$ a $G$-grading on $\cV$. Then: 
\begin{itemize}
\item[1)] If $g\in \supp\Gamma^\sigma$, the subpair $(\cV^\sigma_g, \cV^{-\sigma}_{-g})$ is semisimple.
\item[2)] For any subgroup $H\leq G$ with $H \cap \supp\Gamma \neq\emptyset$, the subpair given by $\cV^\sigma_H := \bigoplus_{h\in H}\cV^\sigma_h$ is semisimple. 
\item[3)] If $\Gamma$ is fine, the homogeneous components are $1$-dimensional.
\end{itemize}
\end{theorem}
\begin{proof}
1) Take $0\neq x\in\cV^\sigma_g$. By Remark~\ref{remarkgrads}, $x$ can be completed to an idempotent of $\cW=(\cV^\sigma_g, \cV^{-\sigma}_{-g})$, so $x$ is vNr in $\cW$. Hence, $\cW$ is vNr too, and by Theorem~\ref{semisimple}, $\cW$ is semisimple.

2) Consider the epimorphism $\alpha \colon G \rightarrow \bar G=G/H$ and the induced $\bar G$-grading $\bar\Gamma = {}^\alpha\Gamma$. Then, $\cV_H$ coincides with the $\bar G$-graded subpair $(\cV^\sigma_{\bar0}, \cV^{-\sigma}_{\bar0})$, that is semisimple by 1).

3) Let $\Gamma$ be fine and assume by contradiction that $\dim\cV^\sigma_g >1$, where we can assume without loss of generality that $\sigma = +$. Then, the subpair $\cW=(\cV^+_g,\cV^-_{-g})$ is semisimple by 1). By Theorem~\ref{semisimple}, $\cW$ is nondegenerate, so we can consider the rank function. We can take an element $x\in\cW^+$ of rank $1$ in $\cW$ (but not necessarily in $\cV$), and complete it to an idempotent $e=(x,y)$ of $\cW$. As in Remark~\ref{remarkgrads}, the Peirce spaces of the Peirce decomposition associated to $e$ are the homogeneous components of a $\ZZ$-grading on $\cV$, which is compatible with $\Gamma$ because the Peirce spaces are graded with respect to $\Gamma$. Thus, combining the $\ZZ$-grading with $\Gamma$ we get a $G\times\ZZ$-grading that refines $\Gamma$, given by $\cV^\sigma_{(g,i)} = \cV^\sigma_g \cap \cV^\sigma_i$. Since $\rk_\cW(x)=1$ and $\FF=\bar\FF$, we have $\FF x = Q(x)\cW^- =: \cW^+_2(e)$. But then, $\cV^+_g \cap \cV^+_2(e) = \cV^+_g \cap Q(x)\cV^- = \cV^+_g \cap Q(x)\cV^-_{-g} = \cV^+_g \cap Q(x)\cW^- = \cW^+_2(e) = \FF x \subsetneqq \cV^+_g$ and the refinement is proper, which contradicts that $\Gamma$ is fine.
\end{proof}

The next Corollary is a nice application of the above results to the study of gradings on Jordan algebras.

\begin{corollary}
Let $J$ be a finite-dimensional semisimple Jordan algebra and $\Gamma$ a fine $G$-grading on $J$ with $\dim J_0 = 1$. Then, all the homogeneous components of $\Gamma$ have dimension $1$.
\end{corollary}
\begin{proof}
Assume by contradiction that some homogeneous component has dimension bigger than $1$. Consider the Jordan pair $\cV = (J,J)$ and let $\widetilde{\Gamma} = (\Gamma, \Gamma)$ be the induced $G$-grading on $\cV$. Since some component of $\widetilde{\Gamma}$ has dimension bigger than $1$, we can refine $\widetilde{\Gamma}$ to a fine grading $\widetilde{\Gamma}'$ on $\cV$, which will have all components of dimension $1$. Then, by Proposition~\ref{regulargradspairs}, the shift $\widetilde{\Gamma}'^{[g]}$ for $g=-\deg(1^+)$ restricts to a group grading on $J$, which is a proper refinement of $\Gamma$, a contradiction. 
\end{proof}

Some examples of homogeneous bilinear forms are given by trace forms: this is the case of gradings on Hurwitz algebras, matrix algebras, and the Albert algebra. Other well-known example is the Killing form of a graded semisimple Lie algebra. The generic trace plays the same role for graded Jordan pairs and graded triple systems.

\begin{proposition} \label{lemmatracejordanpairs} 
Let $\cV$ be a finite-dimensional simple Jordan pair. Then, the generic trace of $\cV$ is homogeneous for any grading on $\cV$.
\end{proposition}
\begin{proof}
Suppose that $\cV$ is $G$-graded. We know by \cite[Proposition~16.7]{L75} that the minimal polynomial $m(T,X,Y)$ of a finite-dimensional Jordan pair $\cV$ is invariant by the automorphism group scheme $\AAut(\cV)$. Hence the generic trace $t$ is $\AAut(\cV)$-invariant, i.e., $t(\varphi^+(x), \varphi^-(y))=t(x,y)$ for all $\varphi\in\Aut_R\cV_R$, $x\in\cV^+_R$, $y\in\cV^-_R$ and $R$ an associative commutative unital $\FF$-algebra. In particular, if we take the group algebra $R = \FF G$ we can consider the automorphism $\varphi$ of $\cV_R$ given by $\varphi^\sigma(v^\sigma_g \otimes 1) = v^\sigma_g \otimes g$ for each $\sigma = \pm$ and each homogeneous element $v^\sigma_g \in \cV^\sigma_g$. In order to avoid confusion, the binary operation in $G$ will be denoted multiplicatively here. Now, fix homogeneous elements $v^+_g \in \cV^+_g$ and $v^-_h \in \cV^-_h$. On the one hand, we know that $t(\varphi^+(v^+_g \otimes 1),\varphi^-(v^-_h \otimes 1)) = t(v^+_g \otimes 1, v^-_h \otimes 1) = t(v^+_g ,v^-_h) \otimes 1$ by $\Aut_R(\cV_R)$-invariance. On the other hand, $t(\varphi^+(v^+_g \otimes 1),\varphi^-(v^-_h \otimes 1)) = t(v^+_g \otimes g, v^-_h \otimes h) = t(v^+_g, v^-_h) \otimes gh$ by definition of $\varphi$. Therefore, $t(v^+_g ,v^-_h) \otimes 1 = t(v^+_g, v^-_h) \otimes gh$. We conclude that $gh=1$ whenever $t(v^+_g ,v^-_h) \neq 0$. 
\end{proof}

\begin{remark} \label{homogeneouskernel}
Assume that we have a graded finite-dimensional simple Jordan pair $\cV$. Hence the generic trace form $t$ of $\cV$ is homogeneous by Proposition~\ref{lemmatracejordanpairs}. If $t$ is nondegenerate, $\cV^{\sigma}_g$ and $\cV^{-\sigma}_{-g}$ are dual relative to $t$ and have the same dimension. For any homogeneous element $x\in\cV^\sigma_g$, define $t_x:\cV^{-\sigma}\rightarrow \FF$, $y\mapsto t(x,y)$. Since $t$ is homogeneous, the subspace $\ker(t_x)$ is graded too; we will use this fact in some proofs later on.
\end{remark}

\section{Exceptional Jordan pairs and triple systems}

In this section we first recall the definitions of the well-known Jordan pairs and triple systems of types bi-Cayley and Albert. We also give examples of gradings on these Jordan systems and compute their universal groups. These examples will be used in the following section to classify the gradings up to equivalence.

\subsection{Jordan pairs and triple systems of types bi-Cayley and Albert}

\begin{df}
The {\em Albert pair} is the Jordan pair associated to the Albert algebra, that is, $\cV_\alb:=(\alb,\alb)$ with the products $Q^\sigma_x(y)=U_x(y):=2L^2_x(y)-L_{x^2}(y)$. Its associated Jordan triple system $\cT_\alb:=\alb$, with the product $Q_x = U_x$, will be called the {\em Albert triple system}.  It is well-known (see \cite[Theorem~1]{MC70} and \cite[Theorem~1]{MC69}) that the $U$-operator can be written as $U_x(y) = T(x,y)x - x^\#\times y$.

We will write for short the vector spaces $\cB:= \cC\oplus\cC$, $\cC^\sigma_1:=(\cC\oplus0)^\sigma$ and $\cC^\sigma_2:=(0\oplus\cC)^\sigma$. Let $n$ be the norm of $\cC$. The quadratic form $q\colon\cB\to\FF$, $q((x_1,x_2)):=n(x_1)+n(x_2)$, will be called the \emph{norm} of $\cB$. The nondegenerate bilinear form defined by $t\colon \cB\times\cB\rightarrow\FF$, $t(x,y)=n(x_1,y_1)+n(x_2,y_2)$ for $x = (x_1,x_2)$, $y = (y_1,y_2) \in\cB$ (i.e., $t$ is the linearization of $q$) will be called the \emph{trace} of $\cB$. 

Denote by $\cV_\alb^{12}$ the Jordan subpair $(\cV_\alb)_1(e)$ of the Peirce decomposition $\cV_\alb = (\cV_\alb)_0(e) \oplus (\cV_\alb)_1(e) \oplus (\cV_\alb)_2(e)$ relative to the idempotent $e=(E_3,E_3)$; that is, $(\cV_\alb^{12})^+=(\cV_\alb^{12})^-:=\iota_1(\cC)\oplus\iota_2(\cC)$. Identifying $\iota_1(\cC)\oplus\iota_2(\cC)\equiv\cC\oplus\cC=\cB$, the trace $t$ of $\cB$ is also defined as a map $(\cV_\alb^{12})^+ \times (\cV_\alb^{12})^- \to \FF$.
\end{df}

\begin{proposition} 
The quadratic and triple products of $\cV_\alb^{12}$ are given by:
\begin{equation*}
U_x(y)= 4t(x,y)x 
 -4n(x_1)\iota_1(y_1) -4n(x_2)\iota_2(y_2) 
 -4\iota_1 \big( \bar y_2(x_2x_1) \big) -4\iota_2 \big( (x_2x_1)\bar y_1 \big),
\end{equation*}
and
\begin{align*}
\{x,y,z\}= U_{x,z}(y)= &~ 4t(x,y)z + 4t(z,y)x 
  -4n(x_1,z_1)\iota_1(y_1) -4n(x_2,z_2)\iota_2(y_2) \\
& -4\iota_1 \big( \bar y_2(x_2z_1+z_2x_1) \big) -4\iota_2 \big( (x_2z_1+z_2x_1)\bar y_1 \big),
\end{align*}
for all $x=\iota_1(x_1)+\iota_2(x_2)$, $y=\iota_1(y_1)+\iota_2(y_2)\in\iota_1(\cC)\oplus\iota_2(\cC)$.
\end{proposition}

\begin{proof} 
Take $x,y\in\iota_1(\cC)\oplus\iota_2(\cC)$. Then,
\begin{align*}
xy=& \big( \iota_1(x_1)+\iota_2(x_2) \big) \big( \iota_1(y_1)+\iota_2(y_2) \big) \\
=&2n(x_1,y_1)(E_2+E_3)+2n(x_2,y_2)(E_3+E_1)+\iota_3(\bar x_1 \bar y_2 + \bar y_1 \bar x_2), \\[.3cm]
L^2_x(y)=&2n(x_1,y_1)\iota_1(x_1) + n(x_2,y_2)\iota_1(x_1) + \iota_2 \big( (y_2x_1)\bar x_1 + (x_2y_1)\bar x_1 \big) \\
&+ n(x_1,y_1)\iota_2(x_2) + 2n(x_2,y_2)\iota_2(x_2) + \iota_1 \big( \bar x_2(y_2x_1) + \bar x_2(x_2y_1) \big) \\
=& 2t(x,y)x + n(x_1)\iota_2(y_2) + n(x_2)\iota_1(y_1) - \iota_1 \big( \bar y_2(x_2x_1) \big) - \iota_2 \big( (x_2x_1)\bar y_1 \big), 
\end{align*}
so we get
\begin{align*}
x^2=& 4n(x_1)(E_2+E_3) + 4n(x_2)(E_3+E_1) + 2\iota_3(\bar x_1 \bar x_2), \\[.3cm]
L_{x^2}(y) =& 4n(x_1)\iota_1(y_1) + 2n(x_2)\iota_1(y_1) + 2\iota_2 \big( (x_2x_1)\bar y_1 \big) \\
&+ 2n(x_1)\iota_2(y_2) + 4n(x_2)\iota_2(y_2) + 2\iota_1 \big( \bar y_2(x_2x_1) \big) \\
=& \big( 4n(x_1)+2n(x_2) \big) \iota_1(y_1) + \big( 2n(x_1)+4n(x_2) \big) \iota_2(y_2) \\
&+ 2\iota_2 \big( (x_2x_1)\bar y_1 \big) +2\iota_1 \big( \bar y_2(x_2x_1) \big). 
\end{align*}
Then, by substituting $U_x(y):=2L_x^2(y)-L_{x^2}(y)$ we obtain the first expression, and its linearization is the second one.
\end{proof}

\begin{df} \label{parB} 
Define the {\em bi-Cayley pair} as the Jordan pair $\cV_\cB:=(\cB,\cB)$ with products:
\begin{equation} \begin{aligned}
Q^\sigma_x(y) = Q_x(y)= & t(x,y)x - \Big( n(x_1)y_1+\bar y_2(x_2x_1),  n(x_2)y_2+(x_2x_1)\bar y_1 \Big) \\
=& \Big( x_1\bar y_1x_1+\bar x_2(y_2x_1), x_2\bar y_2x_2+(x_2y_1)\bar x_1 \Big),  \\[.3cm]
\{x,y,z\}^\sigma = \{x,y,z\}=& Q_{x,z}(y) \\ 
=& t(x,y)z + t(z,y)x  \\
& - \Big( n(x_1,z_1)y_1+\bar y_2(x_2z_1+z_2x_1), n(x_2,z_2)y_2+(x_2z_1+z_2x_1)\bar y_1 \Big) \\
=& \Big( x_1(\bar y_1z_1)+z_1(\bar y_1x_1)+\bar x_2(y_2z_1)+\bar z_2(y_2x_1), \\
& \hspace{2cm}  (x_2\bar y_2)z_2+(z_2\bar y_2)x_2+(x_2y_1)\bar z_1+(z_2y_1)\bar x_1 \Big).
\end{aligned} \end{equation}
\end{df}

Since the products of $\cV_\cB$ and $\cV_\alb^{12}$ are proportional, it is clear that $\cV_\cB$ is a Jordan pair and the map $\cV_\alb^{12} \rightarrow \cV_\cB$, $\iota_1(x_1)+\iota_2(x_2)\mapsto(2x_1,2x_2)$ is an isomorphism of Jordan pairs if $\chr\FF\neq2$. We also define the {\em bi-Cayley triple system} as the Jordan triple system $\cT_\cB:=\cB$ associated to the bi-Cayley pair $\cV_\cB$, so its quadratic and triple products are defined as for $\cV_\cB$.

\begin{df}
Consider $\mathcal{M}_{1\times2}:=(\mathcal{M}_{1\times2}(\cC),\mathcal{M}_{1\times2}(\cC^\text{op}))$, which is known to be a simple Jordan pair (see \cite{L75}). The quadratic products are given by
$Q_x(y)=x(y^*x)$, where $y^*$ denotes $y$ trasposed with coefficients in the opposite algebra. Considering elements in $\cC$, we can write:
\begin{equation} \begin{aligned}
Q^+_x(y)=& x(yx)= \Big( x_1y_1x_1+x_2(y_2x_1), x_1(y_1x_2)+x_2y_2x_2 \Big), \\
Q^-_y(x)=& (yx)y= \Big( y_1x_1y_1+(y_1x_2)y_2, (y_2x_1)y_1+y_2x_2y_2 \Big),
\end{aligned} \end{equation}
where we have omitted some parentheses using the alternativity of $\cC$.
\end{df}

Although the following result is probably known, the author does not know of a reference, so we include the proof.

\begin{proposition} \label{isombicayleypairs}
The Jordan pairs $\cV_\cB$ and $\mathcal{M}_{1\times2}$ are isomorphic.
\end{proposition}

\begin{proof} 
There is an isomorphism  $\varphi=(\varphi^+,\varphi^-) \colon \cV_\cB\rightarrow\mathcal{M}_{1\times2}$ given by:
$$ \varphi^+ \colon (x_1,x_2) \mapsto (\bar x_2,x_1), \hspace{1cm} \varphi^- \colon (y_1,y_2) \mapsto (y_2,\bar y_1). $$
Indeed,
\begin{align*}
\varphi^+(Q_xy)=& \varphi^+ \big( x_1\bar y_1x_1+\bar x_2(y_2x_1), (x_2\bar y_2)x_2+(x_2y_1)\bar x_1 \big) \\
=& \big(\bar x_2y_2\bar x_2+x_1(\bar y_1\bar x_2), x_1\bar y_1x_1+\bar x_2(y_2x_1) \big), \\[.2cm]
Q^+_{\varphi^+(x)} \big( \varphi^-(y) \big) =& Q^+_{(\bar x_2,x_1)}(y_2,\bar y_1)= \big( \bar x_2y_2\bar x_2+x_1(\bar y_1\bar x_2),\bar x_2(y_2x_1)+x_1\bar y_1x_1 \big), \\[.2cm]
\varphi^-(Q_yx)=& \varphi^- \big( y_1\bar x_1y_1+\bar y_2(x_2y_1), (y_2\bar x_2)y_2+(y_2x_1)\bar y_1 \big) \\
=& \big( y_2\bar x_2y_2+(y_2x_1)\bar y_1, \bar y_1x_1\bar y_1+(\bar y_1\bar x_2)y_2 \big), \\[.2cm]
Q^-_{\varphi^-(y)} \big( \varphi^+(x) \big) =& Q^-_{(y_2,\bar y_1)}(\bar x_2,x_1)= \big( y_2\bar x_2y_2+(y_2x_1)\bar y_1, (\bar y_1\bar x_2)y_2+\bar y_1x_1\bar y_1 \big),
\end{align*}
so we get $\varphi^+(Q_xy)=Q^+_{\varphi^+(x)} \big( \varphi^-(y) \big)$ and $\varphi^-(Q_yx)=Q^-_{\varphi^-(y)} \big( \varphi^+(x) \big)$.
\end{proof}

The generic trace form of $\cV_\alb$ is given by $T(x,y) := T(xy)$ where $T$ is the trace form of $\alb$ (\cite[17.10]{L75}), and the generic trace form of $\mathcal{M}_{1\times2}$ is given by $\text{tr}(xy^*)=\text{tr}(x_1y_1+x_2y_2)$ (\cite[17.9]{L75}), where $\text{tr}$ denotes the trace of $\cC$. Thus, applying the isomorphism in Proposition~\ref{isombicayleypairs}, we get that the generic trace of $\cV_\cB$ is the bilinear form $t = n \perp n$, that is, $t(x^+, y^-) = n(x_1, y_1) + n(x_2, y_2)$ for $x=(x_1,x_2)$, $y=(y_1,y_2) \in \cB$. (Note that $t = \frac{1}{4}T|_{\cV_\cB}$.) Also, we will refer to $t$, respectively to $T$, as the trace of $\cT_\cB$, respectively of $\cT_\alb$.

\begin{lemma} \label{lemmatrace} 
For any grading on the Jordan pairs and triple systems of types bi-Cayley or Albert, the trace is homogeneous.
\end{lemma}
\begin{proof}
Consequence of Proposition~\ref{lemmatracejordanpairs} and the fact that gradings on a triple system extend to gradings on the associated Jordan pair.
\end{proof}

\subsection{Some automorphisms} \label{automorphismssubsection}
In order to study the gradings on the Jordan pairs and triple systems under consideration, we will need to use some automorphisms defined in this section.
 
\begin{notation}
Recall that for any automorphism $\varphi = (\varphi^+, \varphi^-)$ of $\cV_\cB$ or $\cV_\alb$, the pair $(\varphi^-, \varphi^+)$ is also an automorphism, which we denote by $\widehat{\varphi}$. 

Denote by $\bar\tau_{12}$ the order 2 automorphism of $\alb$ (and therefore of $\cT_\alb$ and $\cV_\alb$) given by $E_1\leftrightarrow E_2$, $E_3\mapsto E_3$, $\iota_1(x)\leftrightarrow\iota_2(\bar x)$, $\iota_3(x)\mapsto\iota_3(\bar x)$. Similarly, we define $\bar\tau_{23}$ and $\bar\tau_{13}$.

Identifying $\cB$ with $\iota_1(\cC)\oplus\iota_2(\cC)$, the automorphism $\bar\tau_{12}$ of $\alb$ restricts to one of $\cT_\cB$ (and therefore of $\cV_\cB$), denoted also by $\bar\tau_{12}$, and given by: $$\bar\tau_{12} \colon \cB\rightarrow\cB, \quad (x_1,x_2)\mapsto (\bar x_2,\bar x_1).$$

\smallskip

Take $\lambda_1,\lambda_2,\lambda_3\in \FF^\times$ and $\mu_i:=\lambda^{-1}_i\lambda_{i+1}\lambda_{i+2}$. Define $c_{\lambda_1,\lambda_2,\lambda_3}$ by 
\begin{equation}\label{eq:iotas}
\begin{split}
&\iota_i(x)^+\mapsto\iota_i(\lambda_i x)^+,\quad \iota_i(x)^-\mapsto\iota_i(\lambda_i^{-1} x)^-,\\
& E_i^+\mapsto\mu_i E_i^+,\quad E_i^-\mapsto\mu_i^{-1} E_i^-.
\end{split} 
\end{equation}
One checks that $c_{\lambda_1,\lambda_2,\lambda_3}$ is an automorphism of $\cV_\alb$ (these were considered before, for example in \cite[1.6]{G01}). If $\lambda\in \FF^\times$, denote $c_\lambda:=c_{\lambda,\lambda,\lambda}$.

The automorphisms $c_{\lambda_1,\lambda_2,\lambda_3}$ restrict to $\cV_\cB$. For $\lambda,\mu\in \FF^\times$ define $c_{\lambda,\mu} \in \Aut \cV_\cB$ given by:
\begin{equation*}
c^+_{\lambda,\mu} \colon (x_1,x_2) \mapsto (\lambda x_1,\mu x_2), \quad
c^-_{\lambda,\mu} \colon (y_1,y_2) \mapsto (\lambda^{-1}y_1,\mu^{-1}y_2).
\end{equation*}
We also write $c_\lambda:=c_{\lambda,\lambda}$ (which is consistent with notation introduced in the previous paragraph).
\end{notation}

\begin{proposition} 
For each $a\in\cC$, there is an automorphism $\varphi_a$ of $\cV_\cB$ given by:
\begin{equation*}
\varphi_a^+ \colon (x_1,x_2) \mapsto (x_1-\bar x_2a, x_2), \quad
\varphi_a^- \colon (y_1,y_2) \mapsto (y_1, a\bar y_1+y_2),
\end{equation*}
for any $x_1,x_2,y_1,y_2\in\cC$.
\end{proposition}
\begin{proof} 
It suffices to check that $\varphi_a$ is the inner automorphism $\beta((a,0),(0,1))$. (Notice that $\varphi_a$ is the exponential of the derivation $d_a = -\nu((a,0),(0,1))$, which is nilpotent of order $2$ and given by $d_a^+(x_1,x_2)=(-\bar x_2a,0)$, $d_a^-(y_1,y_2)=(0,a\bar y_1)$.)
\end{proof}

\begin{remark} 
Since $\varphi_a\varphi_b=\varphi_{a+b}$ for any $a,b\in\cC$, these automorphisms generate an abelian subgroup of $\Aut\cV_\cB$ isomorphic to $(\cC,+)$. The same is true for $\widehat{\varphi}_a := (\varphi_a^-,\varphi_a^+)$, $a\in\cC$. Note that, since $\cB=\cC\oplus\cC$, we can write 
$$ \varphi_a^+ = \left(\begin{array}{cc} 1 & -r_{\bar a} \\ 0 & 1 \end{array}\right), \quad  \varphi_a^- = \left(\begin{array}{cc} 1 & 0 \\ l_{\bar a} & 1 \end{array}\right), $$ 
where $l_a$, $r_a$ denote the left and right multiplications by $a$ in the para-Cayley algebra $\bar \cC$. This matrix notation is useful to make computations with these automorphisms.
\end{remark}

\begin{proposition}\label{automfiofA} 
Let $\lambda\in \FF$ and $a\in\cC$ be such that $n(a)+\lambda^2=1$. There is an automorphism $\phi_1(a,\lambda)$ of $\alb$ given by:
\begin{align*}
x=\sum_{i=1}^3 \big( \alpha_iE_i+\iota_i(x_i) \big) \; \mapsto &  \;  \alpha_1E_1 + \big( \alpha_2\lambda^2+\alpha_3n(a)+2\lambda n(\bar a,x_1) \big) E_2 \\
& + \big( \alpha_2n(a)+\alpha_3\lambda^2-2\lambda n(\bar a,x_1) \big) E_3 \\
& +\iota_1 \Big( x_1+ \big( \frac{1}{2}\alpha_3\lambda-\frac{1}{2}\alpha_2\lambda-n(\bar a,x_1) \big) \bar a \Big) \\
& +\iota_2(\lambda x_2-\bar x_3a)+\iota_3(\lambda x_3+a\bar x_2).
\end{align*}
\end{proposition}
\begin{proof} Straightforward. \end{proof}

\begin{proposition} \label{phialambda}
Let $a\in\cC$ and $\lambda\in \FF$ be such that $n(a)+\lambda^2=1$. There is an automorphism of $\cT_\cB$ given by:
$$ \varphi_{a,\lambda} \colon (x_1,x_2)\mapsto(\lambda x_1-\bar x_2a,a\bar x_1+\lambda x_2). $$
Moreover, $\varphi_{a,\lambda}\in \Ort^+(\cB,q)$. 
\end{proposition}
\begin{proof} 
Note that, if we identify $\cB$ with $\iota_2(\cC)\oplus\iota_3(\cC)\subseteq\alb$, then $\varphi_{a,\lambda}$ is the restriction of $\phi_1(a,\lambda)$ to $\cB$, so it is an automorphism. We will give a different proof now. 
In case $n(a)=0$, $\lambda=\pm1$, define $\varphi:=\lambda \widehat{\varphi}_{\lambda a} \varphi_{\lambda a} \in \Aut\cV_\cB$, and in case $n(a)\neq0$, define $\varphi:=\widehat{\varphi}_{\mu a}\varphi_a\widehat{\varphi}_{\mu a}\in\Aut\cV_\cB$ with $\mu=\frac{1-\lambda}{n(a)}$. In both cases, it is checked that $\varphi_{a,\lambda} = \varphi\in\Aut\cT_\cB \leq \Ort(\cB,q)$. Since $\det(\varphi_a^\pm)=1=\det(\widehat{\varphi}_a^\pm)$ for any $a\in\cC$, we also have $\det(\varphi_{a,\lambda})=1$, and so $\varphi_{a,\lambda}\in \Ort^+(\cB,q)$.
\end{proof}

\begin{remark} \label{remarkPhialambda}
In $\cT_\cB$ we have $Q_x(x)=q(x)x$ for any $x\in\cB$ and, as a consequence, $\Aut\cT_\cB\leq \Ort(\cB,q)$. Since $\cB = \cC\oplus\cC$, we can write $\varphi_{a,\lambda} = \left(\begin{array}{cc} \lambda & -r_{\bar a} \\ l_{\bar a} & \lambda \end{array}\right)$, where $l_{a}$, $r_{a}$ are the left and right multiplications by $a$ in the para-Cayley algebra $\bar \cC$.
\end{remark}

\begin{df} 
Let $V$ be a finite-dimensional vector space and $q \colon V\rightarrow F$ a nondegenerate quadratic form. Recall that the map $\tau(a)=a$, $a\in V$, is extended to an involution of the Clifford algebra $\mathfrak{Cl}(V,q)$, called the {\em standard involution}. The map $\alpha(a)=-a$, $a\in V$, extended to an automorphism of $\mathfrak{Cl}(V,q)$, produces the standard $\ZZ_2$-grading $\mathfrak{Cl}(V,q)=\mathfrak{Cl}(V,q)_{\bar0}\oplus\mathfrak{Cl}(V,q)_{\bar1}$. The {\em Clifford group} of $\mathfrak{Cl}(V,q)$ is defined as $\Gamma=\Gamma(V,q) := \{ x \in \mathfrak{Cl}(V,q)^\times \med x\cdot V \cdot x^{-1} \subseteq V\}$. Here $\cdot$ denotes the product of $\mathfrak{Cl}(V,q)$. The subgroup $\Gamma^+:= \Gamma \cap \mathfrak{Cl}(V,q)_{\bar0}$ is called the {\em even Clifford group}. The {\em spin group} is defined by $\Spin(V,q):=\{ x \in \Gamma^+ \med x \cdot \tau(x)=1 \}$. Note that $\Spin(V,q)$ is generated by the elements of the form $x \cdot y$ where $x, y \in V$ and $q(x)q(y) = 1$.

For each $u\in\Spin(V,q)$, define the map $\chi_u \colon V\to V$, $x\mapsto u\cdot x\cdot u^{-1}$. It is well-known that $\chi_u$ belongs to the special orthogonal group $\Ort^+(V,q)$, and $\Ort'(V,q):=\{\chi_u \med u\in\Spin(V,q) \}$ is called the \emph{reduced orthogonal group}. Moreover, if $q$ has maximal Witt index then $\Ort'(V,q) \trianglelefteq \Ort^+(V,q)$ and $\Ort^+(V,q)/\Ort'(V,q) \cong \FF^\times/(\FF^\times)^2$ (see \cite[4.8]{J89}), where $(\FF^\times)^2$ is the multiplicative group of squares of $\FF^\times$. Here $\FF$ is assumed to be algebraically closed, so we have $\Ort^+(V,q)=\Ort'(V,q)$. A triple $(f_1,f_2,f_3)\in\Ort(\cC,n)^3$ is said to be \emph{related} if $f_1(\bar x \bar y) = \overline{f_2(x)} ~ \overline{f_3(y)}$ for any $x,y\in\cC$. Note that if $(f_1,f_2,f_3)$ is a related triple, then $(f_2,f_3,f_1)$ is also a related triple. Related triples have the property that $f_i \in \Ort'(\cC,n)$, and there is a group isomorphism 
$$ \Spin(\cC,n) \longrightarrow \{\text{related triples in $\Ort(\cC,n)^3$} \}, \quad u \mapsto (\chi_u,\rho_u^+,\rho_u^-), $$
for certain associated maps $\rho_u^+$ and $\rho_u^-$ (see e.g. \cite{Eld00} for more details).
\end{df}

\begin{remark} \label{remarkrelatedtriples}
Note that, if $(f_1,f_2,f_3)\in\Ort(\cC,n)^3$ is a related triple, then it is easy to check that $(f_1, f_2)$ is an automorphism of the bi-Cayley triple system. It is well-known that the map $\alb \to \alb$, $E_i \mapsto E_i$, $\iota_i(x) \mapsto \iota_i(f_i(x))$ for $i=1,2,3$, is an automorphism of the Albert algebra (see e.g. \cite[Corollary~5.6]{EKmon}).
\end{remark}

\begin{lemma} \label{lemmaAlberto} 
For any $x_1,x_2\in\cC$ of norm $1$, there is a related triple $(f_1,f_2,f_3)$ in $\Ort(\cC,n)^3$ such that $f_i(x_i)=1$ for $i=1,2$. 
Besides, for any $f_1\in \Ort^+(\cC,n)$, there are $f_2,f_3\in \Ort^+(\cC,n)$ such that $(f_1,f_2,f_3)$ is a related triple in $\Ort(\cC,n)^3$.
\end{lemma}
\begin{proof} The first statement was proved in \cite[Lemma~5.25]{EKmon}. For the second part, since $\chi \colon \Spin(\cC,n) \to \Ort'(\cC,n) = \Ort^+(\cC,n)$ is onto, we can write $f_1=\chi_u$ for some $u\in\Spin(\cC,n)$, and $(\chi_u,\rho_u^+,\rho_u^-)$ is a related triple.
\end{proof}

Consider $\mathfrak{Cl}(\cC,n)$ with the the $\ZZ_2$-grading given by $\deg(x)=\bar1$ for each $x\in\cC$, and the standard involution of the Clifford algebra, that is, sending $x \mapsto x$ for $x\in\cC$. Consider $\End(\cC\oplus\cC)$ with the $\ZZ_2$-grading that has degree $\bar0$ on the endomorphisms that preserve the two copies of $\cC$ and degree $\bar1$ on the endomorphisms that swap these two copies, and the involution given by the adjoint relative to the quadratic form $n\perp n$ on $\cC\oplus\cC$.

The next result is a slight modification of \cite[Proposition~(35.1)]{KMRT98}:

\begin{proposition}\label{isomphi} 
Denote by $l_x$, $r_x$ the left and right multiplications in the para-Cayley algebra $\bar\cC=(\cC,*)$.  Then, the map 
\begin{equation*} 
\Phi \colon \cC\rightarrow \End(\cC\oplus\cC),  \quad x\mapsto \left(\begin{array}{cc} 0 & r_{\bar x} \\ l_{\bar x} & 0 \end{array}\right),
\end{equation*}
defines an isomorphism of superalgebras $\Phi \colon \mathfrak{Cl}(\cC,n)\rightarrow\End (\cC\oplus\cC)$ that preserves the involution.
\end{proposition}

\begin{proof} Since $\Phi(x)^2=n(x)\text{id}$ for $x\in\cC$, it follows that $\Phi$ extends to a homomorphism of superalgebras. But since $\mathfrak{Cl}(\cC,n)$ is simple and has the same dimension as $\End(\cC\oplus\cC)$, we have that $\Phi$ is an isomorphism. From $l^*_x=r_x$, we deduce that $\Phi$ is an isomorphism of algebras with involution.  
\end{proof}

\begin{remark}
Given $a\in\cC$ with $n(a)=1$, we have 
$$\Phi(a)=\left(\begin{array}{cc} 0 & r_{\bar a} \\ l_{\bar a} & 0 \end{array}\right)=
\left(\begin{array}{cc} 0 & -r_{\bar a} \\ l_{\bar a} & 0 \end{array}\right)\left(\begin{array}{cc} 1 & 0 \\ 0 & -1 \end{array}\right) = \varphi_{a,0}c_{1,-1}\in\Aut\cT_\cB,$$
and in particular, $\RT:=\Phi(\text{Spin}(\cC,n))\leq\Aut\cT_\cB$. 

For any $u\in\text{Spin}(\cC,n)$, $\Phi(u)=\left(\begin{smallmatrix}\alpha&0\\ 0&\beta\end{smallmatrix}\right)$ if and only if $(\bar\chi_u,\alpha,\beta)$ is a related triple (see \cite[Theorem~5.5]{EKmon}), with $\bar\chi_u(a)=\overline{\chi_u(\bar a)}$, so 
\[
\begin{split}
\RT&=\Phi(\text{Spin}(\cC,n))={\Bigr\{}\begin{pmatrix} \alpha&0\\ 0&\beta\end{pmatrix} : \alpha,\beta\in\Ort(\cC,n)\\ 
&\qquad\text{and there is $\gamma\in\Ort(\cC,n)$ such that $(\gamma,\alpha,\beta)$ is a related triple}{\Bigr\}},
\end{split}
\]
and this explains our notation $\RT$.
The subgroup $\RT\cong\text{Spin}(\cC,n)$ is generated by the elements of the form 
\[
\Phi(a)\Phi(b) = \left(\begin{array}{cc} r_{\bar a}l_{\bar b} & 0 \\ 0 & l_{\bar a}r_{\bar b} \end{array}\right)=\varphi_{a,0}c_{1,-1}\varphi_{b,0}c_{1,-1}=\varphi_{a,0}\varphi_{-b,0},
\]
with $n(a)=n(b)=1$. Note that the group $\Aut\cC$ embeds in $\RT$ because for any automorphism $f$ of $\cC$, $(f,f,f)$ is a related triple.
\end{remark}

\begin{remark} \label{remarkGeneratorsbiCayley}
Consider the subgroup $G_\cV = \langle \varphi_a, \widehat{\varphi}_a, c_\lambda \med a\in\cC, \lambda\in\FF^\times \rangle$ of $\Aut\cV_\cB$ and the subgroup $G_\cT = \langle \varphi_{a,\lambda} \med a\in\cC, \lambda\in\FF, n(a)+\lambda^2=1 \rangle$ of $\Aut\cT_\cB$. (We will prove later that $G_\cV = \Aut\cV_\cB$ and $G_\cT = \Aut\cT_\cB$.) It follows from the proof of Proposition~\ref{phialambda} that $G_\cT \leq G_\cV$. The group $\RT$ of related triples is contained in the subgroup generated by the automorphisms $\varphi_{a,0}$ with $n(a)=1$, so we have $\RT \leq G_\cT$. Also, $(-\id,\id,-\id)$ is a related triple, so $c_{1,-1}=(\id,-\id)\in\RT\leq G_\cT$ and hence $\bar\tau_{12} = \varphi_{1,0}c_{1,-1} \in G_\cT$. 

We claim that $c_{\lambda,\mu} \in G_\cV$ for any $\lambda,\mu \in \FF^\times$. For any $\lambda\in\FF^\times$ and $a\in\cC$ such that $\lambda n(a)=1$, we have $c_{\lambda, 1} = c_{-\sqrt{\lambda}} \varphi_{\sqrt{\lambda}a,0} \varphi_a \widehat{\varphi}_{\lambda a} \varphi_a \in G_\cV$. But since $c_\mu$ belongs to $G_\cV$ for any $\mu\in\FF^\times$, we deduce that $c_{\lambda,\mu} = c_{\lambda\mu^{-1}, 1} c_\mu \in G_\cV$ for any $\lambda,\mu \in \FF^\times$.
\end{remark}

\begin{remark}
Let $J$ be a unital Jordan algebra with associated Jordan pair $\cV = (J, J)$. Let $\Str(J)$ denote the \emph{structure group} of $J$, i.e., the group consisting of all the \emph{autotopies}, that is, the elements $g\in \GL(J)$ such that $U_{g(x)} = g U_x g^\#$ for some $g^\#\in\GL(J)$ and all $x\in J$. The structure group functor $\SStr(J)$ is defined by $\SStr(J)(R) = \Str_R(J_R)$. There is an isomorphism of group schemes $\AAut(\cV) \to \SStr(J)$, which is given by $\Aut_R(\cV_R) \to \Str_R(J_R)$, $(\varphi^+,\varphi^-) \mapsto \varphi^+$ for each associative commutative unital $\FF$-algebra $R$ (see \cite[Proposition~2.6]{L79} and \cite[Proposition~1.8]{L75} for more details). 

Let $M(\alb)$ and $M_1(\alb)$ be the groups of similarities and isometries for the norm of $\alb$ (notation as in \cite[Chap.IX]{J68}). By \cite[Chap.V, Th.4]{J68}, $\alb$ is reduced, so by \cite[Chap.IX, Ex.2]{J68}, a linear map $\alb\rightarrow\alb$ is a norm similarity if and only if it is an isotopy; that is, $M(\alb) = \Str(\alb)$. Also, if we identify $\langle c_\lambda \med \lambda\in \FF^\times\rangle\cong \FF^\times$, we have $M(\alb) = \FF^\times\cdot M_1(\alb)$.

For each norm similarity $\varphi$ of $\alb$, denote $\varphi^\dagger:=(\varphi^{-1})^*$, where $*$ denotes the adjoint relative to the trace form $T$ of $\alb$. Since the trace is invariant under automorphisms, it follows that
the automorphisms of $\cV_\alb$ are exactly the pairs $(\varphi, \varphi^\dagger)$ where $\varphi$ is a norm similarity of $\alb$. We know from \cite[Lemma~1.7]{G01} that, if $\varphi = (\varphi^+, \varphi^-)$ is an automorphism of $\cV_\alb$ where the norm similarity $\varphi^\sigma$ has multiplier $\lambda_\sigma$ then $\lambda_+\lambda_-=1$; also $\varphi^\sigma(x^\#)=\lambda_\sigma\varphi^{-\sigma}(x)^\#$.
Moreover, we have $\varphi^\sigma(x^{-1})=\varphi^{-\sigma}(x)^{-1}$ for each $x\in\alb^\times$ (because  $U_{\varphi^{-\sigma}(x)}\varphi^\sigma(x^{-1}) = \varphi^{-\sigma}(U_xx^{-1}) = \varphi^{-\sigma}(x)$ for each $x\in\alb^\times$).
\end{remark}

\subsection{Orbits of the automorphism groups of bi-Cayley systems}

The trace forms of the bi-Cayley and Albert pairs are nondegenerate, so by Proposition~\ref{orbits}, there are exactly three orbits in $\cB^\sigma$ for the bi-Cayley pair, and four orbits in $\alb^\sigma$ for the Albert pair, all of them determined by the rank function.

\begin{notation} 
Recall that the norm of the vector space $\cB$ is the quadratic form $q=n\perp n \colon \cB\rightarrow\FF$, given by $q(x)=n(x_1)+n(x_2)$ for $x = (x_1, x_2) \in\cB$. For $i=0,1,2$, denote by $\cO_i$ the subset of $\cB$ of elements of rank $i$ for the bi-Cayley pair. For each $\lambda\in \FF$, set $\cO_2(\lambda):=\{x\in\cO_2 \med q(x)=\lambda\}$. Thus $\cO_2=\dot{\bigcup}_{\lambda\in \FF}\cO_2(\lambda)$.
\end{notation}

\begin{lemma} \label{orbitsbiCayleytriple} 
The different orbits for the action of $\Aut \cT_\cB$ on $\cB$ are exactly $\cO_0=\{0\}$, $\cO_1$ and $\cO_2(\lambda)$ with $\lambda\in \FF$. Moreover, for $0\neq x=(x_1,x_2)\in \cB$ we have $x\in\cO_1$ if and only if $x_2x_1=0$ and $n(x_1)=n(x_2)=0$. The orbits are the same if we consider the action of the subgroup $G_\cT = \langle \varphi_{a,\lambda} \med a\in\cC, \lambda\in\FF, n(a)+\lambda^2=1 \rangle$.
\end{lemma}
\begin{proof} 
Recall that $\cO_i$, for $i=0,1,2$, are the orbits of the bi-Cayley pair. Also, note that $\Aut \cT_\cB \leq \Ort(\cB,q)$. Hence, the sets $\cO_0$, $\cO_1$, and $\cO_2(\lambda)$ for $\lambda\in \FF$, are disjoint unions of orbits of the bi-Cayley triple system. 

First, we will check that $\cO_2(\lambda)$ is an orbit for each $\lambda\neq0$. Take $x\in\cO_2(\lambda^2)$ with $\lambda\neq0$. We claim that $x$ belongs to the orbit of $(\lambda 1,0)$. By applying $\bar\tau_{12}$ if necessary, we can assume that $n(x_1)\neq0$. Since $q(x)=\lambda^2\neq0$, $n(x_1)\neq-n(x_2)$ and we can take $\mu\in \FF^\times$ such that $\mu^{-2}=1+\frac{n(x_2)}{n(x_1)}$. The element $a=-\mu n(x_1)^{-1}x_2x_1$ satisfies $n(a)+\mu^2=1$, so we can consider the automorphism $\varphi_{a,\mu}$ (see Proposition~\ref{phialambda}). Then, $\varphi_{a,\mu}(x)=(\mu(1-\frac{n(x_2)}{n(x_1)})x_1,0)$, and by Lemma~\ref{lemmaAlberto}, this element is in the orbit of $(\lambda 1,0)$. Hence, $\cO_2(\lambda^2)$ is an orbit for each $\lambda\neq0$. Since $\FF$ is algebraically closed, $\cO_2(\lambda)$ is an orbit too.

Second, given $0\neq x\in\cB$ we claim that $x\in\cO_1$ if and only if $x_2x_1=0$ and $n(x_1)=n(x_2)=0$. Indeed, $x\in\cO_1$ means that $Q_x\cB = \FF x$, i.e., $(n(x_1)y_1 + \bar y_2(x_2x_1), n(x_2)y_2 + (x_2x_1)\bar y_1) =  t(x,y)x - Q_x(y)$ must belong to $\FF x$ for any $y\in\cB$, which is equivalent to say that $x_2x_1=0$ and $n(x_1)=n(x_2)=0$.

Third, we will prove that $\cO_1$ is an orbit. Take $x=(x_1,x_2)\in\cO_1$. We know that $n(x_1)=n(x_2)=0$ and $x_2x_1=0$. Then, using $\bar\tau_{12}$ if necessary, we can assume that $x_1\neq0$ and $n(x_1)=0$, and by Lemma~\ref{lemmaAlberto} we can also assume that $x_1=e_1$ is a nontrivial idempotent. Take $e_2:=1-e_1$, and consider the Peirce decomposition of $\cC$ associated to the idempotents $e_i$ as always. Since $x_2x_1=0$, we have $x_2=\lambda e_2+u$ with $\lambda\in\FF$, $u\in U$ (see Subsection \ref{sectionCayleyAlbert}). Thus, $x=(e_1,\lambda e_2+u)$. But taking $a=-\lambda e_2-u$ and $\mu=1$ we have $n(a)+\mu^2=1$, so $\varphi_{a,1}$ is an automorphism. Therefore, $\varphi_{a,1}(x)=(e_1+(\lambda e_1+\bar u)(\lambda e_2+u),\lambda e_2+u-(\lambda e_2+u)e_2)=(e_1,0)$. This proves that $\cO_1$ is an orbit.

Finally, we claim that $\cO_2(0)$ is an orbit. Take $x\in\cO_2(0)$, and fix $\bi\in \FF$ with $\bi^2=-1$. It suffices to prove that $x$ is in the orbit of $(1,\bi 1)$. But we will prove first that if $n(x_1)=n(x_2)=0$, then there is an automorphism $\varphi$ of $\cT_\cB$ such that the two components of $\varphi(x)$ are nonisotropic. Indeed, since $x\notin\cO_1$ and $n(x_1)=n(x_2)=0$, we must have $x_2x_1\neq0$, and hence $x_1,x_2\neq0$. If $n(x_1,\bar x_2)\neq0$, it suffices to take $\mu=\frac{1}{\sqrt{2}}$ and apply $\varphi=\varphi_{\mu1,\mu}$ to $x$ to obtain an element with nonisotropic components. Otherwise, $n(x_1,\bar x_2)=0=n(x_i)$ and by Lemma~\ref{lemmaAlberto}, we can assume that $x_1=e_1$ is a nontrivial idempotent. Consider the idempotents $e_1$, $e_2:=1-e_1$ with their Peirce decomposition $\cC=\FF e_1\oplus \FF e_2\oplus U\oplus V$, so we have $x_2=\gamma e_2+u+v$ for some $\gamma\in \FF$, $u\in U$, $v\in V$. Since $x_2x_1\neq0$, we have $v\neq0$. Take $u_1\in U$ with $vu_1=e_2$, so we obtain $\varphi_{u_1,1}(x)=(1-\gamma u_1+uu_1,\gamma e_2+u+u_1+v)$, which has the first component nonisotropic. In conclusion, there is an automorphism $\varphi$ of $\cT_\cB$ such that $\varphi(x)$ has both components nonisotropic. By Lemma~\ref{lemmaAlberto}, we can assume that $x=(\lambda 1, \bi\lambda 1)$, for certain $0\neq\lambda\in \FF$. Take $a\in\cC$ with $\textnormal{tr}(a)=0$ and $n(a)=\frac{\lambda^2-1}{2\lambda^2}$, and $\mu\in \FF$ such that $n(a)+\mu^2=1$. Then, $y:=\varphi_{a,\mu}(x)=(\lambda\mu 1-\lambda\bi a,\lambda a+\lambda\mu\bi 1)$, and we have $n(y_1)=\lambda^2n(\mu1-\bi a)=\lambda^2(\mu^2-n(a))=\lambda^2(1-2n(a))=1$; since $\varphi_{a,\mu}\in \Ort(\cB,q)$, we obtain $n(y_2)=-1$. By Lemma~\ref{lemmaAlberto} again, we can assume that $x=(1,\bi 1)$, and therefore $\cO_2(0)$ is an orbit.
\end{proof}

We have a similar result for the orbits of the bi-Cayley pair:

\begin{lemma} \label{orbitsbiCayleypair} 
The orbits of $\cB^+$ under the action of the group $\Aut\cV_\cB$ coincide with the orbits under the action of $G_\cV = \langle \varphi_a, \widehat{\varphi}_a, c_\lambda \med a\in\cC, \; \lambda\in\FF^\times \rangle$.
\end{lemma}

\begin{proof}
First, recall that $\Aut\cV_\cB$ has $3$ orbits on $\cB^+$, determined by the rank function, that can take values $0$, $1$ and $2$ (see Proposition~\ref{orbits}). From now on, consider the action of $G_\cV$ on $\cB^+$. We have to prove that the orbits under the action of $G_\cV$ are $\cO_0$, $\cO_1$ and $\cO_2$. Clearly, $\cO_0 = \{0\}$ is an orbit of this action. Recall from Remark~\ref{remarkGeneratorsbiCayley} that $G_\cV$ contains the subgroup of related triples and $\bar\tau_{12}$. By Lemma~\ref{lemmaAlberto}, two nonzero elements of $\cC_1=\cC\oplus0$ of the same norm are in the same orbit under the action of $G_\cV$ (because $G_\cV$ contains the subgroup of related triples). Using automorphisms of type $c_\lambda$ and the fact that $\FF = \bar \FF$, we also deduce that two nonisotropic elements of $\cC_1$ belong to the same orbit; a representative element of this orbit is $(1,0)$. Note that $\dim\im Q_x$ is an invariant of the orbit of each element $x\in\cB$. Given $0\neq z\in\cC$ with $n(z)=0$, we have $\dim\im Q_0 = 0$, $\dim\im Q_{(z,0)} = 1$ and $\dim\im Q_{(1,0)} = 8$; consequently, there are exactly $3$ orbits on $\cC_1$. It suffices to prove that each element of $\cB$ belongs to an orbit of $\cC_1$. Fix $x=(x_1,x_2)\in\cB$ with $x_1,x_2\neq0$; we claim that there is an automorphism $\varphi$ in $G_\cV$ such that $\varphi^+(x)\in\cC_1$.

Assume that $n(x_i)\neq0$ for some $i=1,2$. We can apply $\bar\tau_{12}$ if necessary to assume that $n(x_1)\neq0$. Then, take $a=-n(x_1)^{-1}x_2x_1$, so we have $\widehat{\varphi}^+_a(x)=(x_1,0)\in\cC_1$. 

Now, consider the case with $n(x_1) = 0 = n(x_2)$. In the case that $n(x_1,\bar x_2)\neq0$, take $a=1$, so we get that $\varphi^+_a(x)$ has a nonisotropic component, which is the case that we have considered above. Otherwise, we are in the case that $n(x_i)=0=n(x_1,\bar x_2)$. By Lemma~\ref{lemmaAlberto}, without loss of generality we can assume that $e_1 := x_1$ is a nontrivial idempotent of $\cC$. Consider the associated Peirce decomposition $\cC=\FF e_1 \oplus \FF e_2 \oplus U\oplus V$ associated to the idempotents $e_1$ and $e_2 = 1-e_1$. Since $n(e_2, x_2)=n(\bar x_1, x_2)=0$, we have $x_2=\lambda e_2+u+v$ for certain elements $\lambda\in \FF$, $u\in U$, $v\in V$. There are two cases now:

$\bullet$ In case $v\neq0$, we can take $u_1\in U$ such that $vu_1=e_2$, so 
$\varphi^+_{u_1}(x)=(e_1-(\lambda\bar e_2+\bar u+\bar v)u_1,x_2)=(1-\lambda u_1+uu_1,x_2)$, where the first component is nonisotropic (it has norm 1), which is the case considered above.

$\bullet$ In case $v=0$, we have that $\widehat{\varphi}^+_{-\lambda1}(x)=\varphi^-_{-\lambda1}(x)=(e_1,u)$, and we can assume that $x=(e_1,u)$. But if $u\neq0$, there is $v\in V$ such that $uv=-e_1$, so $\bar\tau_{12} \varphi^+_v(x) = \bar\tau_{12} (e_1-\bar uv,u) = \bar\tau_{12} (0,u) = (-u, 0) \in \cC_1$, and we are done.
\end{proof}

\subsection{Automorphism groups of bi-Cayley systems}

In this subsection, we will give an explicit description of the automorphism groups of the bi-Cayley pair and triple system.

\begin{theorem} \label{generatorsAutVB}
The group $\Aut\cV_\cB$ is generated by the automorphisms of the form $\varphi_a$, $\widehat{\varphi}_a$ and $c_\lambda$ (with $a\in\cC$, $\lambda\in\FF^\times$). 
\end{theorem}
\begin{proof}
Take $\varphi\in\Aut\cV_\cB$ and call $G_\cV = \langle \varphi_a, \widehat{\varphi}_a, c_\lambda \med a\in\cC, \lambda\in\FF^\times \rangle$. We have to prove that $\varphi\in G_\cV$. Recall from Remark~\ref{remarkGeneratorsbiCayley} that related triples and automorphisms of type $c_{\lambda,\mu}$ belong to $G_\cV$.

By Lemma~\ref{orbitsbiCayleypair}, there is some element $\varphi'$ of $G_\cV$ such that $\varphi'\varphi(1,0)^+ = (1,0)^+$. Thus, without loss of generality (changing $\varphi$ with $\varphi'\varphi$) we can assume that $\varphi(1,0)^+ = (1,0)^+$. Since the image of the idempotent $((1,0)^+,(1,0)^-)$ is an idempotent of the form $((1,0)^+,(a,b)^-)$, it must be $(1,0)^+ = Q_{(1,0)^+} (a,b)^- = (\bar a,0)^+$, hence $a=1$. The composition $\varphi_{-b}\varphi$ fixes $(1,0)^\pm$, so we can assume (changing $\varphi$ with $\varphi_{-b}\varphi$) that the same holds for $\varphi$. In consequence, the subspaces $\cC_1^\sigma = \im Q_{(1,0)^\sigma}$ and $\cC_2^\sigma = \ker Q_{(1,0)^{-\sigma}}$ must be $\varphi$-invariant. Write $\varphi(0,1)^+ = (0,a)^+$ with $a\in\cC$. Since the element $\varphi(0,1)^+$ has rank $2$, we have $n(a)\neq0$, and composing with an automorphism of type $c_{1,\lambda}$ if necessary we can also assume that $n(a)=1$. Then, by Lemma~\ref{lemmaAlberto}, composing with a related triple we can assume that $\varphi$ fixes $(1,0)^+$ and $(0,1)^+$. Note that the subspaces $\cC_i^\sigma$ are still $\varphi$-invariant and we can write $\varphi^\sigma = \phi_1^\sigma \times \phi_2^\sigma$ with $\phi_i^\sigma \in \GL(\cC_i^\sigma)$. Then, since $(1,0)^+ = \varphi(1,0)^+ =  \varphi (Q_{(1,0)^+} (1,0)^-) = Q_{(1,0)^+} (\phi_1^-(1),0) = (\overline{\phi_1^-(1)},0)^+$, we have $\phi_1^-(1) = 1$, and similarly $\phi_2^-(1) = 1$. Therefore, $\varphi$ fixes the elements $(1,0)^\pm$ and $(0,1)^\pm$. 

Denote $\cC_0 = \{a\in\cC \med \text{tr}(a)=0 \}$. Since the trace $t$ of $\cV_\cB$ is invariant by automorphisms and $(1,0)^\pm$ are fixed by $\varphi$, we obtain that the subspaces $(\cC_0 \oplus 0)^\pm$ are $\varphi$-invariant (note that $\text{tr}(a) = t((a,0), (1,0))$), and the same holds for $(0 \oplus \cC_0)^\pm$. For each $z\in\cC_0$, we have $Q_{(1,0)^+} (z,0)^- = (-z,0)^+$, which implies that $(-\phi_1^+(z),0)^+ = \varphi Q_{(1,0)^+} (z,0)^- = Q_{(1,0)^+} (\phi_1^-(z),0)^- = (\overline{\phi_1^-(z)},0)^+ = (-\phi_1^-(z),0)^+$. Hence $\phi_1^+ = \phi_1^-$ and, in the same manner, $\phi_2^+ = \phi_2^-$. With abuse of notation, we can omit the superscript $\sigma$ and write $\varphi = \phi_1 \times \phi_2$. On the other hand, for each $z\in\cC_0$ we have $\{(1,0)^+, (0,1)^-, (0,z)^+\} = (-z,0)^+$, from where we get that $(-\phi_1(z),0)^+ = \varphi \{(1,0)^+, (0,1)^-, (0,z)^+\} = \{(1,0)^+, (0,1)^-, (0,\phi_2(z))^+\} = (-\phi_2(z),0)^+$. Thus, $\phi_1 = \phi_2$ and, with more abuse of notation we can omit the subscript $i=1,2$ and write $\varphi = \phi \times \phi$, where $\phi \in \GL(\cC)$. Moreover, applying $\varphi$ to the equality $\{(0,1), (0,x), (y,0)\} = (xy,0)$ we obtain $\phi(xy) = \phi(x)\phi(y)$, which shows that $\phi\in\Aut\cC$. Since $\Aut\cC \leq \RT \leq G_\cV$ (with the obvious identifications), we have $\varphi \in G_\cV$ and we are done. 
\end{proof}

\begin{theorem} \label{generatorsAutTB}
The group $\Aut\cT_\cB$ is generated by the automorphisms of the form $\varphi_{a,\lambda}$ (with $a\in\cC$ and $\lambda\in\FF$ such that $n(a)+\lambda^2 = 1$).
\end{theorem}
\begin{proof}
Take $\varphi\in\Aut\cT_\cB$ and call $G_\cT = \langle \varphi_{a,\lambda} \med a\in\cC, \lambda\in\FF, n(a)+\lambda^2=1 \rangle$. We have to prove that $\varphi\in G_\cT$. By Lemma~\ref{orbitsbiCayleytriple}, there is some element $\varphi'$ of $G_\cT$ such that $\varphi'\varphi(1,0) = (1,0)$. Thus, without loss of generality (changing $\varphi$ with $\varphi'\varphi$) we can assume that $\varphi(1,0) = (1,0)$. Now, the subspaces $\cC_1 = \im Q_{(1,0)}$ and $\cC_2 = \ker Q_{(1,0)}$ must be $\varphi$-invariant. Write $\varphi(0,1) = (0,a)$ with $a\in\cC$. We know from Remark~\ref{remarkPhialambda} that $\Aut\cT_\cB \leq \Ort(\cB,q)$, so we have $n(a) = q(0,a) = q(0,1) = 1$. Then, by Lemma~\ref{lemmaAlberto}, composing with a related triple we can assume without loss of generality that $\varphi$ fixes $(1,0)$ and $(0,1)$. Since the subspaces $\cC_i$ are $\varphi$-invariant, we can write $\varphi = \phi_1 \times \phi_2$ with $\phi_i \in \GL(\cC_i)$. With the same arguments as in the proof of Theorem~\ref{generatorsAutVB} we deduce that $\phi_1 = \phi_2 \in \Aut\cC$, and therefore $\varphi$ belongs to $G_\cT$. 
\end{proof}

We introduce now some notation that will be used in the following results of this section:

\begin{notation} \label{notSpinGroups}
We extend the norm $n$ on $\cC$ to a ten-dimensional vector space $W=\cC\perp (\FF e \oplus \FF f)$ with $n(e)=n(f)=0$ and $n(e,f)=1$. Fix $\bi\in\FF$ with $\bi^2=-1$ and note that the elements $x=e+f$ and $y=\bi(e-f)$ are orthogonal of norm $1$. Then, $e=(x-\bi y)/2$, $f=(x+\bi y)/2$. Also, denote $V = \cC \perp \FF x \subseteq W$.
\end{notation}

\begin{lemma} \label{lemmaSpinGroups}
With notation as above, we have $\Spin(W,n) = \langle 1+a\cdot e, 1+a\cdot f \med a\in \cC \rangle$ and $\Spin(V,n) = \langle \lambda 1+a\cdot x \med \lambda\in\FF, a\in \cC, n(a)+\lambda^2 = 1 \rangle$. 
\end{lemma}

\begin{proof} 
First, note that $e\cdot f+f\cdot e=1$, hence $e\cdot f\cdot e=e$ and $f\cdot e\cdot f=f$ in the Clifford algebra $\mathfrak{Cl}(W,n)$. Besides, $x \cdot x = 1$. For each $a\in\cC$, it is easily checked that $(1+a\cdot e)\cdot\tau(1+a\cdot e)=(1+a\cdot e)\cdot(1+e\cdot a)=1$, and also $(1+a\cdot e)\cdot W\cdot (1+e\cdot a)\subseteq W$, so $1+a\cdot e, 1+a\cdot f\in\Spin(W,n)$. Then $G_W := \langle 1+a\cdot e, 1+a\cdot f \med a\in\cC\rangle\leq\Spin(W,n)$. Similarly, $G_V := \langle \lambda 1+a\cdot x \med \lambda\in\FF, a\in \cC, n(a)+\lambda^2 = 1 \rangle \leq \Spin(V,n)$.

Now, note that $\Spin(V,n)$ is generated by elements of the form $(\lambda_1 x + a_1) \cdot (\lambda_2 x + a_2) = (\lambda_1 1+ a_1 \cdot x) \cdot (\lambda_2 1 - a_2 \cdot x)$ with $\lambda_i\in\FF$, $a_i\in\cC$ such that $\lambda_i^2 + n(a_i) = 1$. Therefore, $\Spin(V,n) = G_V$.

Since $(a_1+\lambda_1 e+\mu_1 f) \cdot (a_2+\lambda_2 e+\mu_2 f) = (a_1+\lambda_1 e+\mu_1 f) \cdot x \cdot (-a_2+\mu_2 e+\lambda_2 f) \cdot x$, it is clear that $\Spin(W,n)$ is generated by the elements of the form $g=(a+\lambda e+\mu f)\cdot x$ with $a\in\cC$, $\lambda,\mu\in\FF$ and $n(a)+\lambda\mu=1$, so it suffices to prove that these generators belong to $G_W$.

$\bullet$ Case $\lambda=\mu$. The generator has the form $g=(a+\lambda x)\cdot x=\lambda 1 +a\cdot x$ (i.e., a generator of $G_V$). If $n(a)\neq0$ we can write $\lambda 1 +a\cdot x = (1+\nu a\cdot f) \cdot (1+a\cdot e) \cdot (1+\nu a\cdot f)\in G_W$ with $\nu=\frac{1-\lambda}{n(a)}=\frac{1}{1+\lambda}$ (because $n(a)+\lambda^2=1$. This implies in particular that $-1\in G_W$, because if $a\in\cC$ satisfies $n(a)=1$ and we take $\lambda=0$, then $-1 = (a \cdot x)\cdot(a \cdot x) = (0 + a \cdot x)\cdot(0 + a \cdot x) \in G_W$. On the other hand, if $n(a)=0$, then $\lambda\in\{\pm1\}$ and we can write $\lambda 1 +a\cdot x = \lambda1 \cdot (1+\nu a\cdot f) \cdot (1+\lambda a\cdot e) \cdot (1+\nu a\cdot f)\in G_W$ with $\nu=\lambda/2$. 

$\bullet$ Case $\lambda\neq\mu$. The generator has the form $g=(a+\lambda e+\mu f)\cdot x$. Without loss of generality, we can assume that $\mu\neq0$ (the case $\lambda\neq0$ is similar).

Take $\alpha = \mu^2 \in\FF^\times$ and $b\in\cC$ with $n(b)\alpha=1$. Then, $(1+b\cdot e) \cdot (1+\alpha b\cdot f) \cdot (1+b\cdot e) = b\cdot(e+\alpha f) = \mu b \cdot (\mu^{-1} e + \mu f) = (\mu^{-1} e + \mu f) \cdot (-\mu b)$, and note that $n(-\mu b)=1$. In consequence, for any $b\in\cC$ with $n(b)=1$ we have $(\mu^{-1} e+\mu f)\cdot b\in G_W$, and therefore $(\mu^{-1} e+\mu f)\cdot x = ((\mu^{-1} e+\mu f)\cdot b)\cdot (b\cdot x)\in G_W$ (because $b \cdot x \in G_W$ by the case $\lambda = \mu$). Then, $(1-\mu a\cdot f)\cdot(1-\mu^{-1} a\cdot e)\cdot g = (\mu^{-1}e+\mu f)\cdot x\in G_W$, and therefore $g\in G_W$.
\end{proof}

The groups $\Aut\cV_\cB$ and $\Aut\cT_\cB$ are explicitly described by the following Theorem.

\begin{theorem} \label{ThAutomGroups}
With the same notation as above, define the linear maps $\Phi^\pm \colon W \to \End (\cC \oplus \cC)$ by
$$\Phi^\pm(a)=\left(\begin{array}{cc} 0 & r_{\bar a} \\ l_{\bar a} & 0\end{array}\right), \quad
\Phi^\pm(x)=\left(\begin{array}{cc} 1 & 0 \\ 0 & -1 \end{array}\right), \quad
\Phi^\pm(y)=\left(\begin{array}{cc} \pm\bi & 0 \\ 0 & \pm\bi \end{array}\right),$$
where $a\in\cC$. Then, the linear map
\begin{equation*}
\Psi \colon W \to \End (\cB \oplus \cB), \quad w \mapsto \left(\begin{array}{cc} 0 & \Phi^+(w) \\ \Phi^-(w) & 0 \end{array}\right),
\end{equation*}
defines an algebra isomorphism $\Psi \colon \mathfrak{Cl}(W, n) \to \End (\cB \oplus \cB)$. Moreover, if we identify each $\varphi \in \Aut\cV_\cB$ with
\begin{equation*} 
\left(\begin{array}{cc} \varphi^+ & 0 \\ 0 & \varphi^- \end{array}\right) \in \End(\cB \oplus \cB),
\end{equation*}
then $\Psi$ restricts to a group isomorphism $\Spin(W, n) \to \langle \varphi_a, \widehat{\varphi}_a \med a\in\cC \rangle \leq \Aut\cV_\cB$, which in turn restricts to a group isomorphism $\Spin(V, n) \to \Aut\cT_\cB$. Furthermore, $\Aut\cV_\cB \cong \Gamma^+(W, n)/\langle -\bi z \rangle$ with $\langle -\bi z \rangle \cong \ZZ_2$, where $z = \Psi^{-1}(c_\bi)$. (Recall that $\Gamma^+(W,n)$ is the even Clifford group.)
\end{theorem}

\begin{proof} 
Fix $a\in\cC$. First, note that $\Psi(a)^2 = n(a) \id$, $\Psi(x)^2 = \Psi(y)^2 = \id$. Also, the matrices $\Psi(a)$, $\Psi(x)$ and $\Psi(y)$ anticommute, so we have $\Psi(w)^2 = n(w)\id$ for each $w\in W$. Therefore, the linear map $W \to \End(\cB \oplus \cB)$, $w \mapsto \Psi(w)$, extends to an algebra homomorphism $\mathfrak{Cl}(W, n) \to \End(\cB \oplus \cB)$. Since $\mathfrak{Cl}(W, n)$ is simple and has the same dimension as $\End(\cB \oplus \cB)$, it follows that $\Psi$ is an isomorphism.

It can be checked that $\Psi$ sends $\lambda1 + a \cdot x \mapsto \varphi_{a,\lambda}$ (where $n(a)+\lambda^2=1$). We know by Theorem~\ref{generatorsAutTB} that $\Aut\cT_\cB = \langle \varphi_{a,\lambda} \med  a\in\cC, \lambda\in\FF, n(a)+\lambda^2 = 1 \rangle$, and on the other hand, by Lemma~\ref{lemmaSpinGroups} we have $\Spin(V, n) = \langle \lambda1 + a \cdot x \med a\in\cC, \lambda\in\FF, n(a)+\lambda^2=1 \rangle$, so that $\Psi$ restricts to an isomorphism $\Spin(V, n) \to \Aut\cT_\cB$.

Furthermore, $\Psi$ sends $1+a\cdot e \mapsto \varphi_a$, $1+a\cdot f \mapsto \widehat{\varphi}_a$. By Theorem~\ref{generatorsAutVB} we have that $\Aut\cV_\cB = \langle \varphi_a, \widehat{\varphi}_a, c_\lambda \med a\in\cC, \; \lambda\in \FF^\times\rangle$ and, by Lemma~\ref{lemmaSpinGroups}, we have $\Spin(W, n)=\langle1+a\cdot e,1+a\cdot f \med a\in\cC\rangle$. Consequently, $\Psi$ restricts to a group isomorphism $\Spin(W, n) \to \langle \varphi_a, \widehat{\varphi}_a \med a\in\cC \rangle$. Moreover, we obtain a group epimorphism 
\begin{equation} \label{LambdaEpimorphism}
\Lambda \colon \FF^\times \times \Spin(W, n) \to \Aut\cV_\cB, \quad (\lambda, x) \mapsto c_\lambda \circ \Psi(x). 
\end{equation}

It is well-known that $Z(\Spin(W, n)) = \langle z \rangle \cong \ZZ_4$, with $z^4 = 1$ and $z\notin\FF$ (also $Z(\mathfrak{Cl}(W, n)_{\bar0}) = \FF1 + \FF z$). Since $\Psi$ restricts to an isomorphism $\Spin(W, n) \to \langle \varphi_a, \widehat{\varphi}_a \rangle \leq \Aut\cV_\cB$, replacing $z$ by $-z$ if necessary, we have $\Psi(z)^\pm = \pm\bi \id_\cB$ (because $\Psi(z)\in Z(\Aut\cV_\cB) = \langle c_\lambda \med \lambda\in\FF^\times \rangle$ and $z^4 = 1$). Hence $\Psi(z) = c_\bi$ (note that this implies that $\Psi^{-1}(c_\lambda) = \frac{1}{2}(\lambda + \lambda^{-1})1 + \frac{1}{2\bi}(\lambda - \lambda^{-1})z$).

We claim that $\ker\Lambda = \langle (-\bi, z) \rangle$. It is clear that $\langle (-\bi, z) \rangle \leq \ker \Lambda$. Fix $(\lambda, x) \in\ker\Lambda$, so that $\Lambda(\lambda, x) = c_\lambda \circ \Psi(x) = 1$, i.e.,
\begin{equation*}
\left(\begin{array}{cc} \lambda \id^+ & 0 \\ 0 & \lambda^{-1} \id^- \end{array}\right)
\left(\begin{array}{cc} \Psi(x)^+ & 0 \\ 0 & \Psi(x)^- \end{array}\right)
= \left(\begin{array}{cc} \id^+ & 0 \\ 0 & \id^- \end{array}\right),
\end{equation*}
which in turn implies that $\Psi(x)^\pm \in \FF^\times\id$ and $\Psi(x)\in Z(\Aut\cV_\cB)$. Recall again that $\Psi$ restricts to an isomorphism $\Spin(W, n) \to \langle \varphi_a, \widehat{\varphi}_a \med a\in\cC \rangle \leq \Aut\cV_\cB$, so that $x\in Z(\Spin(W, n)) = \langle z \rangle$ and therefore $\ker\Lambda = \langle (-\bi, z) \rangle \cong \ZZ_4$. Therefore, we obtain $(\FF^\times \times \Spin(W, n))/\langle (-\bi, z) \rangle \cong \Aut\cV_\cB$.

Define a new epimorphism by means of
\begin{equation}
\widetilde{\Lambda} \colon \FF^\times \times \Spin(W, n) \to \Gamma^+(W, n), \quad (\lambda, x) \mapsto \lambda x.
\end{equation}
Then, $\ker \widetilde{\Lambda} = \langle (-1, -1) \rangle \cong \ZZ_2$ and $(\FF^\times \times \Spin(W, n))/\langle (-1, -1) \rangle \cong \Gamma^+(W, n)$.

Finally, note that $(-1, -1)\in \ker\Lambda$. Hence, the epimorphism $\Lambda$ factors through $\widetilde{\Lambda}$, and we obtain an epimorphism $\Gamma^+(W, n) \to \Aut\cV_\cB$ with kernel $\widetilde{\Lambda}\langle (-\bi,z) \rangle = \langle -\bi z \rangle \cong \ZZ_2$.
\end{proof}

\bigskip

Although not needed in what follows, the results in the previous Theorem may be stated in terms of affine group schemes, as indicated on the next result, where the same notations as in the previous Theorem are used. For the definitions of the affine group schemes corresponding to the spin or Clifford groups, the reader is referred to \cite{KMRT98}.

\begin{theorem} \label{ThSchemesBicayley}
Let $\FF$ be an arbitrary field of characteristic not $2$.  
\begin{itemize}
\item The affine group scheme $\AAut(\cT_\cB)$ is isomorphic to $\SSpin(V,n)$.
\item The affine group scheme $\AAut(\cV_\cB)$ is isomorphic to ${\boldsymbol\Gamma}^+(W,n)/\mmu_2$.
\end{itemize}
\end{theorem}
\begin{proof}
The morphism $\Psi$ in the proof of Theorem~\ref{ThAutomGroups} is functorial, so it induces a morphism of affine group schemes $\Psi:\SSpin(V,n)\rightarrow \AAut(\cT_\cB)$. If $\bar\FF$ denotes an algebraic closure of $\FF$, the corresponding homomorphism $\Psi_{\bar\FF}$ of $\bar\FF$-points is an isomorphism by Theorem~\ref{ThAutomGroups}, and the differential $\textrm{d}\Psi:\Lie(\SSpin(V,n))\cong\frso(V,n)\rightarrow \Lie(\AAut(\cT_\cB))=\Der(\cT_\cB)$ is one-to-one, because it is nonzero and the orthogonal Lie algebra $\frso(V,n)$ is simple. To prove that $\Psi$ is an isomorphism is then enough to prove that $\textrm{d}\Psi$ is surjective, or that the dimension of $\Der(\cT_\cB)$ equals the dimension of $\frso(V,n)$, which is $36$. 

For this, let $\cJ$ be the Albert algebra $\alb$ and $e$ the idempotent $E_3$ of $\alb$. Then $\cT_\cB$ is the Peirce component $\cJ_{\frac{1}{2}}(e)=\{x\in\cJ \med ex=\frac{1}{2}x\}$, with triple product inherited from the one in $\cJ$: $\{x,y,z\}=x(yz)+z(xy)-(xz)y$. Then $\cT_\cB$ is a Lie triple system too (see the proof of Theorem~\ref{triplesystemautom}), with $[x,y,z]=\{x,y,z\}-\{y,x,z\}=-2(x,z,y)=-2[L_x,L_y](z)$. The simple Lie algebra $\Der(\cJ)$, which is simple of type $F_4$ (see \cite[Chapter IX, \S 1]{J68}) is graded by $\ZZ_2$, with $\Der(\cJ)\subo=\{d\in\Der(\cJ) \med d(e)=0\}$, of dimension $36$, and $\Der(\cJ)\subuno=\{[L_e,L_x] \med x\in \cT_\cB\}$. Moreover, the assignment $x\mapsto -\frac{1}{8}[L_e,L_x]$ gives an isomorphism of Lie triple systems $(\cT_\cB,[.,.,.])\rightarrow \Der(\cJ)\subuno$. Also $[[L_e,L_x],[L_e,L_y]]=-\frac{1}{4}[L_x,L_y]$ (see the proof of \cite[Chapter IX, Theorem 17]{J68}), and hence $\Der(\cJ)$ is, up to isomorphism, the standard enveloping Lie algebra of our Lie triple system $\Der(\cJ)\subuno$.

Any derivation of the Jordan triple system $(\cT_\cB,\{.,.,.\})$ induces a derivation of $(\cT_\cB,[.,.,.])$ which, in turn, induces an even derivation of its standard enveloping algebra $\Der(\cJ)$. Since the characteristic is not $2$, any derivation of $\Der(\cJ)$ is inner (see \cite[6.7]{Jantzen}), and we conclude that $\Der(\cT_\cB)\cong\Der(\cJ)\subo$, and hence its dimension is $36$. This finishes the proof of the first assertion.

\smallskip

The homomorphism $\Lambda$ in equation~\eqref{LambdaEpimorphism} is functorial and hence it induces a morphism of affine group schemes $\Lambda:\GG_m \times\SSpin(W,n)\rightarrow \AAut(\cV_\cB)$. This last group scheme:
 $\AAut(\cV_\cB)$, is smooth (see \cite[6.5]{L79}), and $\Lambda$ is surjective for $\bar\FF$-points, as shown in the proof of Theorem~\ref{ThAutomGroups}. We conclude (see e.g. \cite[Theorem A.48]{EKmon}), that $\Lambda$ is a quotient map. Moreover, $\textrm{d}\Lambda: \Lie(\GG_m \times\SSpin(W,n))=\FF\times\frso(W,n)\rightarrow \Lie(\AAut(\cV_\cB))=\Der(\cV_\cB)$ is one-to-one, because neither $\FF$ nor the simple Lie algebra $\frso(W,n)$ are in the kernel, and hence $\textrm{d}\Lambda$ is bijective. Therefore $\Lambda$ is separable and $\ker\Lambda$ is smooth (\cite[(22.13)]{KMRT98}). We conclude that $\AAut(\cV_\cB)$ is isomorphic to the quotient $(\GG_m \times\SSpin(W,n))/\boldsymbol{\mu}_4$. 
 
The same arguments apply to the natural morphism $\GG_m \times\SSpin(W,n)\rightarrow \mathbf{\Gamma}^+(W,n)$, whose kernel is the copy of $\mmu_2$ inside $\ker\Lambda=\mmu_4$. Therefore, both $\AAut(\cV_\cB)$ and $\mathbf{\Gamma}^+(W,n)$ are  quotients of $\GG_m \times\SSpin(W,n)$, and using the isomorphism $\mmu_4/\mmu_2\simeq \mmu_2$, we get that $\AAut(\cV_\cB)$ is isomorphic to the quotient $\mathbf{\Gamma}^+(W,n)/\mmu_2$.
\end{proof}

\subsection{Construction of fine gradings on the bi-Cayley pair}

Given a grading on $\cV_\cB$ such that $\cC_i^\sigma$ are graded subspaces for $i=1,2$, and $\sigma=\pm$, we will denote by $\deg^\sigma_i$ the restriction of $\deg$ to $\cC^\sigma_i$.

Recall that the trace of $\cV_\cB$ is homogeneous for each grading, i.e., we have that $t(x^+,y^-)\neq0$ implies $\deg(y^-) = -\deg(x^+)$ for $x^+$, $y^-$ homogeneous elements of the grading. Hence, to give a grading on $\cV_\cB$ it suffices to give the degree map on $\cV_\cB^+$.

\begin{example} \label{CDbasisbiCayleypair} 
Since $\chr\FF\neq2$, we can take a Cayley-Dickson basis $\{x_i\}_{i=0}^7$ of $\cC$, as in Section~\ref{sectionCayleyAlbert}. Let $\deg_\cC$ denote the associated degree of the $\ZZ_2^3$-grading on $\cC$. Then, we will call the set $\{(x_i,0)^\sigma, (0,x_i)^\sigma \med \sigma=\pm \}_{i=0}^7$ a {\em Cayley-Dickson basis} of $\cV_\cB$. It is checked directly that we have a fine $\ZZ^2 \times \ZZ_2^3$-grading on $\cV_\cB$ that is given by $\deg^+_1(x_i)=(1,0,\deg_\cC(x_i))=-\deg^-_1(x_i)$ and $\deg^+_2(x_i)=(0,1,\deg_\cC(x_i))=-\deg^-_2(x_i)$, and will be called the {\em Cayley-Dickson grading} on $\cV_\cB$. This grading is fine because its homogeneous components have dimension $1$.
\end{example}

Note that, for the Cayley-Dickson basis, the triple product is determined by:
\begin{itemize}
\item[i)] $\{(x_i,0),(x_j,0),(x_k,0)\}=(2\delta_{ij}x_k+2\delta_{jk}x_i-2\delta_{ik}x_j, 0)$,
\item[ii)] $\{(x_i,0),(0,x_j),(x_k,0)\}=0$,
\item[iii)] $\{(x_i,0),(x_j,0),(0,x_k)\}=(0, 2\delta_{ij}x_k-(x_kx_i)\bar x_j)$.
\end{itemize}
The rest of the cases are obtained by symmetry in the first and third components of the triple product, and using the automorphism $\bar\tau_{12} \colon \cC\oplus0 \leftrightarrow 0\oplus\cC$, $(x_1,x_2)^\sigma \mapsto (\bar x_2,\bar x_1)^\sigma$.

\begin{example} \label{CartanbasisbiCayleypair}
Let $\{z_i\}_{i=1}^8$ be a Cartan basis of $\cC$, as in Section~\ref{sectionCayleyAlbert}. Then, $\{(z_i,0)^\sigma, (0,z_i)^\sigma \med \sigma=\pm \}_{i=1}^8$ will be called a {\em Cartan basis} of $\cV_\cB$. It is checked directly that we have a fine $\ZZ^6$-grading on $\cV_\cB$ determined by

\begin{center}
\begin{tabular}{|c|l|l|}
 \hline
 deg & $\cC^+_1$ & $\cC^+_2$ \\
 \hline
 $e_1$ & $(0,0,1,0,0,0)$ & $(0,0,0,0,1,0)$ \\
 $e_2$ & $(0,0,0,1,0,0)$ & $(0,0,0,0,0,1)$ \\
 \hline
 $u_1$ & $(1,0,0,1,0,0)$ & $(1,0,0,0,1,0)$ \\
 $u_2$ & $(0,1,0,1,0,0)$ & $(0,1,0,0,1,0)$ \\
 $u_3$ & $(-1,-1,1,0,-1,1)$ & $(-1,-1,1,-1,0,1)$ \\
 \hline
 $v_1$ & $(-1,0,1,0,0,0)$ & $(-1,0,0,0,0,1)$ \\
 $v_2$ & $(0,-1,1,0,0,0)$ & $(0,-1,0,0,0,1)$ \\
 $v_3$ & $(1,1,0,1,1,-1)$ & $(1,1,-1,1,1,0)$ \\
 \hline
\end{tabular}
\end{center}

\noindent and $\deg(x^+)+\deg(y^-)=0$ for any elements $x^+$, $y^-$ of the Cartan basis such that $t(x^+,y^-)\neq0$, and will be called the {\em Cartan grading} on $\cV_\cB$. This grading is fine because its homogeneous components have dimension $1$. (Notice that the projection on the two first coordinates of the group coincides with the Cartan $\ZZ^2$-grading on $\cC$, which behaves well with respect to the product on $\cV_\cB$, so it suffices to show that the projection on the last four coordinates behaves well with respect to the product.)
\end{example}

We will prove now that the grading groups of these gradings are their universal groups.

\begin{proposition} \label{universalCDbiCayleypair}
The Cayley-Dickson grading on the bi-Cayley pair has universal group $\ZZ^2 \times \ZZ_2^3$.
\end{proposition}
\begin{proof}
Let $\{x_i\}_{i=0}^7$ be a Cayley-Dickson basis of $\cC$ with $x_0=1$. Let $\Gamma$ be a realization as a $G$-grading of the associated Cayley-Dickson grading on $\cV_\cB$, for some abelian group $G$. For each element $x$ of the Cayley-Dickson basis of $\cV_\cB$ we have $t(x^+,x^-)\neq0$, and since the trace is homogeneous, it has to be $\deg(x^+)+\deg(x^-)=0$. Define $g_i=\deg^+_1(x_i)=-\deg^-_1(x_i)$, $a=g_0=\deg_1^+(1)$ and $b=\deg^+_2(1)$. If $i\neq j$, then $Q_{(x_i,0)^+}^+(x_j,0)^-=(-x_j,0)^+$, so that we have $2g_i=2g_j$. Thus, $a_i:=g_i-g_0$ has order $\leq2$, and we have $\deg^+_1(x_i)=a_i+a$. If $i\neq0$, $\{(x_i,0)^+,(1,0)^-,(0,1)^+\}=(0,-x_i)^+$, so $\deg^+_2(x_i)=a_i+b$. If $0\neq i\neq j\neq0$, we have $\{(x_i,0)^+,(x_j,0)^-,(0,1)^+\}=(0,-x_i\bar x_j)^+$, and we get $\deg^+_2(x_ix_j)=(a_i+a_j)+b$, and also $\{(0,x_i)^+,(0,1)^-,(x_j,0)^+\}=(-x_ix_j,0)^+$, so $\deg^+_1(x_ix_j)=(a_i+a_j)+a$. Therefore, $\deg_\cC(x_i):=a_i$ defines a group grading by $\langle a_i \rangle$ on $\cC$ that is a coarsening of the $\ZZ_2^3$-grading on $\cC$. Therefore, there is an epimorphism $\ZZ^2 \times \ZZ_2^3 \to G$ that sends $(1,0,\bar0,\bar0,\bar0) \mapsto a$, $(0,1,\bar0,\bar0,\bar0) \mapsto b$, and restricts to an epimorphism $0 \times \ZZ_2^3 \to \langle a_i \rangle$, so we conclude that $\ZZ^2 \times \ZZ_2^3$ is the universal group.
\end{proof}

\begin{proposition} \label{universalCartanbiCayleypair}
The Cartan grading on the bi-Cayley pair has universal group $\ZZ^6$.
\end{proposition}
\begin{proof} 
Let $\{e_i,u_j,v_j \med i=1,2; j=1,2,3\}$ be a Cartan basis of $\cC$. Let $\Gamma$ be a realization as a $G$-grading of the associated Cartan grading on $\cV_\cB$, for some abelian group $G$. Recall that if $t(x^+,y^-)\neq0$ for homogeneous elements $x^+$, $y^-$, since the trace is homogeneous we have $\deg(x^+)+\deg(y^-)=0$, and therefore the degree is determined by its values in $\cV_\cB^+$. Put $a_1=\deg_1^+(e_1)$, $a_2=\deg_1^+(e_2)$, $b_1=\deg_2^+(e_1)$, $b_2=\deg_2^+(e_2)$. To simplify the degree map, define $g_i$ ($i=1,2$) by means of $\deg_1^+(u_1)=g_1+a_2$, $\deg_1^+(u_2)=g_2+a_2$, $g_3=-g_1-g_2$. Then we deduce:
\begin{align*}
& \{(e_1,0)^-,(u_i,0)^+,(0,e_1)^-\}=(0,u_i)^- \; (i=1,2) \Rightarrow \deg_2^+(v_i)=-g_i+b_2 (i=1,2), \\
& \{(v_i,0)^+,(e_2,0)^-,(0,e_2)^+\}=(0,-v_i)^+ \; (i=1,2) \Rightarrow \deg_1^+(v_i)=-g_i+a_1 (i=1,2), \\
& \{(u_i,0)^+,(e_1,0)^-,(0,e_1)^+\}=(0,-u_i)^+ \; (i=1,2) \Rightarrow \deg_2^+(u_i)=g_i+b_1 (i=1,2), \\
& \{(u_2,0)^+,(e_2,0)^-,(0,u_1)^+\}=(0,-v_3)^+ \Rightarrow \deg_2^+(v_3)=-g_3-a_1+a_2+b_1, \\
& \{(v_2,0)^+,(e_1,0)^-,(0,v_1)^+\}=(0,-u_3)^+ \Rightarrow \deg_2^+(u_3)=g_3+a_1-a_2+b_2, \\
& \{(0,u_3)^+,(0,e_2)^-,(e_2,0)^+\}=(-u_3,0)^+ \Rightarrow \deg_1^+(u_3)=g_3+a_1-b_1+b_2, \\
& \{(0,v_3)^+,(0,e_1)^-,(e_1,0)^+\}=(-v_3,0)^+ \Rightarrow \deg_1^+(v_3)=-g_3+a_2+b_1-b_2.
\end{align*}
The relations above show that the set $\{a_1, a_2, b_1, b_2, g_1, g_2\}$ generates $G$. Hence, there is an epimorphism $\ZZ^6 \to G$ determined by
\begin{align*}
& (1,0,0,0,0,0) \mapsto g_1, \quad (0,0,1,0,0,0) \mapsto a_1, \quad (0,0,0,0,1,0) \mapsto b_1, \\
& (0,1,0,0,0,0) \mapsto g_2, \quad (0,0,0,1,0,0) \mapsto a_2, \quad (0,0,0,0,0,1) \mapsto b_2,
\end{align*}
and therefore the universal group is $\ZZ^6$.
\end{proof}

\subsection{Construction of fine gradings on the Albert pair}

Recall that the trace of $\cV_\alb$ is homogeneous for each grading, i.e., we have that $T(x^+,y^-)\neq0$ implies $\deg(y^-) = -\deg(x^+)$ for $x^+$, $y^-$ homogeneous elements of the grading. Hence, to give a grading on $\cV_\alb$ it suffices to give the degree map on $\alb^+$.

\begin{example} \label{ExCDAlbertpair}
Consider a Cayley-Dickson basis $\{x_i\}_{i=0}^7$ on $\cC$ with associated $\ZZ_2^3$-grading and degree map $\deg_\cC$. It can be checked directly that we have a fine $\ZZ^3\times\ZZ_2^3$-grading on $\cV_\alb$, with homogeneous basis $\{ E_j^\sigma, \iota_j(x_i)^\sigma \}$, and determined by 
\begin{align*}
& \deg(E_1^+)=(-1,1,1,\bar0,\bar0,\bar0), & \quad \deg(\iota_1(x_i)^+)=(1,0,0,\deg_\cC(x_i)), \\
& \deg(E_2^+)=(1,-1,1,\bar0,\bar0,\bar0), & \quad \deg(\iota_2(x_i)^+)=(0,1,0,\deg_\cC(x_i)), \\
& \deg(E_3^+)=(1,1,-1,\bar0,\bar0,\bar0), & \quad \deg(\iota_3(x_i)^+)=(0,0,1,\deg_\cC(x_i)),
\end{align*}
and $\deg(x^+)+\deg(y^-)=0$ for any elements $x^+$, $y^-$ of the homogeneous basis such that $t(x^+,y^-)\neq0$,
and will be called the {\em Cayley-Dickson grading} on $\cV_\alb$.
\end{example}

\begin{example} \label{Ex3Albertpair}
Consider the $\ZZ_3^3$-grading on $\alb$ as a grading on $\cV_\alb$ and denote its degree map by $\deg_\alb$. Then, we can define a fine $\ZZ \times \ZZ_3^3$-grading on $\cV_\alb$ by $\deg(x^\sigma) = (\sigma1, \deg_\alb(x))$. (Note that, if we identify $\ZZ_3^3$ with a subgroup of $\ZZ \times \ZZ_3^3$, our new grading is just the $g$-shift of the $\ZZ_3^3$-grading on $\cV_\alb$ with $g=(1,\bar0,\bar0,\bar0)$.)
\end{example}

\begin{example} \label{ExCartanAlbertpair}
Using the triple product, one can check directly that we have a fine $\ZZ^7$-grading on $\cV_\alb$, where the degree map on $\alb^+$ is given by

\begin{center}
\begin{tabular}{|c|l|l|l|}
 \hline
 deg & $\iota_1(\cC)^+$ & $\iota_2(\cC)^+$ & $\iota_3(\cC)^+$ \\
 \hline
 $e_1$ & $(1,0,0,0,0,0,0)$ & $(0,0,0,0,0,1,0)$ & $(0,0,0,0,0,0,1)$ \\
 $e_2$ & $(0,1,0,0,0,0,0)$ & $(-1,-2,1,1,1,1,0)$ & $(2,1,-1,-1,-1,0,1)$ \\
 \hline
 $u_1$ & $(0,0,1,0,0,0,0)$ & $(0,-1,1,0,0,1,0)$ & $(1,1,0,-1,-1,0,1)$ \\
 $u_2$ & $(0,0,0,1,0,0,0)$ & $(0,-1,0,1,0,1,0)$ & $(1,1,-1,0,-1,0,1)$ \\
 $u_3$ & $(0,0,0,0,1,0,0)$ & $(0,-1,0,0,1,1,0)$ & $(1,1,-1,-1,0,0,1)$ \\
 \hline
 $v_1$ & $(1,1,-1,0,0,0,0)$ & $(-1,-1,0,1,1,1,0)$ & $(1,0,-1,0,0,0,1)$ \\
 $v_2$ & $(1,1,0,-1,0,0,0)$ & $(-1,-1,1,0,1,1,0)$ & $(1,0,0,-1,0,0,1)$ \\
 $v_3$ & $(1,1,0,0,-1,0,0)$ & $(-1,-1,1,1,0,1,0)$ & $(1,0,0,0,-1,0,1)$ \\
 \hline
\end{tabular}
\\
\begin{tabular}{|c|l|}
 \hline
 $E_1^+$ & $(0,-1,0,0,0,1,1)$ \\
 $E_2^+$ & $(2,2,-1,-1,-1,-1,1)$ \\
 $E_3^+$ & $(-1,-1,1,1,1,1,-1)$ \\
 \hline
\end{tabular} 
\end{center}

\noindent and $\deg(y^-) := -\deg(x^+)$ if $T(x^+,y^-)\neq0$, where $y \in \{E_i, \iota_i(z) \med i=1,2,3,\ z\in B_\cC\}$ and $B_\cC$ denotes the associated Cartan basis on $\cC$, and will be called the {\em Cartan grading} on $\cV_\alb$. (The proof is similar to the proof of Proposition~\ref{universalCartanbiCayleypair}, and using the fact that the trace is homogeneous.)
\end{example}

\begin{proposition}
The Cayley-Dickson grading on the Albert pair has universal group $\ZZ^3 \times \ZZ_2^3$.
\end{proposition}
\begin{proof} 
Consider a realization as $G$-grading, with $G$ an abelian group, of the Cayley-Dickson grading on $\cV_\alb$. Identify $\iota_1(\cC)\oplus\iota_2(\cC)$ with $\cB$, and notice that the restriction of the grading to these homogeneous components is the Cayley-Dickson grading on $\cV_\cB$. Call $g_i=\deg(\iota_1(x_i)^+)=-\deg(\iota_1(x_i)^-)$, $a=g_0=\deg(\iota_1(1)^+)$, $b=\deg(\iota_2(1)^+)$, $c=\deg(\iota_3(1)^+)$ and $a_i = g_i - g_0$. Using the same arguments of the proof of Proposition~\ref{universalCDbiCayleypair}, we deduce that $\deg(\iota_1(x_i)^+)=a+a_i$, $\deg(\iota_2(x_i)^+)=b+a_i$, $\deg(\iota_3(x_i)^+)=c+a_i$, and also that $\deg_\cC(x_i) := a_i$ defines a group grading which is a coarsening of the $\ZZ_2^3$-grading on $\cC$. Therefore, there is an epimorphism $\ZZ^3 \times \ZZ_2^3 \to G$ that sends $(1,0,0,\bar0,\bar0,\bar0) \mapsto a$, $(0,1,0,\bar0,\bar0,\bar0) \mapsto b$, $(0,0,1,\bar0,\bar0,\bar0) \mapsto c$ and restricts to an epimorphism $0 \times \ZZ_2^3 \to \langle a_i \rangle$. We conclude that $\ZZ^3 \times \ZZ_2^3$ is the universal group.
\end{proof}

\begin{proposition} \label{propCDalb2}
The fine $\ZZ \times\ZZ_2^3$-grading on $\alb$ in \eqref{gradAlbertZxZ23}, considered as a grading on $\cV_\alb$, admits a unique fine refinement, up to relabeling, which has universal group $\ZZ^3 \times \ZZ_2^3$ and is equivalent to the Cayley-Dickson grading.
\end{proposition}
\begin{proof} 
With the notation in \eqref{notGradAlbertZxZ23}, since $U_{\nu_+(1)}(S_-)=16E$ and $\{\nu_+(1),E,\nu_-(1)\}=8\widetilde{E}$, it follows that $E^\sigma$ and $\widetilde{E}^\sigma$ are homogeneous for any refinement of this grading. Set $a=\frac{-1}{\sqrt{2}}1$, $\lambda=\frac{1}{\sqrt{2}}$. It suffices to prove that the automorphism $\varphi=c_{\bi,1,\bi} ~ \phi_1(a,\lambda)$ of the Albert pair (see Proposition \ref{automfiofA} and \eqref{eq:iotas}) is an equivalence between the Cayley-Dickson grading and any fine refinement of the $\ZZ \times\ZZ_2^3$-grading. A straightforward computation shows that:
\begin{equation} \begin{aligned} \label{equivalenceCD}
\varphi^+: \quad
& E_1\mapsto E_1, \quad E_2\mapsto\frac{1}{2}S_+, \quad E_3\mapsto\frac{1}{2}S_-, \\
& \iota_1(1)\mapsto2\widetilde{E}, \quad \iota_1(a)\mapsto\bi\iota_1(a)=\nu(a), \\
& \iota_2(x)\mapsto\frac{1}{\sqrt{2}}\nu_-(x), \quad \iota_3(\bar x)\mapsto\frac{1}{\sqrt{2}}\nu_+(x), \\
\varphi^-: \quad
& E_1\mapsto E_1, \quad E_2\mapsto\frac{1}{2}S_-, \quad E_3\mapsto\frac{1}{2}S_+, \\
& \iota_1(1)\mapsto2\widetilde{E}, \quad \iota_1(a)\mapsto-\bi\iota_1(a)=-\nu(a), \\
& \iota_2(x)\mapsto\frac{1}{\sqrt{2}}\nu_+(x), \quad \iota_3(\bar x)\mapsto\frac{1}{\sqrt{2}}\nu_-(x)
\end{aligned} \end{equation}
so that $\varphi$ takes the homogeneous components of the Cayley-Dickson grading to homogeneous components in any refinement of the $\ZZ\times \ZZ_2^3$-grading, as required.
\end{proof}

\begin{proposition} \label{gradrank3}
The fine $\ZZ_3^3$-grading on $\alb$, considered as a grading on $\cV_\alb$, admits a unique fine refinement, up to relabeling, which has universal group $\ZZ \times\ZZ_3^3$. 
\end{proposition}
\begin{proof} 
This is a consequence of Proposition~\ref{regulargradspairs}. The degree map can be given by $\deg(x^\sigma)=(\sigma1,\deg_\alb(x))$, where $\deg_\alb(x)$ denotes the degree of the $\ZZ_3^3$-grading on $\alb$.
\end{proof}

\begin{proposition} \label{propCartanAlb}
The fine $\ZZ^4$-grading on $\alb$, considered as a grading on $\cV_\alb$, admits a unique fine refinement, up to relabeling, which has universal group $\ZZ^7$ and is the Cartan grading on $\cV_\alb$.
\end{proposition}
\begin{proof}
The proof is arduous but straightforward, so we do not give all the details. Note that, since $\{ \iota_{i+1}(e_1)^\sigma, \iota_i(e_1)^{-\sigma}, \iota_{i+2}(e_1)^\sigma \} = 8E_i^\sigma$, the elements $E_i^\sigma$ must be homogeneous, so the fine refinement is unique. Consider a realization as $G$-grading, with $G$ an abelian group, of the Cartan grading on $\cV_\alb$. One can check directly that the degrees of the elements $\iota_1(e_1)^+$, $\iota_1(e_2)^+$, $\iota_1(u_1)^+$, $\iota_1(u_2)^+$, $\iota_1(u_3)^+$, $\iota_2(e_1)^+$, $\iota_3(e_1)^+$ generate $G$. We conclude that the $G$-grading is induced from the Cartan grading by some epimorphism $\ZZ^7 \to G$ that is the identity on the support (and sends the canonical basis of $\ZZ^7$ to the degrees of the mentioned elements), and so $\ZZ^7$ is the universal group of the Cartan grading on $\cV_\alb$. 
\end{proof}

\subsection{Construction of fine gradings on the bi-Cayley triple system}

\begin{example} \label{ExCDbiCayleytriple}
Consider a Cayley-Dickson basis $\{x_i\}_{i=0}^7$ of $\cC$ and denote by $\deg_\cC$ the degree map of the associated $\ZZ_2^3$-grading. Then $\{(x_i,0), (0,x_i) \}_{i=0}^7$ will be called a {\em nonisotropic Cayley-Dickson basis} of $\cT_\cB$. It can be checked that we have a fine $\ZZ_2^5$-grading on $\cT_\cB$ given by $\deg(x_i,0)=(\bar1,\bar0,\deg_\cC(x_i))$ and $\deg(0,x_i)=(\bar0,\bar1,\deg_\cC(x_i))$, and will be called the {\em nonisotropic Cayley-Dickson grading} on $\cT_\cB$. (The isotropy is relative to the quadratic form $q=n\perp n$ of $\cB$.)
\end{example}

\begin{example} \label{ExisotropicCDbiCayleytriple}
Let $\{x_i\}_{i=0}^7$ be as above. Fix $\bi\in\FF$ with $\bi^2=-1$. Then, $\{(x_i,\pm\bi\bar x_i) \}_{i=0}^7$ will be called an {\em isotropic Cayley-Dickson basis} of $\cT_\cB$. It can be checked that we have a fine $\ZZ \times\ZZ_2^3$-grading on $\cT_\cB$ given by $\deg(x_i,\pm\bi\bar x_i)=(\pm1,\deg_\cC(x_i))$, and will be called the {\em isotropic Cayley-Dickson grading} on $\cT_\cB$. (The isotropy is relative to the quadratic form $q=n\perp n$ of $\cB$.) 
\end{example}

\begin{example} \label{ExCartanbiCayleytriple}
Let $\{z_i\}_{i=1}^8$ be a Cartan basis of $\cC$. Then, we will say that $\{(z_i,0), (0,z_i) \}_{i=1}^8$ is a {\em Cartan basis} of $\cT_\cB$. It can be checked that we have a fine $\ZZ^4$-grading where the degree map is given by the following table:

\begin{center}
\begin{tabular}{|c|l|l|}
 \hline
 deg & $\cC_1$ & $\cC_2$ \\
 \hline
 $e_1$ & $(0,0,1,0)$ & $(0,0,0,-1)$ \\
 $e_2$ & $(0,0,-1,0)$ & $(0,0,0,1)$ \\
 \hline
 $u_1$ & $(1,0,-1,0)$ & $(1,0,0,-1)$ \\
 $u_2$ & $(0,1,-1,0)$ & $(0,1,0,-1)$ \\
 $u_3$ & $(-1,-1,1,2)$ & $(-1,-1,2,1)$ \\
 \hline
 $v_1$ & $(-1,0,1,0)$ & $(-1,0,0,1)$ \\
 $v_2$ & $(0,-1,1,0)$ & $(0,-1,0,1)$ \\
 $v_3$ & $(1,1,-1,-2)$ & $(1,1,-2,-1)$ \\
 \hline
\end{tabular}
\end{center}

\noindent and will be called the {\em Cartan grading} on $\cT_\cB$. (Notice that the projection $\ZZ^4 \to \ZZ^2$ of the degree on the first two coordinates induces the Cartan $\ZZ^2$-grading on $\cC$, so it suffices to show that the last two coordinates behave well with respect to the product, and this is easily checked).
\end{example}

Note that the homogeneous elements of the nonisotropic Cayley-Dickson grading on $\cT_\cB$ are in the orbits $\cO_2(\lambda)$ with $\lambda\in\FF^\times$, the ones of the isotropic Cayley-Dickson grading on $\cT_\cB$ are in the orbit $\cO_2(0)$, and the ones of the Cartan grading on $\cT_\cB$ are in the orbit $\cO_1$ (see Lemma~\ref{orbitsbiCayleytriple}); hence these three gradings cannot be equivalent. This can also be seen from their universal groups, as follows:

\begin{proposition} 
The nonisotropic Cayley-Dickson grading on the bi-Cayley triple system has universal group $\ZZ_2^5$. 
\end{proposition}
\begin{proof} Let $\{x_i\}_{i=0}^7$ be a Cayley-Dickson basis of $\cC$ with $x_0=1$. Consider a realization of the nonisotropic Cayley-Dickson grading on $\cT_\cB$ as $G$-grading for some abelian group $G$. Since the trace $t$ is homogeneous and $t((x_i,0),(x_i,0))\neq0\neq t((0,x_i),(0,x_i))$, it follows that all the elements of $G$ have order $\leq 2$. Call $a=\deg(1,0)$, $b=\deg(0,1)$, $g_i=\deg(x_i,0)$ and $a_i=a + g_i$ for $0\leq i\leq 7$. Note that we have $\deg(x_i,0)=g_i=a_i+a$. Since $\{(1,0),(x_i,0),(0,1)\}=(0,x_i)$ for each $i$, we have $\deg(0,x_i)=a+g_i+b = a_i+b$. If $i\neq j$ with $i,j\neq 0$, we have $\{(x_i,0),(0,1),(0,x_j)\}=(x_ix_j,0)$; hence $\deg(x_ix_j,0)=g_i+b+(a_j+b)=(a_i+a_j)+a$, and it follows that $\deg_\cC(x_i):=a_i$ defines a coarsening of the $\ZZ_2^3$-grading on $\cC$. It is clear that the $G$-grading is induced from the $\ZZ_2^5$-grading by an epimorphism $\ZZ_2^5 \to G$ that sends $(\bar1,\bar0,\bar0,\bar0,\bar0) \mapsto a$, $(\bar0,\bar1,\bar0,\bar0,\bar0) \mapsto b$ and restricts to some epimorphism $0 \times \ZZ_2^3 \to \langle a_i \rangle$. We can conclude that $\ZZ_2^5$ is (isomorphic to) the universal group of the nonisotropic Cayley-Dickson grading on $\cT_\cB$.
\end{proof}

\begin{proposition} 
The isotropic Cayley-Dickson grading on the bi-Cayley triple system has universal group $\ZZ \times\ZZ_2^3$. 
\end{proposition}
\begin{proof} Let $\{x_i\}_{i=0}^7$ be a Cayley-Dickson basis of $\cC$ with $x_0=1$. Consider a realization of the isotropic Cayley-Dickson grading on $\cT_\cB$ as $G$-grading for some abelian group $G$. Call $g_i=\deg(x_i,\bi\bar x_i)$ and $a_i=g_i-g_0$. Since the trace $t$ is homogeneous and $t((x_i,\bi\bar x_i),(x_i,-\bi\bar x_i))\neq0$, it follows that $\deg(x_i,-\bi\bar x_i)=-g_i$. For each $i\neq0$ we have $Q_{(1,\bi1)}(x_i,-\bi\bar x_i)=-2(x_i,\bi\bar x_i)$, so $2g_0=2g_i$. Thus, $a_i$ has order $\leq2$. Moreover, $\deg(x_i,\bi\bar x_i)=a_i+g_0$, $\deg(x_i,-\bi\bar x_i)=a_i-g_0$. But also, for each $i\neq j$ with $i,j\neq0$, we have $\bar x_i=-x_i$, $\bar x_j=-x_j$, $x_ix_j=-x_jx_i$, from where we get $\{(1,\bi1),(x_i,\bi\bar x_i),(x_j,-\bi\bar x_j)\}=-2(x_ix_j,\bi\overline{x_ix_j})$, and taking degrees we obtain $\deg(x_ix_j,\bi\overline{x_ix_j})=(a_i+a_j)+g_0$. In consequence, $\deg_\cC(x_i):=a_i$ defines a coarsening of the $\ZZ_2^3$-grading on $\cC$. Therefore, the $G$-grading is induced from the $\ZZ \times\ZZ_2^3$-grading by an epimorphism $\ZZ \times\ZZ_2^3 \to G$ that sends $(1,\bar0,\bar0,\bar0) \mapsto g_0$ and restricts to some epimorphism $0 \times \ZZ_2^3 \to \langle a_i \rangle$. We can conclude that $\ZZ \times\ZZ_2^3$ is (isomorphic to) the universal group of the isotropic Cayley-Dickson grading on $\cT_\cB$.
\end{proof}

\begin{proposition}
The Cartan grading on the bi-Cayley triple system has universal group $\ZZ^4$. 
\end{proposition}
\begin{proof} Let $\{e_i,u_j,v_j|i=1,2; j=1,2,3\}$ be a Cartan basis of $\cC$. Consider a realization of the Cartan grading on $\cT_\cB$ as $G$-grading for some abelian group $G$. Call $a=\deg(e_1,0)$, $b=\deg(0,e_2)$, and $h_i=\deg(u_i,0)$ for $i=1,2$. We claim that $\{a,b,g_1,g_2\}$ generate $G$. Indeed, since the trace is homogeneous, we get $\deg(e_2,0)=-a$, $\deg(0,e_1)=-b$, and $\deg(v_i,0)=-h_i$ for $i=1,2$. Since $\{(v_1,0),(v_2,0),(0,e_2)\}=(0,u_3)$, we deduce that $\deg(0,u_3) = -h_1-h_2+b = -\deg(0,v_3)$. Also, from $\{(0,e_2),(0,u_3),(e_2,0)\}=(u_3,0)$, we obtain $\deg(u_3,0) = -h_1-h_2-a+2b = -\deg(v_3,0)$. We have proved the claim. It is clear that the $G$-grading is induced from the $\ZZ^4$-grading by an epimorphism $\ZZ^4 \to G$ that sends $(1,0,0,0) \mapsto g_1$, $(0,1,0,0) \mapsto g_2$, $(0,0,1,0) \mapsto a$ and $(0,0,0,1) \mapsto b$, and we can conclude that $\ZZ^4$ is (isomorphic to) the universal group of the Cartan grading on $\cT_\cB$.
\end{proof}

\section{Classification of fine gradings}

\subsection{Classification of fine gradings on the bi-Cayley pair}

Given a grading on a semisimple Jordan pair, by Remark~\ref{remarkgrads}, any homogeneous element can be completed to a maximal orthogonal system of homogeneous idempotents. In the case of the bi-Cayley pair, since the capacity is $2$, it will consist either of two idempotents of rank $1$, or one idempotent of rank $2$. We will cover these possibilities with the following Lemmas.

\begin{lemma} \label{lemmacartan}
Let $\Gamma$ be a fine grading on the bi-Cayley pair such that there is some homogeneous element of rank $1$. Then $\Gamma$ is equivalent to the Cartan grading (Example~\ref{CartanbasisbiCayleypair}).
\end{lemma}
\begin{proof} 
Write $\cV=\cV_\cB$ for short. First, we complete the homogeneous element to a set consisting of two homogeneous orthogonal idempotents of rank $1$. By Theorem~\ref{orbitframes}, we can assume without loss of generality that the homogeneous orthogonal idempotents are $(c_1^+,c_2^-)$ and $(c_2^+,c_1^-)$, where $c_i=(e_i,0)\in\cB$ and $e_i$ are nontrivial orthogonal idempotents of $\cC$ with $e_1+e_2=1$. We will consider the Peirce decomposition $\cC=\FF e_1\oplus \FF e_2\oplus U\oplus V$ associated to the idempotents $e_1$ and $e_2$. Since the generic trace is homogeneous,  
\begin{equation} \label{gradedmap} \begin{aligned}
& f(x^\sigma, y^{-\sigma},z^\sigma) := t(x,y)z+t(z,y)x-\{x,y,z\} \\ & = (n(x_1,z_1)y_1+\bar y_2(x_2z_1+z_2x_1),n(x_2,z_2)y_2+(x_2z_1+z_2x_1)\bar y_1)
\end{aligned} \end{equation}
is a homogeneous map too. By Remark~\ref{homogeneouskernel}, $K^\sigma=\ker(t_{c_1}) \cap \ker(t_{c_2})$ is a graded subspace of $\cV^\sigma$. For each homogeneous $z^+\in K^+$, we have $n(e_1,z_1)=t(c_1,z)=0$ and $f(c_1^+,c_2^-,z^+)=(0,z_2e_1)^+$ is homogeneous. Note that $(0\oplus\cC)^\sigma \subseteq K^\sigma$, so there are homogeneous elements $\{(x_i, y_i)\}_{i=1}^8$ of $K^+$ such that $\{y_i\}_{i=1}^8$ is a basis of $\cC$. Thus, the subspace $(0\oplus\cC e_1)^+ = \sum_{i=1}^8 (0 \oplus \FF y_i e_1)^+$ is graded. Similarly, $(0\oplus \cC e_1)^\sigma$ and $(0\oplus \cC e_2)^\sigma$ are graded for $\sigma=\pm$. Hence, $\cC_2^\sigma=(0\oplus\cC)^\sigma$ is graded, and in consequence $\cC_1^\sigma = (\cC\oplus0)^\sigma  = \bigcap_{x\in\cC_2^{-\sigma}} \ker(t_x)$ is graded too. 

We claim that the homogeneous elements of $\cC_i^+$ and $\cC_i^-$ coincide. Indeed, take homogeneous elements $x^+=(x_1,0)$ and $z^+=(z_1,0)$ of $\cC_1^+$ such that $n(x_1,z_1)=1$. Then, for any homogeneous element $y^-=(y_1,0)$ of $\cC_1^-$, $f(x^+,y^-,z^+)=y^+$ is homogeneous too, and hence the homogeneous elements of $\cC_1^+$ and $\cC_1^-$ coincide; and similarly this is true for $\cC_2^+$ and $\cC_2^-$. Since $\Gamma$ is fine, the supports $\supp \cC_i^\sigma$ are disjoint (because otherwise we could obtain a refinement of $\Gamma$ combining it with the $\ZZ^2$-grading: $\cV_{(\sigma1,0)} = \cC_1^\sigma$, $\cV_{(0,\sigma1)} = \cC_2^\sigma$. 

From now on, we can omit the superscript $\sigma$, because the homogeneous components of $\cV^+$ coincide with those of $\cV^-$. The rest of this proof will be used in the proof of Lemma~\ref{lemmacartan2}.

Recall that $(0\oplus \cC e_i)$ are graded subspaces, where $\cC e_1$ and $\cC e_2$ are isotropic subspaces of $\cC$. Since the trace is homogeneous, there is a homogeneous basis $\{(0,x_i),(0,y_i)\}_{i=1}^4$ of $\cC_2$ such that $\{x_i,y_i\}_{i=1}^4$ is a basis of $\cC$ consisting of four orthogonal hyperbolic pairs, that is, such that $n(x_i,y_j)=\delta_{ij}$, $n(x_i,x_j)=0=n(y_i,y_j)$. It is not hard to see that there is an element of $\Ort^+(\cC,n)$ that sends the elements $\{x_i,y_i\}_{i=1}^4$ to a Cartan basis $\{e_i,u_j,v_j \med i=1,2; j=1,2,3\}$ of $\cC$, and by Lemma~\ref{lemmaAlberto}, that can be done in $\cC_2$ with an automorphism given by a related triple (as in Remark~\ref{remarkrelatedtriples}). Hence, we can assume that we have a homogeneous Cartan basis of $\cC_2$ (and the subspace $\cC_1$ is still graded). Then we have the following graded subspaces:
\begin{align*}
& f((0,e_1),(0,e_2),\cC\oplus0)=\bar e_2(e_1\cC)\oplus0=(\FF e_1+U)\oplus0, \\
& f((0,e_2),(0,e_1),\cC\oplus0)=\bar e_1(e_2\cC)\oplus0=(\FF e_2+V)\oplus0, \\
& f((0,u_1),(0,e_1),(\FF e_1+U)\oplus0)=(\bar e_1(u_1(\FF e_1+U)))\oplus0 \\
& \hspace{5.8cm} =(\FF v_2+\FF v_3)\oplus0, \\
& f((0,v_2),(0,e_2),(\FF v_2+\FF v_3)\oplus0)=(\bar e_2(v_2(\FF v_2+\FF v_3)))\oplus0=(\FF u_1)\oplus0,
\end{align*}
so $(u_1,0)$ is homogeneous, and similarly $(u_i,0)$, $(v_i,0)$ are homogeneous for $i=1,2,3$. Furthermore,
$f((0,u_1),(0,e_2),(v_1,0))=(\bar e_2(u_1v_1),0)=(-e_1,0)$,
so $(e_1,0)$ and $(e_2,0)$ are homogeneous. Since $\Gamma$ is fine, we conclude that $\Gamma$ is the Cartan grading.
\end{proof}


\begin{lemma} \label{lemmacd} 
Let $\Gamma$ be a fine grading on the bi-Cayley pair such that the nonzero homogeneous elements have rank $2$. Then $\Gamma$ is equivalent to the Cayley-Dickson grading (Example~\ref{CDbasisbiCayleypair}).
\end{lemma}
\begin{proof} 
Write for short $\cV=\cV_\cB$. Take a homogeneous element and complete it to a homogeneous idempotent of rank $2$. By Remark~\ref{remarkorbitidempotents}, we can assume without loss of generality that our homogeneous idempotent is $c_1=((1,0)^+,(1,0)^-)$. 
The subpaces $\cC_1^\sigma=\im Q_{(1,0)^\sigma}$ and $\cC_2^\sigma=\ker Q_{(1,0)^{-\sigma}}$ are graded. With the same arguments given in the proof of Lemma~\ref{lemmacartan}, we can deduce that the supports $\supp \cC_i^\sigma$ are disjoint and the homogeneous components of $\cV^+$ coincide with those of $\cV^-$. From now on, we can omit the superscript $\sigma$. The arguments of the rest of this proof will be used in the proof of Lemma~\ref{lemmacd2}.

We can take a homogeneous element $(0,x)$ with $n(x)=1$ (otherwise, $n(x)=0$ and $(0,x)$ would have rank $1$, a contradiction). By Lemma~\ref{lemmaAlberto} and Remark~\ref{remarkrelatedtriples}, there is an automorphism of $\cV$, given by a related triple, that maps $(1,0)\mapsto(1,0)$, $(0,x)\mapsto(0,1)$. In consequence, we can assume that $(1,0)$ and $(0,1)$ are homogeneous. Recall that the map $f$ in Equation~\eqref{gradedmap} is homogeneous. From $f((x,0),(1,0),(0,1))=(0,x)$, it follows that $(x,0)$ is homogeneous if and only if $(0,x)$ is homogeneous, i.e., the homogeneous components coincide in both $\cC_1$ and $\cC_2$. We can take a homogeneous basis $B=\{(x_i,0),(0,x_i)\}_{i=1}^8$ where $x_1=1$ and $n(x_i)=1$ for all $i$. Since the trace is homogeneous and the homogeneous components are $1$-dimensional (by Theorem~\ref{dim1}), we also have $n(x_i,x_j)=t((x_i,0),(x_j,0))=0$, i.e., $\{x_i\}$ is an orthonormal basis of $\cC$. Using the map $f$, it is easy to deduce that $(x_ix_j,0)$ and $(0,x_ix_j)$ are homogeneous for any $1\leq i,j\leq 8$, so actually we can assume, without loss of generality, that 
$B$ is a Cayley-Dickson basis of $\cV$. Since $\Gamma$ is fine, we conclude that $\Gamma$ is the Cayley-Dickson grading.
\end{proof}

\begin{theorem} \label{gradsbiCayleypair}
Let $\Gamma$ be a fine grading on the bi-Cayley pair. Then, $\Gamma$ is equivalent to either the Cartan grading (Example~\ref{CartanbasisbiCayleypair}), with universal group $\ZZ^6$, or the Cayley-Dickson grading (Example~\ref{CDbasisbiCayleypair}), with universal group $\ZZ^2\times\ZZ_2^3$.
\end{theorem}
\begin{proof} Consequence of Lemmas~\ref{lemmacartan} and \ref{lemmacd} (and Propositions~\ref{universalCDbiCayleypair} and \ref{universalCartanbiCayleypair}), since they cover all the possibilities.
\end{proof}

\subsection{Classification of fine gradings on the Albert pair}

Given a grading on a semisimple Jordan pair, by Remark~\ref{remarkgrads}, any homogeneous element can be completed to a maximal orthogonal system of homogeneous idempotents. In the case of the Albert pair, since the capacity is $3$, it will consist either of three idempotents of rank $1$, or one idempotent of rank $2$ and another of rank $1$, or one of rank $3$. We will cover these possibilities with the following Lemmas.

\begin{lemma} \label{gradrk1} 
Let $\Gamma$ be a fine grading on $\cV_\alb$ such that all nonzero homogeneous idempotents have rank 1. Then, $\Gamma$ is equivalent to the Cartan grading (Example~\ref{ExCartanAlbertpair}).
\end{lemma}
\begin{proof}
We can take a set of three orthogonal homogeneous idempotents $F=\{e_1,e_2,e_3\}$, so $F$ is a frame, and up to automorphism (by Theorem~\ref{orbitframes} or Remark~\ref{remarkorbitidempotents}), we can assume that $e_i=(E_i^+,E_i^-)$. Hence, for any permutation $\{i,j,k\}=\{1,2,3\}$, the associated Peirce subspaces, 
$$(\cV_\cB)^\sigma_{jk} = \{ x\in\alb \med D(e_j^\sigma, e_j^{-\sigma})x = x =  D(e_k^\sigma, e_k^{-\sigma})x \} = \iota_i(\cC)^\sigma, $$are graded. It is clear that $\Gamma$ restricts to a grading $\Gamma_\cB$ on the bi-Cayley pair $\cV_\cB:=(\cB,\cB)$, where $\cB:=\iota_1(\cC)\oplus\iota_2(\cC)$. By \cite{S87}, we know that each automorphism of the bi-Cayley pair has a unique extension to the Albert pair that fixes $E_3^+$ and $E_3^-$, and hence we can identify $\Aut\cV_\cB$ with the stabilizer of $e_3$ in $\Aut\cV_\alb$. The nonzero homogeneous elements of $\Gamma_\cB$ must have rank one, and therefore $\Gamma_\cB$ is equivalent to the Cartan $\ZZ^6$-grading. We can apply an automorphism of $\cV_\cB$ extended to $\Aut\cV_\alb$ and assume that we have the Cartan basis on $\cV_\cB$ as in Example~\ref{CartanbasisbiCayleypair}. Then, it is easy to check that we have the homogeneous basis of the Cartan grading on the Albert pair, and consequently, $\Gamma$ is the Cartan $\ZZ^7$-grading on the Albert pair. 
\end{proof}

\begin{lemma} \label{gradrk2}
Let $\Gamma$ be a fine grading on $\cV_\alb$ such that there are two orthogonal homogeneous idempotents, one of rank $1$ and the other of rank $2$. Then, $\Gamma$ is equivalent to the Cayley-Dickson $\ZZ^3\times\ZZ_2^3$-grading (Example~\ref{ExCDAlbertpair}).
\end{lemma}
\begin{proof}
Denote by $e_1$ and $e_2$ the orthogonal homogeneous idempotents, with $\rk(e_1)=1$ and $\rk(e_2)=2$.
By Remark~\ref{remarkorbitidempotents} we can assume that $e_1^\sigma = E:=E_1$ and $e_2^\sigma = \widetilde{E}:=E_2+E_3$. The Peirce subspace $\cB^\sigma := \{x\in\alb \med D(e_1^\sigma, e_1^{-\sigma})x = x = D(e_2^\sigma, e_2^{-\sigma})x \} = \iota_2(\cC)\oplus\iota_3(\cC)$ is graded, and we can identify it with the bi-Cayley pair $\cV_\cB$. The grading $\Gamma_\cB$ induced on $\cV_\cB$ must be equivalent to the Cayley-Dickson $\ZZ^2\times\ZZ_2^3$-grading (because the Cartan grading on $\cV_\cB$ can only be extended to the Cartan grading on $\cV_\alb$, which does not have homogeneous elements of rank $2$). By the same arguments used in the proof of Lemma~\ref{gradrk1}, we can apply an automorphism of the bi-Cayley pair extended to $\cV_\alb$ to assume that we have a homogeneous basis of $\cV_\alb$ as in Proposition~\ref{propCDalb2} (the elements of $\cV_\cB$ are of the form $\nu_{\pm}(x)$). We conclude that $\Gamma$ is equivalent to the Cayley-Dickson $\ZZ^3\times\ZZ_2^3$-grading.
\end{proof}

\begin{lemma} \label{gradrk3}
Let $\Gamma$ be a fine grading on $\cV_\alb$ with some homogeneous idempotent of rank $3$. Then, $\chr\FF\neq3$ and $\Gamma$ is equivalent to the $\ZZ\times\ZZ_3^3$-grading (Example~\ref{Ex3Albertpair}).
\end{lemma}
\begin{proof}
Let $e$ be a homogeneous idempotent of rank $3$. By Remark~\ref{remarkorbitidempotents}, we can assume, up to automorphism, that $e=(1^+,1^-)$, where $1$ is the identity of $\alb$. By Theorem~\ref{dim1}, the homogeneous components are $1$-dimensional, and on the other hand the trace is homogeneous and nondegenerate, so the restriction of the trace to the subpair $(\FF1^+,\FF1^-)$ must be nondegenerate, which forces $\chr\FF\neq3$.

By Proposition~\ref{regulargradspairs}, if $g=-\deg(1^+)$ and $\deg_g$ is the degree map of the shift $\Gamma^{[g]}$ of $\Gamma$, then $\deg_g(x^+)=\deg_g(x^-)$ for any homogeneous element $x\in\alb$, and $\deg_g$ restricts to a grading $\Gamma_\alb$ on $\alb$. Since $\Gamma$ is fine, its homogeneous components are $1$-dimensional by Proposition~\ref{dim1}, and this is also true for $\Gamma_\alb$. Therefore, $\Gamma_\alb$ must be, up to equivalence, the $\ZZ_3^3$-grading, because this is the only grading on $\alb$ with $1$-dimensional homogeneous components. Finally, since $\Gamma$ is a shift of the $\ZZ_3^3$-grading $(\Gamma_\alb,\Gamma_\alb)$ on $\cV_\alb$, this forces $\Gamma$ to be the $\ZZ\times\ZZ_3^3$-grading (see Proposition~\ref{regulargradspairs}).
\end{proof}

\begin{theorem} 
The fine gradings on the Albert pair are, up to equivalence, the Cartan $\ZZ^7$-grading (Example~\ref{ExCartanAlbertpair}), the Cayley-Dickson $\ZZ^3\times\ZZ_2^3$-grading (Example~\ref{ExCDAlbertpair}), and the $\ZZ\times\ZZ_3^3$-grading (Example~\ref{Ex3Albertpair}). The $\ZZ\times\ZZ_3^3$-grading only occurs if $\chr\FF\neq3$.
\end{theorem}
\begin{proof} 
This result follows since Lemmas~\ref{gradrk1}, \ref{gradrk2} and \ref{gradrk3} cover all possible cases.
\end{proof}

\subsection{Classification of fine gradings on the bi-Cayley triple system}

Recall that we defined the norm of $\cB$ as the quadratic form $q\colon\cB\to\FF$, $q(x,y):=n(x)+n(y)$. Also, we already know that $\Aut\cT_\cB\leq\Ort(\cB,q)$, and the nonzero isotropic elements of $\cB$ are exactly the ones contained in the orbits $\cO_1$ and $\cO_2(0)$ of $\cT_\cB$.

\begin{lemma} \label{lemmacartan2} 
Let $\Gamma$ be a fine grading on $\cT_\cB$ with some homogeneous element in the orbit $\cO_1$. Then $\Gamma$ is, up to equivalence, the Cartan grading on $\cT_\cB$ (Example~\ref{ExCartanbiCayleytriple}).
\end{lemma}

\begin{proof} 
Let $x$ be homogeneous in the orbit $\cO_1$. We claim that we can take a homogeneous element $y$ in the orbit $\cO_1$ and such that $t(x,y)=1$. Indeed,
it suffices to consider the grading $(\Gamma,\Gamma)$ on the bi-Cayley pair and complete the element $x$ to a homogeneous idempotent $(x,y)$ of the pair (recall that we have $\rk(e^+)=\rk(e^-)$ for any idempotent). Since the trace form is invariant for automorphisms of the pair and all idempotents of rank $1$ of the pair are in the same orbit, it follows that $t(x,y)=1$ (it suffices to check this for an idempotent of rank $1$ of the pair).

Up to automorphism, by Lemma~\ref{orbitsbiCayleytriple}, we can assume that $x=(e_1,0)$ with $e_1$ a nontrivial idempotent of $\cC$. Consider, as usual, the Peirce decomposition of $\cC$ relative to the idempotents $e_1$ and $e_2:=\bar e_1$. By Lemma~\ref{orbitsbiCayleytriple}, we know that $n(y_1)=n(y_2)=0$ and $y_2y_1=0$. Since $n(y_1)=0$ and $n(e_1,y_1)=t(x,y)=1$, there is an automorphism given by a related triple (see Lemma~\ref{lemmaAlberto}) that sends $(e_1,0)\mapsto(e_1,0)$, $y\mapsto(e_2,y_2)$. Thus, we can also assume that $y=(e_2,y_2)$. Since $y_2y_1=0$, it follows that $y_2=\lambda e_1+v$ with $\lambda\in \FF$, $v\in V$. Take $a=-y_2$ and $\mu=1$ (so $n(a)+\mu^2=1$). We have $\varphi_{a,\mu}(e_1,0)=(e_1,0)$ and $\varphi_{a,\mu}(e_2,y_2)=(e_2,0)$. Therefore, we can assume that $(e_i,0)$ are homogeneous for $i=1,2$.

Since the trace is homogeneous, $f(x,y,z):=t(x,y)z+t(z,y)x-\{x,y,z\}$ is a homogeneous map and $\ker(t_x)$ is graded. For any homogeneous $z\in \ker(t_x)$, we have $n(e_1,z_1)=t(x,z)=0$, and so $f((e_1,0),(e_2,0),z)=(n(e_1,z_1)e_2,z_2e_1)=(0,z_2e_1)$ is homogeneous.
In consequence $(0\oplus\cC e_1)$ is graded. Similarly, $(0\oplus\cC e_2)$ is graded, and hence $\cC_2$ is graded. Since the trace is homogeneous, the subspace orthogonal (for the trace) to $\cC_2$, which is $\cC_1$, is graded too. 
We can conclude the proof with the same arguments given in the end of the proof of Lemma~\ref{lemmacartan}.
\end{proof}

\begin{lemma} \label{lemmacd2}
Let $\Gamma$ be a fine grading on $\cT_\cB$ with some homogeneous element in some orbit $\cO_2(\lambda)$ with $\lambda\neq0$. Then $\Gamma$ is, up to equivalence, the nonisotropic Cayley-Dickson grading (Example~\ref{ExCDbiCayleytriple}).
\end{lemma}
\begin{proof}
It is clear that $\Gamma$ cannot be equivalent to the Cartan grading, because there is a homogeneous element $x$ in the orbit $\cO_2(\lambda)$ with $\lambda\neq0$ and in the Cartan grading all the homogeneous elements have rank $1$. In particular, by Lemma~\ref{lemmacartan2}, all nonzero homogeneous elements of $\Gamma$ must have rank $2$. Up to automorphism and up to scalars, we can assume by Lemma~\ref{orbitsbiCayleytriple} that $x=(1,0)$. Then, $\cC_1 = \im Q_x$ and $\cC_2 = \ker Q_x$ are graded subspaces, and we can conclude with the same arguments given in the end of the proof of Lemma~\ref{lemmacd}.
\end{proof}

\begin{lemma} \label{lemmacdisotropic} 
Let $\Gamma$ be a fine grading on $\cT_\cB$ where all the nonzero homogeneous elements are in the orbit $\cO_2(0)$. Then $\Gamma$ is, up to equivalence, the isotropic Cayley-Dickson grading (Example~\ref{ExisotropicCDbiCayleytriple}).
\end{lemma}
\begin{proof} 
Take a nonzero homogeneous element $x\in\cB$. Since $x \in \cO_2(0)$, up to automorphism we can assume that $x=(1,\bi1)$ for some $\bi\in\FF$ with $\bi^2=-1$. Then, $W := \im Q_x = \ker Q_x = \{(z,\bi\bar z) \med z\in\cC\}$ is a graded subspace. Let $\cC_0$ denote the traceless octonions and set $V:=\{(z_0,\bi \bar z_0) \med z_0\in\cC_0\}$, $W':=\{(z, -\bi\bar z) \med z\in\cC\}$, $V':=\{(z_0, -\bi \bar z_0) \med z_0\in\cC_0\}$, $x':=(1,-\bi1)$. Consider the map $t_x : \cB \to \cB$, $z \mapsto t(x,z)$. Since the trace is homogeneous, $\ker t_x = W \oplus V' = \FF x \oplus V \oplus V'$ is a graded subspace. Hence $Q_x(\ker t_x) = V$ is graded too. (Note that $V$ and $V'$ are isotropic subspaces which are paired relative to the trace form, and $x$ is paired with $x'$ too. But in general, $\FF x'$, $W'$ and $V'$ are not graded subspaces.) The subspace $V^\perp = \FF x' \oplus W$ is graded because the trace is homogeneous, so we can take a homogeneous element $\widetilde{x} = x' + \lambda x + v$ with $\lambda \in\FF$, $v\in V$. Since $\widetilde{x}\in \cO_2(0)$, we have $q(\widetilde{x})=0$, so $\lambda = 0$ and $\widetilde{x} = x' + v$. Put $v = (w,\bi \bar w)$ with $w\in\cC_0$, so $\widetilde{x} = (1+w, -\bi1 + \bi \bar w)$.

We claim that there is an automorphism such that $\varphi(x)\in \FF x$ and $\varphi(\widetilde{x}) \in\FF x'$. If $v=0$ there is nothing to prove, so we can assume $w\neq0$. We consider two cases.

First, consider the case $n(w)=0$. Set $\mu = \frac{1}{2}(1+\bi)$, $a=\mu w$, $\lambda = 1$. Then $\lambda^2 + n(a) = 1$, and hence $\varphi_{a,\lambda}$ is an automorphism. It is not hard to check that $\varphi_{a,\lambda}(x) = (b, \bi b)$ and $\varphi_{a,\lambda}(\widetilde{x}) = (b, -\bi b)$, where $b=1+\frac{1}{2}(1-\bi)w$. Since $n(b)=1$, by Lemma~\ref{lemmaAlberto} we can apply an automorphism given by a related triple that sends $(b,\bi b) \mapsto x=(1,\bi1)$ and $(b,-\bi b) \mapsto x'=(1,-\bi1)$, so we are done with this case.

Second, consider the case $n(w)\neq0$. Take $\lambda,\mu\in\FF$ such that $\lambda^2+\mu^2 n(w)=1$ and $\mu=\frac{1-2\lambda^2}{2\bi\lambda}$. (Replace the expression of $\mu$ of the second equation in the first one, multiply by $\lambda^2$ to remove denominators, take a solution $\lambda$ of this new equation, which exists because $\FF$ is algebraically closed and is nonzero because $n(w)\neq0$. Then take $\mu$ as in the second equation, which is well defined because $\lambda\neq0$.) Moreover, it is clear that $2\lambda^2-1\neq0$, because otherwise we would have $\mu=0$ and the first equation would not be satisfied. Set $a=\mu w$, so we have $\lambda^2 + n(a) = 1$ and therefore $\varphi_{a,\lambda}$ is an automorphism, that sends $x \mapsto (b,\bi b)$, $\widetilde{x} \mapsto (\gamma b,-\bi \gamma b)$, where $b=\lambda1-\bi\mu w$ and $\gamma = (\lambda+\bi\mu n(w))\lambda^{-1}$ (this is easy to check using the two equations satisfied by $\lambda$ and $\mu$). Note that $n(b)=2\lambda^2-1\neq0$, so again we can compose with an automorphism given by a related triple to obtain $\varphi(x)\in \FF x$ and $\varphi(\widetilde{x}) \in\FF x'$.

By the last paragraphs, we can assume that $x=(1,\bi1)$ and $x'=(1,-\bi1)$ are homogeneous elements. Therefore, $\im Q_x = W$, $Q_x(\ker t_x) = V$, $\im Q_{x'} = W'$ and $Q_{x'}(\ker t_{x'}) = V'$ are graded subspaces (where $V$, $V'$, $W$ and $W'$ are defined as above). Note that for each $z\in\cC_0$, $(z,\bi \bar z)\in V$ is homogeneous if and only if $(z, -\bi \bar z)\in V'$ is homogeneous because $Q_x(z, -\bi \bar z) = -2 (z,\bi \bar z)$ and $Q_{x'}(z, \bi \bar z) = -2 (z, -\bi \bar z)$ for any $z\in\cC_0$. On the other hand, if $Z=(z,\bi \bar z)$ is homogeneous for some $z\in \cC$, then $n(z)\neq0$, because otherwise we would have $Z\in\cO_1$ by Lemma~\ref{orbitsbiCayleytriple}, which is not possible.

Take a homogeneous element $x_1=(z_1, \bi\bar z_1) \in V$. Since $n(z_1)\neq0$, scaling $x_1$ we can assume that $n(z_1)=1$. Also, $x'_1 := Q_{x'}(x_1) = (z_1, -\bi\bar z_1) \in V'$ is homogeneous. Since the trace is homogeneous, we can take a homogeneous element $x_2=(z_2, \bi\bar z_2) \in V \cap \ker t_x \cap \ker t_{x'} \cap \ker t_{x_1} \cap \ker t_{x'_1}$. Note that $n(z_1,z_2)=0=n(1,z_2)$, and scaling $x_2$ if necessary, we will assume that $n(z_2)=1$. Then $x'_2 = (z_2, -\bi\bar z_2) \in V'$ is homogeneous. Furthermore, for any homogeneous elements $(y_i, \pm\bi\bar y_i)$, $i=1,2$, we have that $\{(y_1,\bi\bar y_1) , (1,\bi1) , (y_2,-\bi\bar y_2)\}= 2(y_1y_2,\bi\overline{y_1y_2})$ is homogeneous too. Thus, in our case, $(x_1x_2, \pm \bi \overline{x_1x_2})$ are homogeneous. Again, since the trace is homogeneous, we can take homogeneous elements $x_3=(z_3, \bi \bar z_3)$ and $x'_3=(z_3, -\bi \bar z_3)$, with $n(z_3)=1$ and $z_3$ orthogonal to $\lspan\{1, z_1, z_2,z_1z_2 \}$. Notice that $\{z_1, z_2, z_3 \}$ are homogeneous elements generating a $\ZZ_2^3$-grading on $\cC$, and the elements $\{x, x', x_i, x'_i \med i=1,2,3 \}$ generate an isotropic Cayley-Dickson grading on the bi-Cayley triple system. Note that there is only one orbit of isotropic Cayley-Dickson bases (up to constants) on $\cT_\cB$, because the same is true for Cayley-Dickson bases (up to constants) on $\cC$. We can conclude the proof since $\Gamma$ is fine.
\end{proof}

\begin{theorem}
Any fine grading on the bi-Cayley triple system is equivalent to one of the three following nonequivalent gradings: the nonisotropic Cayley-Dickson $\ZZ_2^5$-grading (Example~\ref{ExCDbiCayleytriple}), the isotropic Cayley-Dickson $\ZZ\times\ZZ_2^3$-grading (Example~\ref{ExisotropicCDbiCayleytriple}), or the Cartan $\ZZ^4$-grading (Example~\ref{ExCartanbiCayleytriple}).
\end{theorem}
\begin{proof} This is a consequence of Lemmas~\ref{lemmacartan2}, \ref{lemmacd2} and \ref{lemmacdisotropic}.
\end{proof}

\begin{remark} 
We already know that the isotropic and nonisotropic Cayley-Dickson gradings on the bi-Cayley triple system are not equivalent. However,  the isotropic Cayley-Dickson grading on the bi-Cayley pair (defined in the obvious way) and the (nonisotropic) Cayley-Dickson grading on the bi-Cayley pair are equivalent. This equivalence is given by the restriction of the automorphism in Equation~\eqref{equivalenceCD} to the bi-Cayley pair defined on $\cB=\iota_2(\cC)\oplus\iota_3(\cC)$.
\end{remark}

\subsection{Classification of fine gradings on the Albert triple system}

\begin{theorem} 
There are four gradings, up to equivalence, on the Albert triple system. Their universal groups are: $\ZZ^4\times\ZZ_2$, $\ZZ_2^6$, $\ZZ \times\ZZ_2^4$ and $\ZZ_3^3\times\ZZ_2$.
\end{theorem}
\begin{proof} 
This is a consequence of Corollary~\ref{triplegrads}, Proposition~\ref{regulargradstriples} and the classification of fine gradings on the Albert algebra (\cite{EK12a}). 
\end{proof}

\section{Induced gradings on Lie algebras $\mathfrak{e}_6$ and $\mathfrak{e}_7$}

It is well-known that $\TKK(\cV_\cB)=\mathfrak{e}_6$ and $\TKK(\cV_\alb)=\mathfrak{e}_7$. Recall that $\dim \mathfrak{e}_6 = 78$ and $\dim \mathfrak{e}_7 = 133$. We will study now the gradings induced by the TKK construction from the fine gradings on $\cV_\cB$ and $\cV_\alb$. Note that the classification of fine gradings, up to equivalence, on all finite-dimensional simple Lie algebras over an algebraically closed field of characteristic $0$ is complete (\cite[Chapters 3-6]{EKmon}, \cite{Eld14}, \cite{Yu14}).
A classification of the fine gradings on $\mathfrak{e}_6$, for the case $\FF = \CC$, can be found in \cite{DV12}.

Recall that we always assume, unless otherwise stated, that the base field $\FF$ is algebraically closed with $\chr\FF\neq2$.

Recall that, if $\Gamma$ is a grading on a finite-dimensional algebra $A$, a sequence of natural numbers $(n_1, n_2, \dots)$ is called the \emph{type} of the grading $\Gamma$ if there are exactly $n_i$ homogeneous components of dimension $i$, for $i\in\NN$. Note that $\dim A = \sum_i i \cdot n_i$.

\begin{proposition}
The Cartan $\ZZ^6$-grading on the bi-Cayley pair extends to a fine grading with universal group $\ZZ^6$ and type $(72,0,0,0,0,1)$ on $\mathfrak{e}_6$, that is, a Cartan grading on $\mathfrak{e}_6$. Similarly, the Cartan $\ZZ^7$-grading on the Albert pair extends to a fine grading with universal group $\ZZ^7$ and type $(126,0,0,0,0,0,1)$ on $\mathfrak{e}_7$, that is, a Cartan grading on $\mathfrak{e}_7$.
\end{proposition}
\begin{proof}
This is a consequence of Theorem~\ref{correspondence} and the fact that the only gradings up to equivalence with these universal groups on the Lie algebras are the Cartan gradings. (Recall that Cartan gradings on simple Lie algebras are induced by maximal tori. By \cite[Section 21.3]{H75}, the maximal tori of $\Aut(\mathfrak{e}_6)$ are conjugate, so their associated $\ZZ^6$-gradings on $\mathfrak{e}_6$ must be equivalent. The same holds for the $\ZZ^7$-gradings on $\mathfrak{e}_7$.)
\end{proof}

\begin{proposition} \label{inducedgradbiCayley}
The Cayley-Dickson $\ZZ^2\times\ZZ_2^3$-grading on the bi-Cayley pair extends to a fine grading with universal group $\ZZ^2\times\ZZ_2^3$ and type $(48,1,0,7)$ on $\mathfrak{e}_6$.
\end{proposition}
\begin{proof}
This is a consequence of Theorem~\ref{correspondence}, except for the type, which we will now compute. Set $e=(0,0,\bar0,\bar0,\bar0)$ and write $L = \mathfrak{e}_6$, $\cV = \cV_\cB$. If $\nu(x,y)\in L_e^0$, it must be $\deg(x^+) + \deg(y^-) = e$ and hence $\FF x = \FF y$. For elements in the Cayley-Dickson basis of $\cV_\cB$ we have
$\{(x_i,0),(x_i,0),\cdot\} = m_{2,1}$ and $\{(0,x_i),(0,x_i),\cdot\} = m_{1,2}$, where $m_{\lambda,\mu} : \cB \to \cB$, $(a,b) \mapsto (\lambda a,\mu b)$. It follows that $L_e^0$ is spanned by $(m_{2,1},-m_{2,1})$ and $(m_{1,2},-m_{1,2})$. In particular, dim$L^0_e=2$.

Take $g=(0,0,t)\in G=\ZZ^2\times\ZZ_2^3$ with $0\neq t\in\ZZ_2^3$. Given a homogeneous element $x\in L^1 = \cV^+$ in the Cayley-Dickson basis of $\cV$, there is a unique $y\in L^{-1} = \cV^-$ in the Cayley-Dickson basis such that $\nu(x,y) = [x,y]\in L_g^0$, i.e., $\deg(x^+) + \deg(y^-) = g$, and in that case we always have $x\neq y$. Take different elements $x_i$, $x_j$ in the Cayley-Dickson basis of $\cC$ such that $\deg((x_i,0)^+) + \deg((x_j,0)^-) = g$. There are four such pairs $\{i,j\}$. Then, we have four linearly independent elements of $L_g^0$ such that their first components are given by:
\begin{align*}
\{(x_i,0),(x_j,0),\cdot\}=-\{(x_j,0),(x_i,0),\cdot\}=\left\{\begin{array}{l}
(x_j,0)\mapsto2(x_i,0) \\ (x_i,0)\mapsto-2(x_j,0) \\ (0,x_k)\mapsto(0,-(x_kx_i)\bar x_j) \ \text{for any $k$} \\ 0 \text{ otherwise}.
\end{array}\right.
\end{align*} 
It follows that dim$L_g^0 \geq 4$, and there are seven homogeneous components of this type, one for each choice of $t$.

Take now $g=(1,-1,t)$ with $t\in\ZZ_2^3$ (the case $g=(-1,1,t)$ is similar). Take elements $(x_i,0)$ and $(0,x_j)$ in the Cayley-Dickson basis such that $\deg((x_i,0)^+) + \deg((x_j,0)^-) = g$. Note that, for elements in the Cayley-Dickson basis we have
\begin{align*} 
\{(x_i,0),(0,x_j),\cdot\}=\left\{\begin{array}{l}
(0,x_k)\mapsto(\bar x_k(x_jx_i),0) \\
(x_k,0)\mapsto(0,0)
\end{array}\right.,
\end{align*}
which is a nonzero map. Hence $L_g^0 \neq 0$, and therefore $\dim L_g^0 \geq 1$. Note that there are 8 homogeneous components with degrees $g=(1,-1,t)$ for $t\in\ZZ_2^3$, and 8 more with degrees $g=(-1,1,t)$ for $t\in\ZZ_2^3$.

Finally, the subspace $L^1\oplus L^{-1}=V^+\oplus V^-$ consists of other $32$ homogeneous components of dimension $1$. The sum of the subspaces already considered has dimension at least $2 + 4 \cdot 7 + 16 + 32 = 78 = \dim L$. Therefore, the inequalities above are actually equalities and the type of the grading is $(48,1,0,7)$.
\end{proof}

\begin{remark}
For the TKK construction $L=\TKK(\cV_\cB)$ of $\mathfrak{e}_6$, we have that $L^0 = \text{Der}(\cV_\cB) \cong \mathfrak{d}_5 \oplus Z$ where $Z$ is a $1$-dimensional center (see the proof of Theorem~\ref{ThSchemesBicayley}). The $\ZZ^6$-grading on $L$ restricts to a $\ZZ^5$-grading of type $(40,0,0,0,0,1)$ on $L^0$, which restricts to the Cartan $\ZZ^5$-grading on $\mathfrak{d}_5$. On the other hand, the $\ZZ^2\times\ZZ_2^3$-grading on $L$ restricts to a $\ZZ\times\ZZ_2^3$-grading of type $(16,1,0,7)$ on $L^0$.
\end{remark}

\begin{proposition} 
The Cayley-Dickson $\ZZ^3\times\ZZ_2^3$-grading on the Albert pair extends to a fine grading with universal group $\ZZ^3\times\ZZ_2^3$ and type $(102,0,1,7)$ on $\mathfrak{e}_7$.
\end{proposition}
\begin{proof}
This is a consequence of Theorem~\ref{correspondence}, except for the type, which we will now compute. Notice that $L^1 \oplus L^{-1} = \alb^+ \oplus \alb^-$ consists of $54$ homogeneous components of dimension 1.

Set $e = (0,0,0,\bar0,\bar0,\bar0) \in \ZZ^3\times\ZZ_2^3$. Note that $\{E_i,E_i,\cdot\}$ acts multiplying by $2$ on $E_i$, and multiplying by $0$ on $E_{i+1}$ and $E_{i+2}$. Therefore, $\dim L^0_e \geq 3$. 

Recall from the proof of Proposition~\ref{inducedgradbiCayley} that the $\ZZ^2 \times \ZZ_2^3$-grading on $\mathfrak{e}_6$ (induced from $\cV_\cB$) has $16$ components of dimension $1$ with associated degrees $(\pm1,\mp1,g)$ with $g\in\ZZ_2^3$. Therefore, by symmetry for our grading on $\mathfrak{e}_7$, there must be $16 \times 3 = 48$ homogeneous components of at least dimension 1 (the dimension may increase on $\mathfrak{e}_7$), with associated degrees $(\pm1,\mp1,0,g)$, $(\pm1,0,\mp1,g)$, $(0,\pm1,\mp1,g)$, where $g\in\ZZ_2^3$. These components span a subspace of dimension at least $48$. 

Recall also that the $\ZZ^2 \times \ZZ_2^3$-grading on $\mathfrak{e}_6$ has $7$ components of dimension $4$ and degrees $(0,0,g)$ with $e\neq g\in\ZZ_2^3$, so the $\ZZ^3 \times \ZZ_2^3$-grading on $\mathfrak{e}_7$ has at least $7$ components, with degrees $(0,0,0,g)$ with $e\ne g\in \ZZ_2^3$, of dimension at least $4$, whose sum spans a subspace of dimension at least $28$.

Finally, note that de sum of the previous subspaces has dimension at least $54 + 3  + 48 + 28 = 133 = \text{dim} \mathfrak{e}_7$. Hence, the inequalities in the dimensions above are equalities, and the result follows.
\end{proof}

\begin{proposition}
Assume that $\chr\FF\neq3$. The $\ZZ \times \ZZ_3^3$-grading on the Albert pair extends to a fine grading with universal group $\ZZ \times \ZZ_3^3$ and type $(55,0,26)$ on $\mathfrak{e}_7$.
\end{proposition}
\begin{proof} 
This is a consequence of Theorem~\ref{correspondence}, except for the type, which we will now compute. We know that our grading satisfies that, if $\alb^\sigma_g = \FF x$ for some $0\neq x \in \alb$, then $\alb^{-\sigma}_{-g} = \FF x^{-1}$. Hence $L_e^0$ is spanned by elements of the form $\nu(x,x^{-1})$. But it is well-known that, if an element $x$ is invertible in a Jordan algebra, then $\{x,x^{-1},\cdot\} = 2 \id$. Therefore, $L_e^0 = \FF(\id, -\id)$ has dimension $1$. (Actually, $L_e^0$ is the center of $L^0$.) Moreover, the subspace $L^1 \oplus L^{-1} = \alb^+ \oplus \alb^-$ consists of $54$ homogeneous components of dimension 1. 

The rest of homogeneous components span a subspace of dimension $133 - 55 = 78$ (actually, a subalgebra isomorphic to $\mathfrak{e}_6$) and support $\{(0,g) \med 0\neq g \in \ZZ_3^3 \}$, and since its homogeneous components are clearly in the same orbit under the action of $\Aut\Gamma$ (see Theorem~\ref{weylZxZ33} and its proof for more details), each of them must have dimension $78/26=3$.
\end{proof}

\section{Weyl groups}

Now we will compute the Weyl groups of the fine gradings on the Jordan pairs and triple systems of types bi-Cayley and Albert.

As a consequence of Corollary~\ref{tripleWeyl} and the classification of the Weyl groups of fine gradings on $\alb$ (see \cite{EKmon}), we already know the Weyl groups of the fine gradings on the Albert triple system.

\begin{theorem} \label{weylnonisotropicCD}
Let $\Gamma$ be either the Cayley-Dickson $\ZZ^2 \times \ZZ_2^3$-grading on the bi-Cayley pair, or the nonisotropic Cayley-Dickson $\ZZ_2^5$-grading on the bi-Cayley triple system. Then, 
\begin{equation*} \begin{aligned}
& \cW(\Gamma) \cong \left\{ \left(\begin{array}{c|c} A & 0 \\ \hline B & C \end{array}\right) \in \GL_5(\ZZ_2) \med 
  A\in\langle\tau\rangle, B\in\mathcal{M}_{3\times2}(\ZZ_2), C\in\GL_3(\ZZ_2) \right\},
\end{aligned}
\end{equation*}
where $\tau=\left(\begin{array}{cc}0&1 \\ 1&0 \end{array}\right) \in \GL_2(\ZZ_2)$.
\end{theorem}
\begin{proof} 
We will prove now the first case. Identify $\ZZ^2$ and $\ZZ_2^3$ with the subgroups $\ZZ^2 \times 0$ and $0 \times \ZZ_2^3$ of $G=\ZZ^2 \times \ZZ_2^3$. Let $\{ a, b \}$ and $\{ a_i \}_{i=1}^3$ denote the canonical bases of the subgroups $\ZZ^2$ and $\ZZ_2^3$. Let $\Gamma_\cC$ be the $\ZZ_2^3$-grading on $\cC$. It is well-known (see \cite{EKmon}) that $\cW(\Gamma_\cC) \cong \Aut(\ZZ_2^3) \cong \GL_3(\ZZ_2)$. If $f\in\Aut\Gamma_\cC$, then $f\times f\in\Aut\Gamma$ (notation as in the proof of Theorem \ref{generatorsAutVB}), and with an abuse of notation we have $\cW(\Gamma_\cC)\leq\cW(\Gamma)\leq\Aut(G)$. 

Since $\bar\tau_{12}$ induces the element $\tau$ of $\cW(\Gamma)$ of order 2, given by $a\leftrightarrow b$, that commutes with $\cW(\Gamma_\cC)$, we have $\langle\tau\rangle \times \GL_3(\ZZ_2)\leq\cW(\Gamma)$. Furthermore, from Lemma~\ref{lemmaAlberto} we can deduce that there is a related triple $\varphi$ that induces an element $\bar\varphi$ of $\cW(\Gamma)$ of the form $a\mapsto a+c_1$, $b\mapsto b+c_2$ with $c_i\in\ZZ_2^3$, $c_1\neq0$. Without loss of generality, composing $\varphi$ with some element of $\cW(\Gamma_\cC)$ if necessary, we can also assume that $\bar\varphi$ fixes $a_i$ for $i=1,2,3$. It is clear that $\bar\varphi$ and $\langle\tau\rangle \times \GL_3(\ZZ_2)$ generate a subgroup $\cW$ of $\cW(\Gamma)$ isomorphic to the one stated in the result. It remains to show that $\cW(\Gamma) \leq \cW$.

Take $\phi\in\cW(\Gamma)$; we claim that $\phi\in\cW$. By $\phi$-invariance of $\supp\Gamma$, either $\phi(a)=a+c$, or $\phi(a)=b+c$, for some $c\in\ZZ_2^3$, so if we compose $\phi$ with elements of $\cW$ we can assume that $\phi(a)=a$. Since the torsion subgroup $\ZZ_2^3$ is $\phi$-invariant, if we compose with elements of $\cW(\Gamma_\cC)$ we can also assume that $\phi(a_i)=a_i$ ($i=1,2,3$). Finally, by $\phi$-invariance of $\supp\Gamma$ and $\ZZ_2^3$, it must be $\phi(b)=b+c$ for some $c\in\ZZ_2^3$, and composing again with elements of $\cW$ we can assume in addition that $\phi(b)=b$. Hence $\phi=1$ and $\cW(\Gamma) = \cW$.

Finally, let $\Gamma'$ denote the $\ZZ_2^5$ grading on the bi-Cayley triple system. Note that $\Gamma'$ induces a coarsening $(\Gamma', \Gamma')$ of $\Gamma$ on $\cV_\cB$, where we can identify the set $\supp\Gamma'$ with $\supp\Gamma^+$. Since $\cW(\Gamma)$ is determined by the action of $\Aut(\Gamma)$ on $\supp\Gamma^+ \equiv \supp\Gamma'$, it follows that we can identify $\Aut(\Gamma')$ with a subgroup of $\Aut(\Gamma)$ and hence $\cW(\Gamma') \leq \cW(\Gamma)$. Recall that $\bar\tau_{12}$, $\Aut(\Gamma_\cC)$ and $\varphi$ induce the generators of $\cW(\Gamma)$, and on the other hand these are given by elements of $\Aut(\Gamma')$, so we also have $\cW(\Gamma) \leq \cW(\Gamma')$ with the previous identification, and therefore $\cW(\Gamma) = \cW(\Gamma')$.
\end{proof}

\begin{theorem}
Let $\Gamma$ be the isotropic Cayley-Dickson $\ZZ \times \ZZ_2^3$-grading on the bi-Cayley triple system. Then $\cW(\Gamma)$ is the whole $\Aut(\ZZ\times\ZZ_2^3)$.
\end{theorem}
With the natural identification we can express this result as follows:
\begin{equation*}
\cW(\Gamma) \cong \left\{\left(\begin{array}{c|c} A & 0 \\ \hline B & C \end{array}\right) \med A\in\{\pm1\}, B\in\mathcal{M}_{3\times1}(\ZZ_2), C\in\GL_3(\ZZ_2) \right\}.
\end{equation*}
\begin{proof}
The proof is similar to the proof of Theorem~\ref{weylnonisotropicCD}, so we do not give all the details. The block with $\GL_3(\ZZ_2)$ is induced by automorphisms of $\cC$ extended to $\cT_\cB$. The blocks with $0$ and $\{\pm1\}$ are obtained by $\cW(\Gamma)$-invariance of the torsion subgroup and the support of $\Gamma$. (Both automorphisms $\bar\tau_{12}$ and $c_{1,-1}$ induce the element that generates the block $\{\pm1\}$ in $\cW(\Gamma)$.)
Take a homogeneous $a\in\cC$ with nonzero degree in the associated Cayley-Dickson grading on $\cC$. Hence, $\text{tr}(a)=0$, and scaling we can assume that $n(a)=1$. Consider the automorphism $\varphi=\Phi(a)$ of $\cT_\cB$, with $\Phi$ as in Proposition~\ref{isomphi}. It is checked that $\varphi(x_i,\pm\bi\bar x_i) = \pm\bi(x_ia,\pm\bi\overline{x_ia})$, and therefore $\varphi$ belongs to $\Aut(\Gamma)$ and induces a nonzero element of the block $\mathcal{M}_{3\times1}(\ZZ_2)$. We conclude that all the block $\mathcal{M}_{3\times1}(\ZZ_2)$ appears, which concludes the proof.
\end{proof}

\begin{remark}
The Weyl group in the result above is isomorphic to the Weyl group of the $\ZZ\times\ZZ_2^3$-grading on the Albert algebra.
\end{remark}

\begin{theorem} \label{weylZxZ33}
Let $\Gamma$ be the fine grading on $\cV_\alb$ with universal group $\ZZ \times \ZZ_3^3 $. Then, with the natural identification of $\Aut(\ZZ\times\ZZ_3^3)$ with a group of $4\times 4$-matrices,
\begin{equation*}
\cW(\Gamma) \cong \left\{\left(\begin{array}{c|c}  1 & 0  \\ \hline A & B \end{array}\right) \med A\in\mathcal{M}_{3\times1}(\ZZ_3), B\in\SL_3(\ZZ_3) \right\}.
\end{equation*}
\end{theorem}
\begin{proof} 
Set $G=\ZZ \times \ZZ_3^3 $ and identify the subgroups $\ZZ$ and $\ZZ_3^3$ with $\ZZ \times 0$ and $0 \times \ZZ_3^3$. Let $\Gamma_\alb$ be the $\ZZ_3^3$-grading on $\alb$ as given in Equation~\eqref{Z33grading}, and $\deg_\alb$ its degree map. Let $a_1,a_2,a_3$ denote the canonical generators of $\ZZ_3^3$ (hence $\deg_\alb(X_i) = a_i$), and write $a$ for the generator $1$ of $\ZZ$. Therefore, for each homogeneous element $x\in\alb$, we have $\deg(x^\pm)=(\pm1, \deg_\alb(x))$ in $\Gamma$. By \cite{EK12b}, $\cW(\Gamma_\alb)\cong \SL_3(\ZZ_3)$. With the identification $\Aut\alb\leq\Aut\cV_\alb$, we have $\Aut\Gamma_\alb\leq\Aut\Gamma$, and we can also identify $\cW(\Gamma_\alb)$ with a subgroup of $\cW(\Gamma)$ which acts on $\ZZ_3^3$ and fixes the generator $a$ of $\ZZ $. Since $\supp\Gamma^+$ and the torsion subgroup $\ZZ_3^3$ are $\cW(\Gamma)$-invariant, we deduce that $\SL_3(\ZZ_3)$, $0$ and $1$ appear in the block structure of $\cW(\Gamma)$, as it is described above. We claim now that $\mathcal{M}_{3\times1}(\ZZ_3)$ appears in the block structure, and for this purpose, it suffices to find an element of $\cW(\Gamma)$ given by $a\mapsto a+a_3$ and that fixes each $a_i$. Take $\varphi=c_{1,\omega^2,\omega}$ with $\omega$ a primitive cubic root of 1, and the induced automorphism $\tau$ in $\cW(\Gamma)$ (notation $c_{\lambda_1,\lambda_2,\lambda_3}$ as in Subsection~\ref{automorphismssubsection}). It is clear that $\tau(a_1)=a_1$ and $\tau(a_2)=a_2$. Since $\varphi(1)=\sum_{i=1}^3\omega^{2i}E_i$, we have $\tau(a)=a+a_3$. Also, $\varphi(\sum_{i=1}^3\omega^{2i}E_i)=\sum_{i=1}^3\omega^iE_i$, from where we get $\tau(a+a_3)=a+2a_3$, and so $\tau(a_3)=a_3$. We conclude that $\tau$ is the element of $\cW(\Gamma)$ that we were searching, and hence a subgroup $\cW$ of $\cW(\Gamma)$ as in the statement appears. It remains to prove that $\cW(\Gamma)\leq\cW$.

Take $\phi\in\cW(\Gamma)$; we claim that $\phi\in\cW$. Without loss of generality for our purpose, if we compose with elements of $\cW$ we can assume that $\phi(a)=a$. It suffices to show that $\phi$ acts on $\ZZ_3^3$ as an element of $\SL_3(\ZZ_3)$. We know by \cite{EK12b} that there are two equivalent but nonisomorphic $G$-gradings on $\alb$, that we denote by $\Gamma^+ = \Gamma_\alb$ and $\Gamma^-$. (The nonisomorphy is due to the existence of homogeneous elements $X_i$ in $\Gamma^+$ and $X'_i$ in $\Gamma^-$, with $X_i$ and $X'_i$ of the same degree, and such that $(X_1X_2)X_3=\omega X_1(X_2X_3)$ and $(X'_1X'_2)X'_3=\omega^2 X'_1(X'_2X'_3)$.) Notice that the product in the algebra is determined by the triple product and the elements $1^\pm$ (because $\{x,1,y\}=xy$), so it follows that $\Gamma^+$ and $\Gamma^-$ are neither isomorphic when they are considered as $\ZZ_3^3$-gradings on $\cV_\alb$. Thus, the whole $\GL_3(\ZZ_3)$ cannot appear in the block structure. Since $\SL_3(\ZZ_3)$ has index $2$ in $\GL_3(\ZZ_3)$, we deduce that $\SL_3(\ZZ_3)$ is exactly what appears in the block structure of $\cW(\Gamma)$. We conclude that $\phi$ acts on $\ZZ_3^3$ as an element of $\SL_3(\ZZ_3)$, and so $\phi\in\cW$.
\end{proof}

\begin{theorem} 
Let $\Gamma$ be the Cayley-Dickson $\ZZ^3 \times \ZZ_2^3$-grading on $\cV_\alb$. Then, 
\begin{equation*}
\cW(\Gamma) \cong \left\{\left(\begin{array}{c|c} A & 0 \\ \hline B & C \end{array}\right) \med A\in\Sym(3)=\langle\tau,\sigma\rangle, B\in\mathcal{M}_3(\ZZ_2), C\in\GL_3(\ZZ_2) \right\},
\end{equation*}
with
\begin{equation*}
\tau=\left(\begin{array}{ccc} 0&1&0\\1&0&0\\0&0&1\end{array}\right), \quad
\sigma=\left(\begin{array}{ccc} 0&0&1\\1&0&0\\0&1&0\end{array}\right) \in \GL_3(\ZZ_2).
\end{equation*}
\end{theorem}
\begin{proof}
Set $G=\ZZ^3 \times \ZZ_2^3$ and identify $\ZZ^3$ and $\ZZ_2^3$ with the subgroups $\ZZ^3\times0$ and $0\times\ZZ_2^3$. Let $a_1,a_2,a_3$ be the canonical generators of $\ZZ_2^3$, and $b_i=\deg^+(E_i)$, for $i=1,2,3$, the generators of $\ZZ^3$. Let $\Gamma_\cC$ be the $\ZZ_2^3$-grading on $\cC$. It is well-known that $\cW(\Gamma_\cC)\cong\Aut(\ZZ_2^3)\cong \GL_3(\ZZ_2)$, and the automorphisms of $\cC$ are extended to related triples, which are also extended to $\Aut(\cV_\alb)$, and hence $\GL_3(\ZZ_2)$ appears in the block structure. Since the torsion subgroup $\ZZ_2^3$ is $\cW(\Gamma)$-invariant, the zero block must appear. The homogeneous components consisting of elements of rank 1 are exactly the ones of the idempotents $E_i$, and therefore the set $\{b_1,b_2,b_3\}$ is $\cW(\Gamma)$-invariant. This implies that the $(1,1)$-block is, up to isomorphism, a subgroup of $\Sym(3)$; since there are elements of $\Aut\Gamma$ that permute the idempotents $E_i$, the group $\Sym(3)$ must be what appears in the block. On the other hand, for the Cayley-Dickson grading $\Gamma'$ of $\cV_\cB$, we know that there are related triples in $\Aut\Gamma'$ that do not fix the subgroup $\ZZ^2$ of the universal group, and these are obtained as restriction of elements of $\Aut\Gamma$ that do not fix $\ZZ^3$, so it follows that $\mathcal{M}_3(\ZZ_2)$ must appear in the block structure. This concludes the proof.
\end{proof}

\bigskip

The Weyl groups of the Cartan gradings on $\cT_\cB$, $\cV_\cB$ and $\cV_\alb$ are given as follows:

\begin{theorem} \label{theoremWeylCartanGradings}
The Weyl groups of the Cartan gradings on the bi-Cayley pair and the Albert pair are isomorphic to the Weyl groups of the root systems of type $D_5$ and $E_6$, respectively. 

The Weyl group of the Cartan grading on the bi-Cayley triple system is isomorphic to $\ZZ_2^4 \rtimes \Sym(4)$, i.e., the automorphism group of the root system of type $B_4$ (or $C_4$).
\end{theorem}
\begin{proof}
The result follows by reducing the computation of the Weyl group to the corresponding problem of the Lie algebras obtained by the TKK process, and will be omitted.
\end{proof}

\section*{Acknowledgments}

I would like to give thanks to professor Alberto Elduque for his supervision, numerous corrections, and valuable suggestions, which have considerably improved this paper, written as a part of my PhD thesis. I am also very thankful to the referee for his many corrections and improvements of this paper.


\end{document}